\documentclass[EJP]{ejpecp} 

\SHORTTITLE{SI Epidemics on Evolving Graphs}

\TITLE{Susceptible--Infected Epidemics on Evolving Graphs} 

\AUTHORS{%
	Rick~Durrett\footnote{Duke University, United States of America.
		\EMAIL{rtd@math.duke.edu}}
	\and 
	Dong~Yao\footnote{Corresponding author. Jiangsu Normal  University, China.  \EMAIL{wonderspiritfall@gmail.com}
	}
} 


\AMSSUBJ{60J27} 

\SUBMITTED{March 17, 2020} 
\ACCEPTED{July 9, 2022} 

\VOLUME{0}
\YEAR{2022}
\PAPERNUM{110}
\DOI{10.1214/22-EJP828}

\ABSTRACT{
	The evoSIR model is a modification of the usual SIR process on a graph $G$ in which $S-I$ connections are broken at rate $\rho$ and the $S$ connects to a randomly chosen vertex. The evoSI model is the same as evoSIR but recovery is impossible. In \cite{DOMath} the critical value for evoSIR was computed and simulations showed that when $G$ is an Erd\H os-R\'enyi graph with mean degree 5, the system has a discontinuous phase transition, i.e., as the infection rate $\lambda$ decreases to $\lambda_c$, the fraction of individuals infected during the epidemic does not converge to 0. In this paper we study evoSI dynamics on graphs generated by the configuration model. We show that there is a quantity $\Delta$ determined by the first three moments of the degree distribution, so that the phase transition is discontinuous if $\Delta>0$ and continuous if $\Delta<0$.  }

\usepackage{latexsym,amssymb,amsmath,amsfonts,amsthm,graphicx,url}

\usepackage{tikz}
\usepackage{hyperref}
\usepackage{caption}
\usepackage{subcaption}

\def\sqz{\kern -0.2em}

\newtheorem{prop}{Proposition}
\theoremstyle{definition}

\newtheorem{model}{Model}
\newcommand{\abs}[1]{\left\vert#1\right\vert}
\renewcommand{\hat}{\widehat}
\renewcommand{\tilde}{\widetilde}
\newcommand{\T}{T\wedge \gamma_n}
\newcommand{\N}{\mathbb{N}}
\def\beq{ \begin{equation} }
	\def\eeq{ \end{equation} }
\def\mn{\medskip\noindent}

\def\EE{\hat{\E}}
\def\S{\overline{S}}
\def\X{\overline{X}}

\def\I{\overline{I}}
\def\M{\overline{M}}
\def\s{\overline{s}}
\def\x{\overline{x}}

\def\z{z\wedge \gamma_n}
\def\t{t\wedge \gamma_n}

\def\square{\vcenter{\vbox{\hrule height .4pt
			\hbox{\vrule width .4pt height 5pt \kern 5pt
				\vrule width .4pt} \hrule height .4pt}}}

\def\CP{\xrightarrow{\P}}
\def\ep{\epsilon}

\def\D{\mathbb{D}}
\def\BD{\mathbf{D_n}}

\def\var{\hbox{Var}\,}

\def\ER{Erd\H{o}s-R\'enyi}

\def\E{\mathbb{E}}
\def\P{\mathbb{P}}
\def\CM{\mathbb{CM}}

\definecolor{darkblue}{rgb}{0,0,0.6}

\def\clearp{\clearpage}
\def\clearp{}

\numberwithin{equation}{subsection}
\numberwithin{theorem}{section}

\begin{document}
	
	
	
	\section{Introduction}

	In the SIR model, individuals are in one of  three states: $S=$ susceptible, $I=$ infected, $R=$ removed (cannot be infected). Often this epidemic takes place in a homogeneously mixing population. However, here, we have a graph $G$ that gives the social structure of the population; vertices represent individuals and edges connections between them. $S-I$ edges become $I-I$ at rate $\lambda$, i.e., after a time with an exponential($\lambda$) distribution.  An individual remains infected for an amount of time $T$ which can be either deterministic or random with a pre-specified distribution. Once individuals leave the infected state, they enter the removed state. In addition, we will allow the graph to evolve: $S-I$ edges are broken at rate $\rho$ and the susceptible individual connects to an individual chosen uniformly at random from the graph. This process is called evoSIR where `evo' stands for `evolving'. We will also consider the simpler $SI$ epidemic in which infecteds never recover, and its evolving version $evoSI$, as well as some other variations on this theme. 
	
	We will use the term \emph{final epidemic size} or \emph{final size of the epidemic} to refer to the number of vertices that are eventually removed in SIR epidemics or eventually infected in SI epidemics.  We say   a large epidemic (or large outbreak) occurs if the epidemic infects more than $\ep n$ individuals ($n$ is the size of the total population) for some $\ep>0$ independent of $n$. 
	The critical value is the smallest infection rate such that a large outbreak occurs with probability bounded away from 0 as $n\to\infty$. 
	
	As the main contribution of this paper, we show that, if the underlying graph $G$ is sampled from the configuration model $\CM(n,D)$ (defined in Model \ref{config} in Section \ref{britton}), then there is an explicit quantity $\Delta$ given in \eqref{deldef} (determined by the first three moments of $D$) such that the following holds for the evoSI model on $G$:
	\begin{itemize}
		\item If $\Delta>0$, then the fraction of infected vertices (conditionally on a large outbreak) doesn't converge to 0 as $\lambda$ approaches the critical value $\lambda_c$. In other words, there is a discontinuous phase transition.
		\item If $\Delta<0$, then the phase transition is continuous for $\lambda$ near $\lambda_c$.
	\end{itemize}
	See Theorem \ref{Q1new} for a precise statement. 
	
	The introduction is organized into seven subsections. Sections \ref{sec;DOM}--\ref{sec:BB} are devoted to a review of  previous work. The main results of the paper (Theorem
	\ref{critical_value} and Theorem \ref{Q1new}) are  presented in Section \ref{sec:state}. Sections \ref{pfth15}--\ref{pfth18} sketch the proof of Theorem \ref{Q1new}. 
	
	\subsection{DOMath \cite{DOMath}}\label{sec;DOM}
	evoSIR was stiudied by three Duke undergraduates in the summer and fall of 2018 (Yufeng Jiang, Remy Kassem, and Grayson York) under the direction of  Matthew Junge and Rick Durrett. They considered two possibilities for the infection time: the Markovian case in which infections last for an exponential time with mean 1, and the case in which each infection lasts for exactly time 1. Here we restrict out attention to their results for the second case, which is simpler due to its connection with independent bond percolation. In any SIR model each edge will be $S-I$ (or $I-S$) only once. When that happens, in the fixed infection time case without rewiring, the infection will be transferred 
	to the other end with probability 
	\beq
	\tau^f = \P(T \le 1 ) = 1-e^{-\lambda},
	\label{tauf}
	\eeq
	and the transfers for different edges are independent. Here the `$f$' in the superscript is for ``fixed time." Due to the last observation, we can delete edges with probability $e^{-\lambda}$ and the connected components of the resulting graph will give the epidemic sizes when one member of the cluster is infected.  
	
	In \cite{DOMath} $G$ was an \ER($n,\mu/n$) random graph in which there are $n$ vertices and each pair is independently connected with probability $\mu/n$. The following result is well-known.
	
	\begin{theorem}\label{thin}
		Consider the SIR process
		on \ER($n,\mu/n$) with fixed infection time. The reduced graph after deletion of edges as described above is \ER$(n,\mu\tau^f/n)$. So, if we start with one infected and the rest of the population susceptible, a large outbreak occurs with positive probability if and only if $\mu\tau^f > 1$. If $z_0$ is the fixed point smaller than $1$ of the generating function 
		\beq
		G_0(z) = \exp(-\mu\tau^f (1-z)), 
		\label{ftgf}
		\eeq
		then $1-z_0$ gives both the limiting probability that an infected individual will start a large epidemic, and the fraction of individuals who will become infected when a large epidemic occurs.
	\end{theorem}

	Things become more complicated when we introduce rewiring of $S-I$ edges at rate $\rho$. To be able to use the ideas in the proof of Theorem \ref{thin}, \cite{DOMath} introduced the delSIR model in which edges  are deleted instead of rewired. In the evo (or del) version of the model, in order for the infection to be transmitted along an edge, infection must come  before any rewiring (or deletion) and before time 1. To compute this probability, note that (i) the probability that infection occurs before rewiring is $\lambda/(\lambda+\rho)$ and (ii) the minimum of two independent exponentials with rates $\lambda$ and $\rho$ is an exponential with rate $\lambda+\rho$, so the transmission probability is
	\beq
	\tau^f_r = \frac{\lambda}{\lambda+\rho}(1-e^{-(\lambda+\rho)}).
	\label{taufr}
	\eeq
	Here the `$r$' subscript is for ``rewire."  By a standard coupling argument one can show that evoSI dominates delSI.
	
	\begin{lemma} \label{couple}
		For fixed parameters, there exists a coupling of evoSI and delSI so that there are no fewer infections in the delSI model than in evoSI. Same is true if we replace SI by SIR.  
	\end{lemma}
	
	The next result, Theorem 1 in \cite{DOMath}, shows that  evoSIR has the same critical value as delSIR, and in the subcritical case the expected cluster sizes are the same.

	\begin{theorem}\label{ftcrit}
		The critical value $\lambda_c$ for a large epidemic in fixed infection time delSIR or evoSIR  epidemic with rewiring is given by the solution of $\mu\tau^f_r(\lambda)=1$. 
		Moreover, if $\lambda < \lambda_c$, then the ratio of the expected epidemic size in delSIR to the size in evoSIR converges to 1 as the number of vertices goes to $\infty$.
	\end{theorem}
	
	The formula for the critical value is easily seen to be correct for the delSIR since, by the reasoning above, there is a large epidemic if and only if the reduced graph in which edges are retained with probability $\tau^f_r$ has a giant component. From Lemma \ref{couple} one can see that the delSIR model has a larger ($\ge$) critical value than evoSIR. Thus, one only has to prove the reverse inequality. Intuitively, the equality of the two critical values holds because a subcritical delSIR epidemic dies out quickly, so it is unlikely that rewirings will influence the outcome.

	When $n$ is large, the degree distribution, which is Binomial($n-1,\mu/n$), is approximately Poisson with mean $\mu$. Due to Poisson thinning,  the number of new infections directly caused by  one $I$ in delSIR in an otherwise susceptible population is asymptotically Poisson with mean $\mu\tau^f_r$, and hence has limiting generating function of the distribution  Poisson($\mu\tau^f$)
	\beq
	G_1(z) = \exp(-\mu\tau^f_r (1-z)).
	\label{hatG}
	\eeq
	The following result, Theorem 2 in \cite{DOMath}, identifies the probability of a large outbreak.
	
	\begin{theorem} \label{ftperc}
		If $z_0<1$ is the fixed point of $G_1(z)$, then $1-z_0$ gives the probability of a large delSIR or evoSIR outbreak.
	\end{theorem}
	
	\noindent
	In the case of the delSIR model, $1-z_0$ is the fraction of individuals infected in a large epidemic. It is easy to see that this proportion goes to 0
	at the critical value $\mu_c = 1/\tau_r ^f =1$. The next simulation suggests that this is not true in the case of evoSIR.

	\begin{figure}[h] 
		\centering
		\includegraphics[width=4in,keepaspectratio]{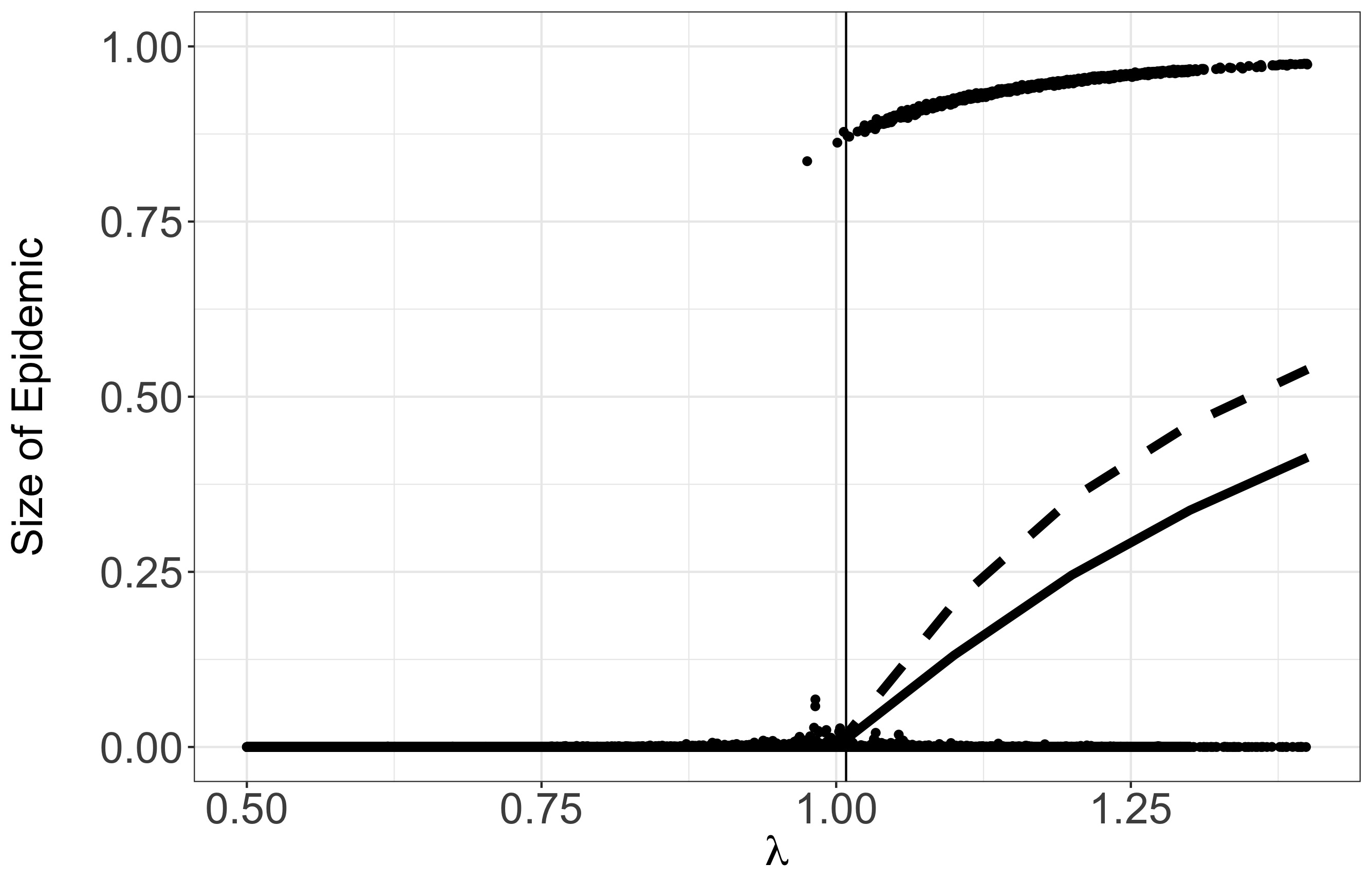}
		\caption{Simulation of the fixed time evoSIR on an \ER\, graph with $\mu=5$, $\rho=4$ and $\lambda$ varying.  $\lambda_c \approx 1.0084$ in agreement with Theorem \ref{ftcrit}. The bottom curve is the final size of the delSIR epidemic with the same parameters. The dashed line above it is an approximation derived in \cite{DOMath} that  turned out to be inaccurate. The top curve comes from simulating evoSIR.}
		\label{fig:disco}
	\end{figure}
	
	\subsection{Britton et al. \cite{BJS,LBSB} }\label{britton}

	As the authors of \cite{DOMath} were finishing up the writing of their paper, they learned of two papers by Britton and collaborators that study epidemics on evolving  graphs with  exponential infection times. \cite{BJS} studies a one parameter family of models (SIR-$\omega$) that interpolates between delSIR and evoSIR. To facilitate later referencing we attach labels to the next two descriptions.

	\begin{model}\label{siromega}
		{\bf SIR-$\omega$ epidemic.} In this model, an infected individual infects each neighbor at rate $\lambda$, and recovers at rate $\gamma$. A susceptible individual drops its connection to an infected individual at rate $\omega$. The edge is rewired with probability $\alpha$ and dropped with probability $1-\alpha$. Since evoSIR ($\alpha=1$) and delSIR ($\alpha=0$) have the same critical values and survival probability it follows that this holds for all $0 \le\alpha\le 1$ since by a coupling argument the final epidemic size of the case $0<\alpha<1$ can be sandwiched between delSIR and evoSIR.
	\end{model}
	
	These epidemics take place on graphs generated by 
	\begin{model} \label{config}
		{\bf Configuration model.} Given a nonnegative integer $n$ and a positive integer valued random variable $D$, take $n$ i.i.d.~copies $D_1,\ldots, D_n$ of $D$. If the sum  $\sum_{i=1}^n D_i$ is odd, then we replace $D_n$ by $D_n+1$. We then construct a graph $G$ on $n$ vertices as follows. We attach $D_1,\ldots, D_n$ half-edges to vertices $1, 2\ldots, n$, respectively and then pair these half-edges uniformly at random to form a graph. We call this random graph \textit{the configuration model} on $n$ vertices with degree distribution $D$ and denote it by $\CM(n,D)$.  We assume $D$ has finite second moment so that the resulting graph will be a simple graph with nonvanishing probability as $n\to\infty$. See Theorem 3.1.2 in \cite{RGD}. We refer readers to \cite[Chapter 7]{vdH1}, \cite[Chapters 4 and 7]{vdH2} and \cite[Chapter 2]{vdH3} for more details on the configuration model.
	\end{model}
	
	\begin{remark}\label{rem:config}
		Throughout the paper, unless otherwise specified,  we always consider the annealed probability measure with respect to the configuration model. In other words the randomness is taken over both the degrees $D_1, \ldots, D_n$ and the construction of $\CM(n,D)$ based on  the degrees.
	\end{remark}
	Britton, Juher, and Saldana \cite{BJS} studied the initial phase of the epidemic starting with one infected at vertex $x$ (chosen uniformly at random from all vertices) using a  branching process approximation. Let $Z^n_m$ be the number of vertices at distance $m$ from $x$ in the graph $\CM(n,D)$. For any fixed $k\in \N$, $\{Z^n_m, 0\leq m\leq k\}$ converges to  the following two-phase branching process $\{Z_m,0\leq m\leq k\}$.
	The number of children in the first generation has the distribution $D$ while subsequent generations have the distribution $D^*-1$ where $D^*$ is the {\it size-biased} degree distribution 
	$$
	\P( D^* = j) = \frac{jp_j}{m_1}, \quad j\geq 0.
	$$
	Here $p_j=\P(D=j)$ and $m_1 = \E(D)$ is the mean of $D$. 
	Later we will also use $m_i:=\E(D^i)$ to denote the $i$-th moment of $D$ for $i\geq 1$. This  follows  from the construction of the configuration model: $x$ connects to other vertices with probability proportional to their degrees so individuals in generations $m \ge 1$ have the $D^*-1$ children instead of $D$.  The `-1' is because   one edge is used in making the connection from $x$.
	Before moving on to epidemics on the configuration model, we note that if
	\beq
	G(z) = \sum_{k=0}^\infty p_k z^k,
	\label{F_2}
	\eeq
	then the generating function of $D^*-1$ is
	\beq
	\hat{G}(z) = \sum_{j=1}^\infty \frac{jp_j}{m_1} z^{j-1} 
	= \frac{G'(z)}{G'(1)}.
	\label{G1f}
	\eeq
	
	In the SIR-$\omega$ model, the probability that an infection will cross an $S-I$ edge is 
	$$
	\tau=\frac{\lambda}{\lambda+\gamma + \omega}.
	$$ 
	Thus we get another two-phase branching process $\bar Z_m$ defined as follows. $\bar Z_0=1$, $\bar	Z_1=\mbox{Binomial}(D,\tau)$ and  future generations have offspring distrbution Binomial$(D^*-1,\tau)$. 
	One can see from this description that the limiting branching process $\bar Z_m$ will have positive survival probability if 
	$$
	1 < \tau \E(D^*-1) = \tau \left(\frac{m_2}{m_1} - 1 \right)= \frac{(m_2-m_1)\tau}{m_1}.
	$$
	Correspondingly, 
	there will be a large epidemic in the SIR-$\omega$ model if
	\beq
	R_0 = \frac{\lambda}{\lambda+\gamma + \omega} \cdot \frac{m_2-m_1}{m_1} > 1.
	\label{R0omega}
	\eeq
	This follows from results on percolation in random graphs (see \cite{Fou} and \cite{Jperc}). 
	Thus for fixed values of $\gamma$ and $\omega$
	\beq
	\lambda_c = (\gamma+\omega) \frac{m_1}{m_2-2m_1}.
	\label{crvomega}
	\eeq
	When $\gamma=0$ and $\omega=\rho$ which is the SI-$\omega$ model in our notation,
	\beq
	\lambda_c = \frac{\rho m_1}{m_2-2m_1}.
	\label{crvSI}
	\eeq
	The critical values of delSI and evoSI only depends on the ratio $\rho/\lambda$, so it is natural to define a parameter  (this $\alpha$ is different from the $\alpha$ used in the definition of SIR-$\omega$ model)
	\begin{equation}\label{alpha}
		\alpha=\rho m_1/\lambda
	\end{equation} 
	that has $\alpha_c =m_2-2m_1$.
	We will only consider the del  and evo endpoints of the one parameter family of models SIR-$\omega$, so after the discussion of previous work is completed, there should be no confusion between our $\alpha$ and theirs.

	Work of Leung, Ball, Sirl, and Britton \cite{LBSB} demonstrated the paradoxical fact that individual preventative measures may lead to a larger final size of the epidemic. They proved this rigorously for SI epidemics on the configuration model with two degrees and conducted simulation studies for many social networks.  Note that in Figure \ref{fig2} (taken from Figure 1 of \cite{LBSB}) the final size increases with the rewirng rate when the rewiring rate is small. This simulation does not suggest that the phase transition was discontinuous.
	
	\begin{figure}[h!]
		\centering
		\includegraphics[width=4in,keepaspectratio]{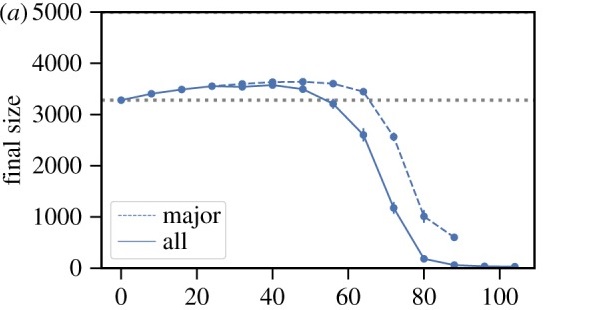}
		\caption{Social distancing can lead to an increase in the final epidemic size in the configuration model. The $x$-axis indicates the rewiring rate. The horizontal line is the final size when $\omega=0$.}
		\label{fig2}
	\end{figure}

	On the other  hand, Figure \ref{fig:disco2} from \cite{DOMath} gives a similar simulation that clearly shows the discontinuity. 

	\begin{figure}[h!] 
		\centering
		\includegraphics[width=4in,keepaspectratio]{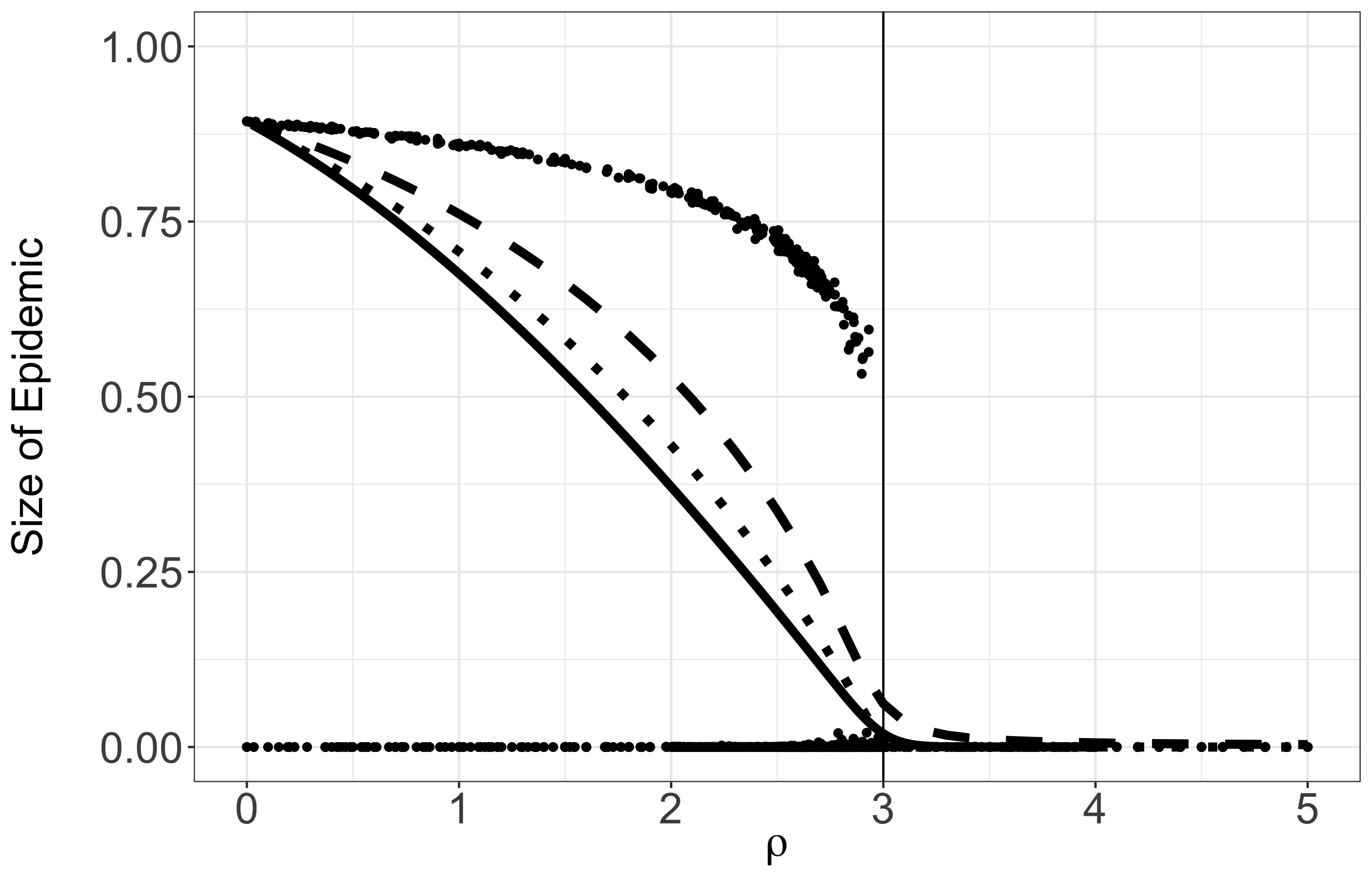}
		\caption{Simulation of a continuous time \ER \ graph with $\mu=5$, $\gamma=1$, $\lambda =1$ and $\rho$ varying. Note that $\rho_c \approx 3$, the value predicted by \eqref{crvomega}. The lowest line is the final size of delSIR which is continuous. The dotted and dashed lines are approximations derived in \cite{DOMath} that  turned out to be inaccurate. The top curve comes from simulating evoSIR.}
		\label{fig:disco2}
	\end{figure}

	\subsection{Ball and Britton \cite{BB}} \label{sec:BB}
	
	Ball and Britton \cite{BB} analyzed the evoSIR and evoSI epidemics on  \ER\ graphs. Their construction
	uses properties that are special to that case. See \cite[Section 2.3]{BB} for
	details and  also \cite{BrOn,Neal}
	for earlier examples of the use of this
	construction.  They solved the case of the SI-$\omega$ model on \ER\ graphs completely, but for the SIR-$\omega$ model there is a gap between the necessary and sufficient condition for a discontinuous phase transition.
	See \eqref{gap0} and the comment after it. 
	
	To prove results about the epidemics on an \ER \ random graph with mean degree $\mu$ they first consider a tree in which each vertex has a Poisson($\mu$) number of descendants and develop a branching process approximation for the SIR-$\omega$ epidemic in which infections (births) cross an edge with probability $\lambda/(\lambda+\gamma+\omega)$. Let $I(t)$ be the total number of infected individuals, $I_E(t)$ be the number of infectious edges, and $T(t)$ be the total progeny in the branching process on the tree. Let $I^n(t)$, $I^n_E(t)$ and $T^n(t)$ be the corresponding quantities for an \ER($n,\mu/n$) random graph on $n$ vertices where initially a randomly chosen vertex is infected. They show in their Theorem 2.1 that if $t_n = \inf\{ t : T(t) \ge \log n\}$ then the two systems can be defined on the same space so that
	$$
	\sup_{0 \le t \le t_n} | ( I^n(t), I^n_E(t), T^n(t)) -   ( I(t), I_E(t), T(t)) |
	\CP 0
	$$
	as $n\to\infty$.
	
	Let $S^n(t)$ be the number of susceptibles at time $t$ and 
	$W^n(t)$ be the number of susceptible-susceptible edges created by rewiring by time $t$ and let $X^n(t) = ( S^n(t), I^n(t), I^n_E(t), W^n(t))$. Let $x(t) =(s(t), i(t), i_E(t),w(t))$ be the solution of the ODE
	\begin{align}
		\frac{ds}{dt} & = -\lambda i_E,
		\nonumber \\
		\frac{di}{dt} & = - \gamma i + \lambda i_E,
		\label{BBode} \\
		\frac{di_E}{dt} & = -\lambda i_E + \lambda\mu i_E s - \lambda \frac{i_E^2}{2}
		+ 2\lambda i_E \frac{w}{s} -\omega i_E(1-\alpha + \alpha (1-i)),
		\nonumber\\
		\frac{dw}{dt} & = w\alpha i_E s -  2\lambda i_E \frac{w}{s}.
		\nonumber
	\end{align}
	
	\noindent
	Here $\alpha$ is the probability that an edge is rewired as in the definition of SIR-$\omega$ model. 
	Theorem 2.2 in \cite{BB}, which is proved  using results of Darling and Norris \cite{DarNor}, shows that for any $t_0>0$
	$$
	\sup_{0 \le t \le t_0} | X^n(t)/n - x(t) | \CP 0
	$$
	as $n\to\infty$, provided that $I^n(t)/n\to i(0)>0$. It is interesting to note, see their Section 3, that the ODE in \eqref{BBode} is closely related to the ``pair approximation'' for SIR-$\omega$ model. 
	
	To explain the phrase in quotes, we note that ``mean-field equations'' come from pretending that the states of site are independent; the pair approximation from assuming it is a Markov chain. In practice, this approach means that probabilities involving three sites are reduced to probabilities involving 1 and 2 sites using a conditional independence property. For the details of the computation see Chapter 7 in \cite{AB}.

	Letting $T^{n} = n - S^n(\infty)$ (which is the final size of the epidemic) they make Conjecture 2.1 that one can interchange two limits $n\to\infty$ and $t\to\infty$ to conclude
	$$
	T^n/n \to 1 - s(\infty).
	$$
	To formulate a result that is independent of the validity of the conjecture they 
	let $$x^\ep(t) = (s^\ep(t),i^\ep(t),i_E(t),w^\ep(t))$$ be the solution to the ODE 
	when 
	$$
	x^\ep(0) = (1-\ep,\ep,L^{-1}\ep,0) \quad\hbox{where}\quad
	L = \frac{\lambda}{\lambda(\mu-1) - \omega}.
	$$
	Letting $\tau_{SIR} = 1-\lim_{\ep\downarrow 0} s^\ep(\infty)$, their Theorem 2.3 states that
	\beq
	\lim_{\lambda \downarrow \lambda_c} \tau_{SIR}
	\begin{cases} = 0 
		&\hbox{ if $\gamma > \omega(2\alpha-1)$ 
			or $\mu < 2\omega\alpha/[\omega(2\alpha-1) - \gamma]$,}\\
		> 0 
		&\hbox{ if $\gamma < \omega(2\alpha-1)$ 
			and $\mu > 2\omega\alpha/[\omega(2\alpha-1) - \gamma]$.}
	\end{cases}
	\label{BBSIR}
	\eeq
	When $\alpha=1$, $\omega=\rho$ and $\gamma=1$ then we have a discontinuous phase transition if 
	$$
	\rho>1 \quad\hbox{and} \quad \mu > \frac{2\rho}{\rho-1}.
	$$

	For the case of the delSIR model ($\alpha=0$, $\omega=\rho$ and $\gamma=1$) studied in \cite{DOMath}, the transition is always continuous. This follows from Theorem 2.3 since  $1 > -\rho$. From the last calculation we see that the phase transition is always continuous if $\alpha < 1/2$.
	
	In the case of the SI-$\omega$ model they show (see Theorem 2.6 of \cite{BB}) that if $\mu > 1$ and  $\omega$ and $\alpha$ are held fixed then the phase transition is discontinuous if and only if $\alpha > 1/3$ and $\mu > 3\alpha/(3\alpha-1)$. When $\alpha=1$ this is $\mu > 3/2$ which is the condition in Example \ref{PoiSI} in Section \ref{sec:state}.
	
	Theorem 2.4 in \cite{BB} gives results for the epidemic starting from a single infected individual. If we let $\tau^1_{SIR}$  be the limiting fraction of final epidemic size conditionally on a large outbreak then
	\begin{equation}\label{gap0}
		\lim_{\lambda \downarrow \lambda_c} \tau^1_{SIR}
		\begin{cases} = 0 
			&\hbox{ if $\gamma > \omega(3\alpha-1)$ 
				or $\mu < 3\omega\alpha/[\omega(3\alpha-1) - \gamma]$,}\\
			> 0 
			&\hbox{ if $\gamma < \omega(2\alpha-1)$ 
				and $\mu > 2\omega\alpha/[\omega(2\alpha-1) - \gamma]$.}
		\end{cases}
	\end{equation}
	See their paper for a precise statement. Remark 2.4 in \cite{BB} states the conjecture that the 3's in the first condition should be 2's.

	For SI-$\omega$ model they compute the fraction of final epidemic size $\tau^1_{SI}$
	$$
	\lim_{\lambda\downarrow \lambda_c} \tau^1_{SI} := \tau_0(\mu,\alpha).
	$$
	To give the value of  $\tau_0(\mu,\alpha)$, we need some notations. For $\mu>1$ and $\alpha\in[0,1]$ let
	$$
	\theta(\mu,\alpha) = \frac{2\alpha(\mu-1)}{\mu + \alpha(\mu-1)}\mbox{ and }
	f_0(x) = \log(1-x) + \frac{x}{1-\theta(\mu,\alpha)},
	$$
	then
	$\tau_0(\mu,\alpha)$ is the largest solution in $[0,1)$ of $f_0(x)=0$. 
	
	Ball and Britton also made connections of their paper with our paper (as well as an earlier version of this paper) in \cite[Section 4]{BB}. In particular, they showed in their Figure 5 that 
	in the case $\mu=2$  their predicted final size $\tau_0(2,1)$  agrees with simulation results well.

	\subsection{Statement of our main results}\label{sec:state}
	
	From now on the reader can forget about the meaning of notations used by Ball and Britton. We fix $\rho$, the rewiring rate, and vary $\lambda$. We let $\alpha=\rho m_1/\lambda$.  In view of the definition of
	$\Delta$ in \eqref{deldef}, the natural assumption is $\E(D^3) < \infty$. Some of our
	results can be proved under this assumption, while some need something a
	little stronger. Specifically, we need finite fifth moment to prove 
	\eqref{delta>0new}.
	To simplify things \textbf{we assume  $\E(D^5) < \infty$ throughout.}
	
	\begin{theorem}\label{critical_value}
		Assume $\E(D^5)<\infty$.
		Consider the delSI model and evoSI model on $\CM(n,D)$.
		(i) the critical values of delSI and evoSI are the same, i.e., 
		\begin{equation*}
			\alpha_c=m_2-2m_1; \mbox{ equivalently, }\lambda_c=\frac{\rho m_1}{m_2-2m_1}. 
		\end{equation*} (ii) When $\alpha < \alpha_c$, which is the supercritical case, the probability of a large epidemic is the same in the two models, which is equal to the survival probability $q(\lambda)$ of the two-stage branching process $\bar Z_m$ defined in Section \ref{britton} (with the $\tau$ there equal to $\lambda/(\lambda+\rho)$ in our notation).
	\end{theorem}
	
	\noindent
	The proof of Theorem \ref{critical_value} that we give in Section 2 is very similar to one for Theorem \ref{ftcrit} given in \cite{DOMath} for \ER \ random graphs. Here the fact that we have only  $\E(D^5)<\infty$ rather than exponential upper bounds on $\P(D\ge k)$ changes some of the estimates.
	
	Here and in what follows, formulas are sometimes easier to evaluate if we use the ``factorial moments'' $\mu_k = \E[D(D-1)\cdots (D-k+1)]$, since these can be computed  from the $k$-th derivative of the generating function. To translate between the two notations:
	$$
	\mu_1=m_1, \qquad \mu_2=m_2-m_1, \qquad \mu_3 = m_3 -3m_2 + 2m_1.
	$$
	In particular $\alpha_c = m_2-2m_1 = \mu_2 - \mu_1$.
	
	\medskip
	Our next result gives an almost sufficient and necessary condition for the discontinuous phase transition of evoSI.  We use the word `almost' since the case $\Delta=0$ is not treated here. 
	\begin{theorem}\label{Q1new} 
		Assume $\E(D^5)<\infty$. Consider the evoSI epidemic on the configuration model $\CM(n,D)$  with one uniformly randomly chosen vertex initially infected. Let 
		\beq
		\Delta= -\frac{\mu_3}{\mu_1} + 3 (\mu_2-\mu_1).
		\label{deldef}
		\eeq
		Let $I_{\infty}$ be the final epidemic size. 
		If  $\Delta>0$, then there is a discontinuous phase transition. For some $\ep_0>0$ and some $\delta_0>0$, 
		\begin{equation}\label{delta>0new}
			\lim_{\eta\to 0} \liminf_{n\to\infty}\P_{1}(I_{\infty}/n>\ep_0|I_{\infty}/n>\eta)=1 \quad\hbox{for all $\alpha_c-\delta_0<\alpha< \alpha_c$.}
		\end{equation}	
		If $\Delta<0$, then there a continuous phase transition. For any $\ep>0$, there exists some $\delta>0$, so that
		\begin{equation}\label{delta<0new}
			\lim_{n\to\infty}\P_{1}(I_{\infty}/n>\ep)=0 \quad\hbox{for $\alpha_c-\delta<\alpha< \alpha_c$.}
		\end{equation}
	\end{theorem}
	\noindent
	
	To see what this result says we consider some examples. 
	
	\begin{example}
		{\bf Random $r$-regular graph, $r\ge 3$.} Here $nr$ must be even.
		If we choose the degree distribution $\P(D=r)=1$, and condition the graph to be simple, i.e., no self-loops or parallel edges
		then the result is a random regular graph. The case $r=2$ is excluded because in that case the graph consists of a number of circles.
		The critical value is $\alpha_c = m_2-2m_1 = r^2 - 2r > 0$ when $r>2$. 
		For $k\leq r$, $$\mu_k = r(r-1)\cdots (r-k+1),$$ so 
		\begin{align*}
			\Delta & = -(r-1)(r-2) + 3(r(r-1) - r) \\
			& = -(r-1)(r-2) + 3r(r-2) =(r-2)(2r+1) >0,
		\end{align*}
		and the phase transition is discontinuous for all $r\ge 3$.
	\end{example}

	\begin{example}
		{\bf Geometric($p$).} The factorial moments are $\mu_1=1/p$, $\mu_2 = 2(1-p)/p^2$, and $\mu_3 = 6(1-p)^2/p^3$. $\alpha_c = \mu_2 - \mu_1 = (2-3p)/p^2$, so we need to take $p<2/3$ to have $\alpha_c>0$. 
		\begin{align*}
			\Delta & = - \frac{6(1-p)^2}{p^2} + 3 \left( \frac{2(1-p)}{p^2} - \frac{1}{p} \right) \\
			& = -\frac{6}{p^2} + \frac{12}{p} -6 + \frac{6}{p^2} - \frac{6}{p} - \frac{3}{p}  = \frac{3}{p} - 6,
		\end{align*}
		so the phase transition is discontinuous if $p<1/2$.
	\end{example}

	Our last example concerns the configuration model generated from Poisson distribution:
	
	\begin{example} \label{PoiSI}
		{\bf  Poisson($\mu$).}  The factorial moments  $\mu_k = \mu^k$, so the critical value $\alpha_c = \mu_2-\mu_1 = \mu^2-\mu$, which is positive if $\mu>1$. This condition is natural since if $\mu<1$ then there is no giant component in the graph and a large epidemic is impossible. 
		$$
		\Delta= -\mu^2 +3(\mu^2-\mu) = 2\mu^2-3\mu^2,
		$$
		so the phase transition for evoSI is discontinuous if $\mu>3/2$, which is the result given in \cite{BB}.
	\end{example}
	
	\begin{remark}
		We believe that the result of Example \ref{PoiSI} also holds for \ER($n,\mu/n$). To prove this rigorously,  one first has to prove a quenched version of Theorem \ref{Q1new} (i.e., showing that \eqref{delta>0new} and \eqref{delta<0new} hold with high probability over \emph{any} degree sequence $D_1,\ldots, D_n$ (that are \emph{not necessarily i.i.d.}) such that the $k$-th factorial moment of the empirical distribution converges to $\mu^k$ for any $k\geq 0$). We believe that this can be shown using the same ideas  in the proof of Theorem \ref{Q1new}.  Then one can transfer results for the configuration model  to \ER($n,\mu/n$) using
		\cite[Theorems 7.18 and 7.19]{vdH1}, which says that conditionally on having the same degrees, the random graphs generated from these two models have the same distribution.
	\end{remark}
	As a notational note, in this paper  we will use $C, C_1, C_2,\cdots$ to denote various constants whose specific values might change from line to line. Occasionally when we have an important constant we will number it by the formula it first appeared in, e.g., $C_{\ref{maxbd}}$ below.
	
	\subsection{Sketch of Proof of Theorem  \ref{Q1new} }\label{pfth15}
	The proof of Theorem \ref{Q1new} is done by constructing auxiliary models that are upper/lower bounds for evoSI.  
	We introduce a process which we call avoSI (avo is short for avoiding infection)   in Section \ref{sec:avoSI} and prove that the final set of infected sites in avoSI stochastically dominates evoSI. We also construct a lower bounding process which we call AB-avoSI in Section \ref{sec:ABavoSI}, where we prove that the final set of infected sites in evoSI stochastically dominates AB-avoSI. The $AB$ in the name comes from the two counters associated with half-edges that prevent transmission of  infections along $S-I$ edges created by $I-I$ rewirings.
	
	The starting point to analyze evoSI via  avoSI and AB-avoSI is the following Lemma \ref{avocrv}. Let $q(\lambda)$ be the survival probability for the two-phase branching process $\{\bar Z_m,m\geq 0\}$ introduced in Section \ref{britton}. Recall that the individual in the first generation has offspring distribution $\mbox{Binormial}(D,\lambda/(\lambda+\rho))$ while later generations have offspring distribution $\mbox{Binormial}(D^*-1,\lambda/(\lambda+\rho))$ where $D^*$ is the size-biased version of $D$,
	$$
	\P(D^*=j) = \frac{j \P(D=j)}{\E(D)},j\geq 0.
	$$
	\begin{lemma} \label{avocrv}
		AvoSI, evoSI, AB-avoSI  and delSI have the same critical value $\lambda_c$ and in the supercritical regime $\lambda>\lambda_c$ the probability of a large outbreak is equal to $q(\lambda)$ in all four models.  
	\end{lemma}
	
	\begin{proof}
		Results in Section 3.1 and 4.1 imply that  
		$$
		\hbox{avoSI $\succeq$ evoSI $\succeq$ AB-avoSI $\succeq$ delSI}
		$$ 
		where  epidemic1 $\succeq$ epidemic2 means that the two epidemics can be constructed on the same space so that the final epidemic size in epidemic1 is greater than or equal to that of epidemic2.
		In fact, we will prove this chain of comparisons  in Lemmas \ref{avosi>evosi}, \ref{abavosi<evosi} and \ref{abdel}, respectively. 
		It remains to show that avoSI and delSI has the same critical value and probability of a large outbreak. This is proved in Lemma \ref{critical_avosi}.
	\end{proof}

	Below we will use $\lambda_c$ and $\alpha_c=\rho \lambda_c/m_1$  to denote the critical value.
	Recall the definition of the generating function $G$ in \eqref{F_2}. 
	Consider a function $f$ defined by
	\begin{equation}\label{defoff}
		f(w)=\log \left(\frac{m_1 w}{G'(w)+\alpha (1-w)G(w)}\right) + \frac{\alpha}{2}(w-1)^2.
	\end{equation}
	\begin{theorem}	\label{epsize}
		Assume $\E(D^5)<\infty$.
		Consider the avoSI epidemic on the configuration model $\CM(n,D)$ with one uniformly randomly chosen vertex initially infected.  Suppose $\alpha<\alpha_c$ so that we are in the supercritical regime.
		Let $\tilde{I}_{\infty}$ be the final epidemic size.
		Set
		\begin{align}
			\label{mM}
			\sigma &=\sup\{w:0<w<1,f(w)= 0\} \quad\hbox{with $\sup(\emptyset) = 0$}, \\
			\label{limitsize}
			\nu& =1-\exp\left(-\frac{\alpha}{2}(\sigma-1)^2 \right)G(\sigma).
		\end{align}
		If we suppose 
		
		\medskip
		$(\star)$ either $\sigma=0$  or  $0<\sigma<1$ and there is a $\delta>0$ so that $f<0$ on $(\sigma-\delta,\sigma)$, 
		
		\mn
		then for any $\ep>0$,
		$$
		\lim_{n\to\infty} \P(\tilde{I}_{\infty}/n<\nu+\ep )=
		\lim_{\eta\to 0}\liminf_{n\to\infty} \P(\tilde{I}_{\infty}/n>\nu-\ep |\tilde{I}_{\infty}/n>\eta)=1.
		$$
	\end{theorem}
	
	\noindent
	Though $\nu$ does not give the correct final size of the evoSI epidemic,  the formula for $f(w)$ is accurate enough for $w$ near 1 to identify when the phase transition is continuous.
	\begin{theorem}\label{Q1} 
		Consider the avoSI epidemic on the configuration model $\CM(n,D)$ and let $\tilde I_\infty$ be the final epidemic size.  Set
		\begin{equation}\label{deltadef}
			\Delta= -\frac{\mu_3}{\mu_1} + 3 (\mu_2-\mu_1).
		\end{equation}
		If $\Delta<0$, then there a continuous phase transition. For any $\ep>0$, there exists some $\delta>0$, so that
		\begin{equation}\label{delta<0}
			\lim_{n\to\infty}\P(\tilde I_{\infty}/n>\ep)=0 \quad\hbox{for $\alpha_c-\delta<\alpha< \alpha_c$.}
		\end{equation}
	\end{theorem}
	
	We can show that $\Delta>0$ implies that there is a discontinuous phase transition in avoSI, but that result does not help us prove Theorem \ref{Q1new}.
	To get Theorem \ref{Q1} from Theorem \ref{epsize} we compute,
	see Section \ref{sec:pfth7}, that
	
	\begin{equation}\label{fprop}
		\begin{split}
			f'(1) &= - \left( \frac{m_2-2m_1}{m_1} - \frac{\rho}{\lambda} \right) \quad\hbox{which is $<0$ for $\alpha< \alpha_c$},\\
			f'(1) &= 0, \quad f''(1) = \Delta \quad\hbox{when $\alpha=\alpha_c$}.
		\end{split}
	\end{equation}

	When $\Delta>0$, as $w$ decreases from 1 the curve of $f$ turns up, and $\sigma$ stays bounded away from 0.
	When $\Delta<0$,   the curve  of  $f$ turns down, and $\sigma$ converges to 1 as $\alpha\to\alpha_c$. See Figure 4.

	\newpage

	\begin{figure}[h!] 
		\begin{center}
			\includegraphics[height=3.2in,keepaspectratio]{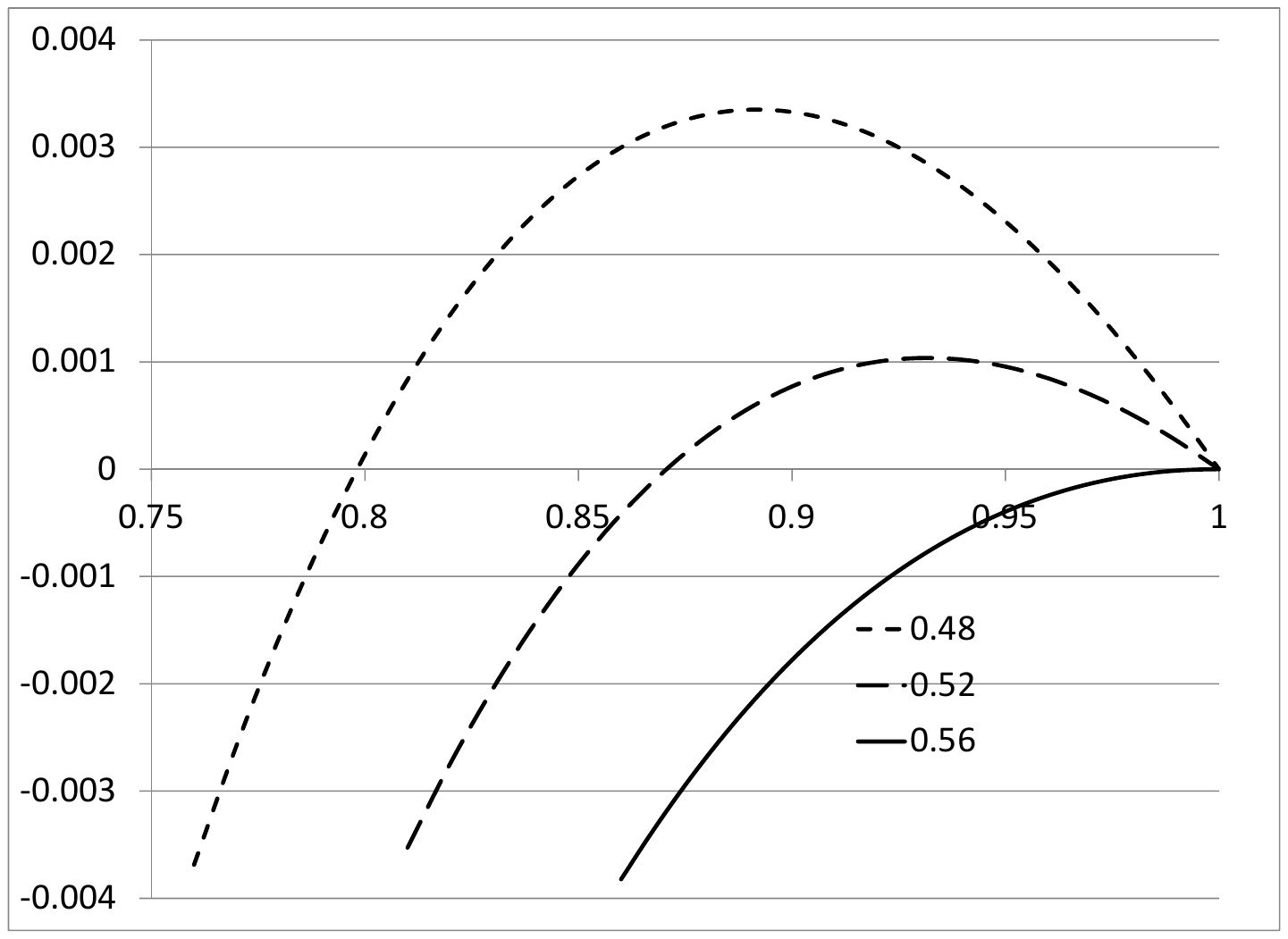}
			
			\includegraphics[height=3.2in,keepaspectratio]{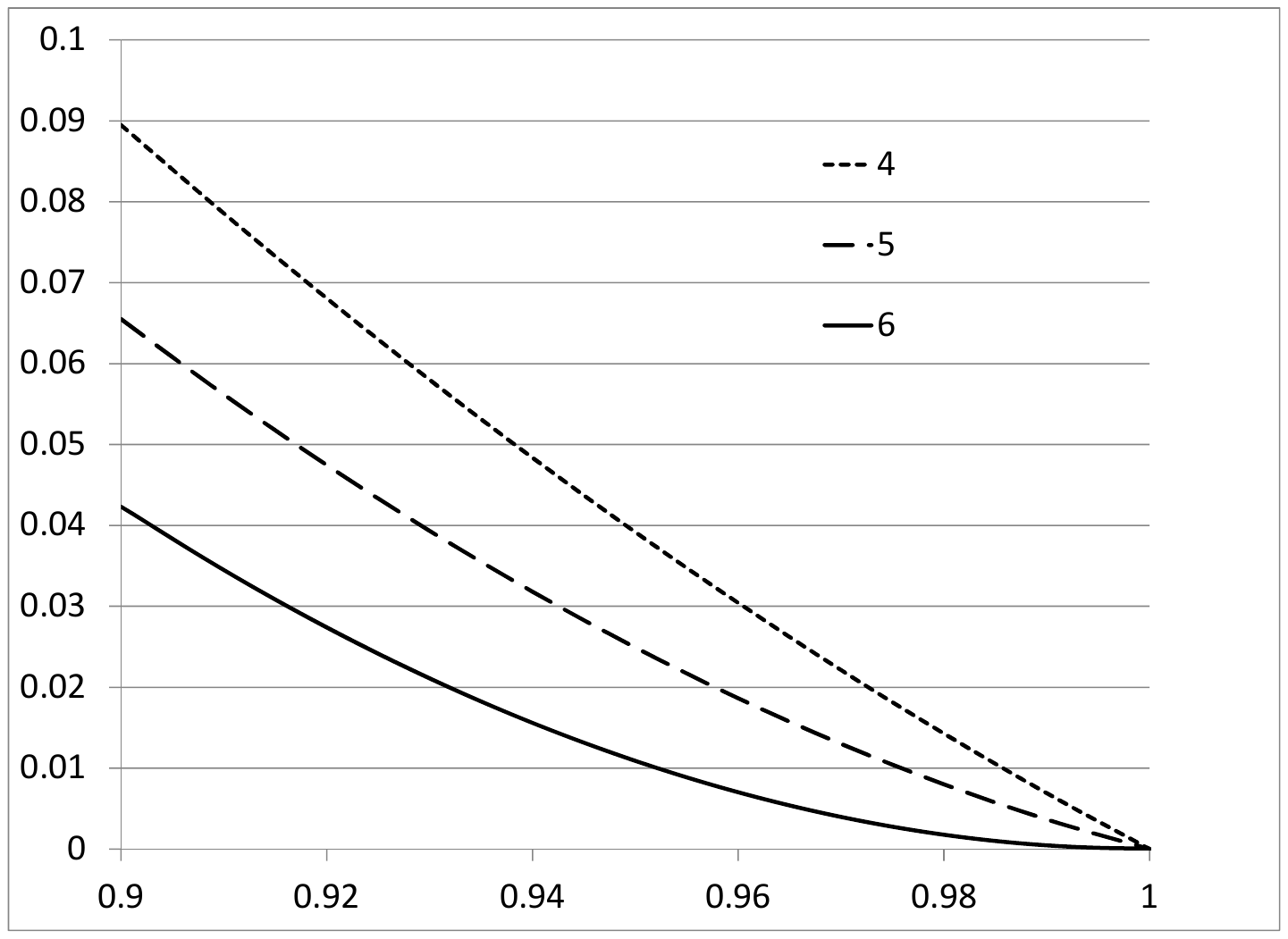}
			\caption{The behavior of $f(w)$  near 1 with respect to different $\alpha's$  for the \ER\, graph. In the top graph  $\mu=1.4$, which has $\Delta<0$. $\alpha_c = \mu^2-\mu = .56$. Notice that as $\alpha$ increases to 0.56 the intersection with the $x$ axis tends to 1, so the transition is continuous. In the bottom graph $ \mu= 3$, which has $\Delta>0$. $\alpha_c = \mu^2-\mu = 6$. Notice that when
				$\alpha\le\alpha_c$, $f(w)> 0$ for $w\in[0.9,1)$, so $\sigma$ is bounded away from 1.}
		\end{center}
		\label{fig:ER2}
	\end{figure}

	\begin{theorem}\label{Q1AB} 
		Consider the AB-evoSI epidemic on the configuration model $\CM(n,D)$. Let $\check{I}_{\infty}$ be the final epidemic size.  Let $\Delta$ be the quantity defined in \eqref{deltadef}.
		If  $\Delta>0$ then there is a discontinuous phase transition. For some $\ep_0>0$ and some $\delta_0>0$, 
		\begin{equation}\label{delta>0AB}
			\lim_{\eta\to 0} \liminf_{n\to\infty}\P_{1}(\check{I}_{\infty}/n>\ep_0|\check{I}_{\infty}/n>\eta)=1 \quad\hbox{for all $\alpha_c-\delta_0<\alpha< \alpha_c$.}
		\end{equation}	
	\end{theorem}
	\mn
	Theorem \ref{Q1new} follows from Theorem \ref{Q1} and Theorem \ref{Q1AB}.
	\begin{proof}[Proof of Theorem \ref{Q1new}]
		Lemma \ref{avocrv} implies that
		\begin{equation}\label{cmp}
			\lim_{\eta\to0} \liminf_{n\to\infty}\P(I_{\infty}/n>\eta)=\lim_{\eta\to 0}\liminf_{n\to\infty} \P(\check{I}_{\infty}/n>\eta).
		\end{equation}
		Equations \eqref{cmp}, \eqref{delta>0AB} and the fact evoSI $\succeq$ AB-avoSI (proved in Lemma \ref{abavosi<evosi}) imply that
		$$
		\lim_{\eta\to 0} \liminf_{n\to\infty}\P(I_{\infty}/n>\ep_0|I_{\infty}/n>\eta)\geq 
		\lim_{\eta\to 0} \liminf_{n\to\infty}\P(\check{I}_{\infty}/n>\ep_0|\check{I}_{\infty}/n>\eta)=1, 
		$$
		for $\alpha_c-\delta_0<\alpha< \alpha_c$.
		This is exactly \eqref{delta>0new}. 	Equation \eqref{delta<0new} follows from \eqref{delta<0} and the fact that avoSI $\succeq$ evoSI. 
	\end{proof}

	\subsection{Sketch of Proof of Theorem \ref{epsize}} \label{sec:pfstrategy}

	To begin to explain the ideas behind the  analysis of epidemics on evolving graphs we need to recall some history. Volz \cite{Volz} was the first to derive a limiting ODE system for an SIR epidemic on a (static) graph generated by the configuration model. Miller \cite{Miller} later simplified the derivation to produce a single ODE. The results of Volz and Miller were based on heuristic computations, but later their conclusion was made rigorous by Decreusfond et al \cite{Decr} assuming $\E (D^5)<\infty$. 
	
	Janson, Lukzak, and Windridge \cite{JLW} proved the result under more natural assumptions. They studied the epidemic on the graph by revealing its edges dynamically while the epidemic spreads. 
	Recall that the configuration is constructed using half-edges.  The authors in \cite{JLW} call a half-edge free if it has not yet been paired with another half-edge. They call a half-edge susceptible, infected or removed according to the state of its vertex. To modify their construction to include rewiring we add the third bullet below.   Hereafter we use``randomly chosen'' and ``at random''  to mean that the distribution of the choice is uniform over the set of possibilities.
	\begin{itemize}
		\item
		Each free infected half-edge chooses a free half-edge at rate $\lambda$.
		Together the pair forms an edge and is removed from the collection of half-edges. If the pairing is with a susceptible half-edge then
		its vertex becomes infected and all its edges become infected half-edges. 
		
		\item
		Infected vertices recover and enter the removed state at rate 1.
		
		\item 
		Each infected half-edge gets removed from the vertex that it is attached to at rate $\rho$ and immediately becomes re-attached to a randomly chosen vertex. 
	\end{itemize}

	To analyze the avoSI model  we follow the approach in Janson, Luczak, and Windridge \cite{JLW} and construct the graph as we run the infection process. The construction of this process and its coupling to evoSI are described in Section \ref{sec:avoSI}. Initially the graph consists of half-edges connected to vertices, as in the configuration model construction before the half-edges are paired. 
	Let  $\tilde{X}_t$ be the total number of half-edges at time $t$ and let $\tilde{X}_{I,t}$ be the number of half-edges that are attached to infected vertices and let $\tilde S_{t,k}$ be the number of susceptible vertices with $k$ half-edges at time $t$.
	The evolution of $\tilde{S}_{t,k}$ in avoSI is given by, see \eqref{eqStk1},
	$$
	d\tilde{S}_{t,k}=-\left(\lambda \tilde{X}_{I,t}\frac{k\tilde{S}_{t,k}}{\tilde{X}_t-1}\right)dt+\left(1_{\{k\geq 1\}}\rho \tilde{X}_{I,t} 
	\frac{\tilde{S}_{t,k-1}}{n}\right)dt-\left(\rho \tilde{X}_{I,t} 
	\frac{\tilde{S}_{t,k}}{n}\right)dt+d\tilde{M}_{t,k},
	$$
	where $\tilde{M}_{t,k}$ is a martingale and we have returned to using $\lambda$ as the infection rate and $\rho$ as the rewiring rate.

	Following \cite{JLW} we time-change the process by multiplying the original transition rates by $(\tilde{X}_t-1)/(\lambda \tilde{X}_{I,t})$. 
	Let $\X_t$ be the number of half-edges at time $t$ in the time changed process, and let  $\X_{S,t}$ be the number of half-edges that are attached to susceptible vertices.
	Using $\S_{t,k}$ for the time-changed process the new dynamics are, see \eqref{eqStk2},
	\begin{align}
		d\left(\frac{\S_{t,k}}{n}\right) & =-\left(k\frac{\S_{t,k}}{n}\right) dt
		+\left(1_{(k\geq 1)} \frac{\rho }{\lambda}\frac{\X_t-1}{n}\frac{\S_{t,k-1}}{n}\right)dt
		\nonumber\\
		&-\left(\frac{\rho }{\lambda}\frac{\X_t-1}{n}\frac{\S_{t,k}}{n}\right)dt+d\left(\frac{\M_{t,k}}{n}\right).
		\label{tceq}
	\end{align}
	Note that, thanks to the time change, the number of infected half-edges $\X_{I,t}$ no longer appears in the equation.  Let $\gamma_n$ be the first time there are no infected half-edges.   Let
	$w(t)=\exp(-t)$  and $m_1=\E(D)$. The key to the proof of Theorem \ref{epsize} is to show
	\begin{align}
		&\sup_{0\leq t\leq  \gamma_n}\abs{\frac{\X_t}{n}-m_1w(t)^2}\CP 0,
		\nonumber\\
		&\sup_{0\leq t\leq \gamma_n}
		\abs{	\frac{\sum_{k=0}^{\infty} \S_{t,k}}{n}- F_0(w(t)) }\CP 0,
		\label{3lim}\\
		&\sup_{0\leq t  \leq \gamma_n}
		\abs{	\frac{\sum_{k=0}^{\infty} k\S_{t,k}}{n}-F_1(w(t)) }\CP 0,
		\nonumber 
	\end{align} where
	\begin{align}
		F_0(w) & =   \exp(-(\alpha/2)(w-1)^2 )G(w), 
		\nonumber\\
		F_1(w)& = \exp(-(\alpha/2)(w-1)^2 )w(G'(w)+\alpha (1-w)G(w)).
		\label{Fidef}
	\end{align}
	From the results above, we see that
	\beq
	\frac{\X_t}{\X_{S,t}} \CP \frac{m_1 w}{ \exp(-(\alpha/2)(w-1)^2 ) \cdot(G'(w)+\alpha (1-w)G(w)) }.
	\label{whatisf}
	\eeq
	The logarithm of the right-hand side is $f(w)$.
	Under assumption ($\star$),  
	$$
	\sigma = \sup\{ w : 0< w < 1, f(w)=0\}
	$$  
	gives the time $z=-\log(\sigma)$ at which the infection dies out in the time-changed process and $\nu$ defined in \eqref{limitsize} gives the fraction of sites which have been infected. 
	
	There are four steps in the proof of \eqref{3lim}:
	\begin{itemize}
		\item
		In Section \ref{sec:tight} we show that for each fixed $k\in \N$, $\{\S_{t,k}/n, t\geq 0 \}_{n\geq 1}$ is a tight sequence of processes.
		\item
		In Section \ref{sec:convStk} we show that any subsequential limit satisfies a system of differential equations \eqref{stk} that has a unique solution $\bar s_{t,k}$, so  $\S_{t,k}/n \to \bar s_{t,k}$.
		\item
		Section \ref{sec:sum_stk} we deal with the technicality of showing that the limit of $\sum_{k=0}^{\infty} k \bar S_{t,k}/n$ is the sum of the limits  $\sum_{k=0}^{\infty} k\bar s_{t,k}$.
		\item
		In Section \ref{sec:pfth6} we complete the proof by establishing the formulas for $\sigma$ and $\nu$.
	\end{itemize}

	\subsection{Sketch of Proof of Theorem \ref{Q1AB}}\label{pfth18}
	In the AB-avoSI model, each half-edge $i$ has two indices $A(i,t)$ and $B(i,t)$. 
	\begin{itemize}
		\item The infection index   $A(i,t)=0$ if $i$ has not been infected by time $t$. 
		If $i$ first become an infected half-edge at time $s$, then we set 
		$A(i,t)=s$ for all $t\geq s$. 
		\item 
		The rewiring index $B(i,t)=0$ if $i$ has not been rewired by time $t$.
		If $i$ gets rewired at time $s$, then we update the value of $B(i,s)$ to be $s$, regardless of whether $i$ has been rewired before or not. $B(i,\cdot)$ remains constant between consecutive rewirings. 
	\end{itemize}
	
	Suppose an infected  half-edge $i$ pairs with a susceptible half-edge $j$ at tine $t$, then (in the AB-avoSI model) $i$ will transmit an infection to $j$ if and only if $A(i,t)>B(j,t)$. See Section \ref{sec:ABavoSI} for more details about the AB-avoSI model and its relationship to evoSI. Let $\check{S}_{t,k}$ be the number of susceptible vertices with $k$ half-edges at time $t$ and set
	\begin{equation*}
		G_{i,j}=\{I(i,t)=1,  A(i,t)\leq B(j,t) \}.
	\end{equation*}
	Here $I(i,t)$ is an indicator function such that $I(i,t)=1$ if half-edge $i$ is an infected half-edge at time $t$ (see the first paragraph of Section \ref{sec:moment} for the definitions of the notations $I(i,t),S(i,t),S(i,k,t),D(j,t)$ appearing below). 
	As in the avoSI model we make a time-change by multiplying the original transition rates by $(\check{X}_t-1)/(\lambda \check{X}_{I,t})$.  Using a hat to denote the quantities after the time-change in the AB-avoSI model,  we have that, for all $k\geq 0$, 
	\begin{equation*}
		\begin{split}
			d\hat{S}_{t,k}&=-k\hat{S}_{t,k}\, dt
			+1_{\{k\geq 1\}} \frac{\rho}{\lambda}\frac{\hat{S}_{t,k-1}}{n}(\hat{X}_t-1) \, dt
			- \frac{\rho}{\lambda}\frac{\hat{S}_{t,k}}{n}(\hat{X}_t-1)\, dt\\
			&+\frac{1}{\hat{X}_{I,t}}\left( \sum_{i,j=1}^{\hat X_0} 
			1_{G_{i,j}} 1_{\{ S(j,k+1,t)=1 \}}  \right)dt+d\hat{M}_{t,k},
		\end{split}
	\end{equation*} 
	where $\hat{M}_{t,k}$ is a martingale. See equation \eqref{newStk}. This system of equations is not solvable but we can expand in powers of $t$ to study the time-changed system for small $t$. If we let $\hat{X}_t, \hat{X}_{I,t}, \hat{S}_t$ be the number of half-edges, the number of infected half-edges and the number of susceptible vertices, respectively, then we have
	\begin{equation}
		\begin{split}
			\hat{X}_{I,t}&=\hat{X}_{I,0}+\int_0^t\left(-2(\hat{X}_u-1)+\sum_{k=0}^{\infty} k^2\hat S_{u,k}-\frac{\rho}{\lambda}\frac{\hat S_u}{n}(\hat{X}_u-1)  \right) du
			+\hat{M}_t \\
			&-\frac{1}{\hat{X}_{I,t}}\int_0^t \left( \sum_{i,j=1}^{\hat X_0} 1_{G_{i,j}} (D(j,u)-1) 1_{\{ S(j,u)=1\}} \right)du, 
		\end{split}
	\end{equation}
	where $\hat{M}_t$ is a martingale. See equation \eqref{xitnew}.  Define
	\begin{equation}
		E(t)=\frac{1}{\hat{X}_{I,t}}\left( \sum_{i,j=1}^{\hat X_0} 1_{G_{i,j}} D(j,t) 1_{\{ S(j,t)=1\}} \right).
	\end{equation}
	By expanding $\hat{S}_{t,k}$ around $t=0$ up to the second order, we get, for any $\ep>0$, $\lambda$ close to $\lambda_c$ and $t_0$ close to 0, 
	\begin{equation}\label{-lb}
		\begin{split}
			\lim_{n \to \infty}\P\biggl(\hat{X}_{I,t}& \geq \left(\frac{\rho m_1}{2\lambda_c^2}(\lambda-\lambda_c)t+\frac{m_1\Delta}{4}t^2-\ep\right)n \\
			&-\int_0^t E(u)du, \forall 0\leq t\leq \gamma_n\wedge t_0\biggr)=1.
		\end{split}
	\end{equation}
	See \eqref{lbstep12} of Lemma \ref{lbstep2}.
	Here the $\Delta$ is the same as the one in \eqref{deldef}, i.e., 
	$$
	\Delta= -\frac{\mu_3}{\mu_1} + 3 (\mu_2-\mu_1).
	$$
	
	\begin{remark}
		In the case of avoSI, we have that, for $n$ large, 
		\begin{equation*}
			\X_{I,t}=\X_t-\X_{S,t}=\X_{S,t}\left(\frac{\X_t}{\X_{S,t}}-1\right) \sim \X_{S,t}(f(\exp(-t))-1)
		\end{equation*}
		with $f$ defined in \eqref{defoff}. By expanding $f$ up to the second order and using \eqref{fprop} along with the fact that $\exp(-t)=1-t+t^2/2+o(t^2)$, we get 
		\begin{equation}
			\begin{split}
				\X_{I,t}& \sim nm_1 (-f'(1)(t-t^2/2)+f''(1)t^2/2+o(t^2))\\
				&\geq nm_1\left(\rho(\lambda-\lambda_c)t/(2\lambda_c^2)+(\Delta t^2)/4 \right)
			\end{split}
		\end{equation}
		for small $t$ and $\lambda$ close to $\lambda_c$. In the case of AB-avoSI, see \eqref{xit4}, we have,
		$$
		n\left(\frac{\rho m_1(\lambda-\lambda_c)}{4\lambda_c^2} t 
		+ \frac{m_1\Delta}{8} t^2 - \ep_2 - \ep_6 \right)
		$$
		as a lower bound when $t> \ep$ (the $\ep_2$ and $\ep_6$ here are some small numbers depending on $\ep$). The two expansions do not match but both linear terms  vanish at $\lambda_c$ and the quadratic terms have the same sign so this is good enough.
	\end{remark}
	
	The proof of Theorem \ref{Q1AB} is organized into five steps:
	
	\begin{itemize}
		\item In Section \ref{sec:ABavoSI} we define the AB-avoSI process and prove that evoSI stochastically dominates AB-avoSI.
		\item In Section \ref{sec:moment} we derive basic moment estimates for various quantities that will prepare us for later proofs.  See Lemma \ref{lbstep1}. 
		\item In Section \ref{sec:rough} we give rough upper and lower bounds for $\hat{I}_t$ and $\hat{X}_{I,t}$ involving the integral of $E(t)$. See Lemma \ref{lbstep2}. We also give an easy upper bound for $E(t)$ in \eqref{eubound1}.
		\item In Section \ref{sec:refined} we decompose $E(t)$ into two parts (see \eqref{ede}) and give refined bounds for each part. See Lemmas \ref{ctlet2} and \ref{ctlet3}.
		\item In Section \ref{sec:complete} we combine our estimates to complete the proof. 
	\end{itemize}

	\clearp

	\section{Proof of Theorem \ref{critical_value}} \label{sec:pfth2}
	
	\subsection{Coupling of evoSI and delSI}\label{sec:evodel}
	We first prove Lemma \ref{couple} before proving Theorem \ref{critical_value}.
	We define three sequences of random variables, which  will serve as the joint randomness to couple evoSI and delSI:

	\begin{itemize}
		\item
		Let $T_{e,\ell}$, $\ell \ge 1$  be independent exponential random variables with mean $1/\lambda$.  
		
		\item
		Let  $R_{e,\ell}$, $\ell \ge 1$ be independent exponential random variables with mean $1/\rho$.  
		
		\item
		Let $U_{e,\ell}$, $\ell\geq 1$ be independent random variables chosen uniformly at random from all vertices. 
	\end{itemize}

	\medskip
	{\bf Construction of evoSI.}
	We define three sets of edges in evoSI at time $t$:
	\begin{itemize}
		\item  {\it  Active edges},  denoted by $\mathcal{E}^a_t$,  are the edges at time $t$ that connect an infected vertex and a susceptible vertex.
		\item  {\it Uninfected edges},  denoted by $\mathcal{E}^0_t$, connect two susceptible vertices. 
		\item {\it Inactive edges}, denoted by $\mathcal{E}^i_t$, have both ends infected. Once an edge becomes inactive it remains inactive forever.
	\end{itemize}
	The three sets form a partition of all edges.
	
	The three set-valued processes just defined are right-continuous pure jump processes.  At time 0 we randomly choose a vertex $u_0$ to be infected.  $\mathcal{E}^a_{0}$ consists of the edges with one endpoint at $u_0$.  $\mathcal{E}^0_{0}$ is the collection of all edges in the graph minus the set $\mathcal{E}^a_0$. $\mathcal{E}^i_{0}=\emptyset$. We will consider the corresponding sets for delSI, but they will be denoted by $\mathcal{D}_t$ to avoid confusion.
	
	For each undirected edge $e$, let $\tau^{e}_{e,\ell}$ be the $\ell$-th time the edge becomes active (the superscript `$e$' is short for evo). To make it easier to describe the dynamics, suppose that at time $\tau^{e}_{e,\ell}$ we have $e=\{x_{e,\ell},y_{e,\ell}\}$ with $x_{e,\ell}$ infected and $y_{e,\ell}$ susceptible.
	
	\begin{itemize}
		\item
		Let $T_{e,\ell}$, $\ell \ge 1$  be  the time between $\tau_{e,\ell}$  and the infection of $y_{e,\ell}$ by $x_{e,\ell}$.  
		
		\item
		Let  $R_{e,\ell}$, $\ell \ge 1$ be the time between $\tau_{e,\ell}$  and the time when $y_{e,\ell}$ breaks its connection to $x_{e,\ell}$ and rewires.  
		
		\item
		Let $U_{e,\ell}$ be the  vertex that $y_{e,\ell}$ connects to at time $\tau^{e}_{e,\ell} + R_{e,\ell}$ (if rewiring occurs).
	\end{itemize}

	\mn
	{\bf Initial step.} 
	To simplify writing formulas, let 
	\begin{equation}\label{defs}
		S_{e,\ell} = \min\{T_{e,\ell},R_{e,\ell}\}.
	\end{equation}
	At time 0, a randomly chosen vertex $u_0$ is infected. For a vertex $x$, let ${\cal N}^0(x,t),{\cal N}^a(x,t)$ and ${\cal N}^i(x,t)$ be the collection of uninfected, active and inactive edges  that are connected to $x$ at time $t$, respectively.  At time 0  the edges  ${\cal N}^0(u_0,0)=\{e_{0,1},\ldots, e_{0,k}\}$ are added to the list of active edges, where $k$ is the degree of $u_0$.  We have $\mathcal{E}^a_{0}=\{e_{0,1},\ldots, e_{0,k}\}$. Suppose $e_{0,j}$ connects $u_0$ and $y_j$. At time
	$$
	J^e_1 =  \min_{1\le j \le k} S_{e_{0,j},1}
	$$
	the first event occurs. ($J$ is for jump and $e$ stands for evo.)  Let $i$ be the index that achieves the minimum. 
	
	\mn
	(i) If $R_{e_{0,i},1}<T_{e_{0,i},1}$, then at time $J^e_1$ vertex $y_i$ breaks its connection with $u_0$ and rewires to $U_{e_{0,i},1}$.  If $U_{e_{0,i},1}$ is susceptible at time $J^e_1$, we move the edge $e_{0,i}$  to $\mathcal{E}^0_{J^e_1}$. On the initial step this will hold unless $U_{e_{0,i},1}=u_0$ in which case nothing has changed.
	
	\mn
	(ii) If $T_{e_i,1}<R_{e_i,1}$ then at time $J^e_1$ vertex $y_i$ becomes infected by $u_0$. We move  $e_{0,i}$ to $\mathcal{E}^i_{J^e_1}$. We add edges in ${\cal N}^0(y_i,J^e_1-)$ to ${\cal E}^a_{J^e_1}$.

	\mn
	{\bf Induction step.} 
	For any active edges $e$ present at time $t$, let $L^e(e,t) = \sup\{\ell : \tau^e_{e,\ell}\leq  t\}$
	and let $$V^e(e,t)=\tau^e_{e,L^e(e,t)}+S_{e,L^e(e,t)}$$ be the time of the next event (infection or rewiring) to affect edge $e$. Again, the superscripts `$e$' in $L^e(e,t)$ and $V^e(e,t)$ imply that these quantities are for the evoSI model. 
	Suppose we have constructed the process up to time $J^e_m$ for some $m\in \N$.
	If there are no active edges present at time $J^e_{m}$, the construction is done. Otherwise, let
	$$
	J^e_{m+1}=\min_{e\in \mathcal{E}^a_{J^e_m}} V^e(e,J^e_{m}).
	$$
	Let $e_m$ be the edge that achieves the minimum of $V^e(e,J^e_m)$,
	let $x(e_m)$ be the infected endpoint of $e_m$ and $y(e_m)$ be the susceptible endpoint of $e_m$. To simplify notation let $L_m =L^e(e_m,J^e_m)$.

	\mn
	(i) If $R_{e_m,L_m}<T_{e_m,L_m}$, then at time $J^e_{m+1}$ vertex $y(e_m)$ breaks its connection with $x(e_m)$ and rewires to $U_{e_m, L_m}$.   If $U_{e_n,L_n}$ is susceptible at time $J^e_{m+1}$, then $e_m$ is moved to  $\mathcal{E}^0_{J^e_{n+1}}$. Otherwise it remains active.
	
	\mn
	(ii) If $T_{e_n,L_n}<R_{e_n,L_n}$, then at time $J^e_{m+1}$ the vertex $y(e_n)$ is infected by $x(e_m)$ and $e_m$ is moved to  $\mathcal{E}^i_{J_{n+1}}$. Further,
	\begin{itemize}
		\item
		all edges $e'$ in ${\cal N}^0(y(e_m),J^e_{m+1}-)$ are moved to ${\cal E}^a_{J^e_{m+1}}$.  Since $y(e_m)$ has just become infected, the other end of $e'$ must be susceptible at time $J^e_{m+1}$.
		
		\item
		all edges $e''$ in ${\cal N}^a(y(e_m),J^e_{m+1}-)$ are moved to ${\cal E}^i_{J^e_{m+1}}$.  Since $y(e_m)$ has just become infected, (a) the other end of $e''$ must be infected at time $J^e_{m+1}$, and (b) $e''$ cannot have been inactive earlier.
	\end{itemize}
	
	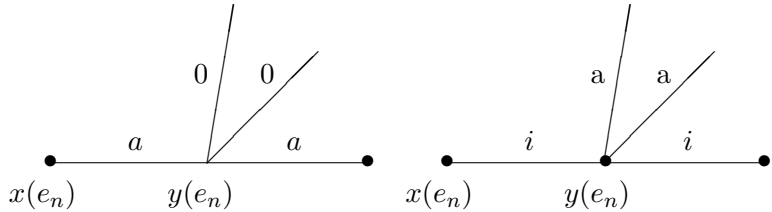
\begin{figure}[ht]
		\begin{center}
			\begin{picture}(310,100)
				\put(30,30){\line(1,0){60}}
				\put(90,30){\line(1,6){10}}
				\put(90,30){\line(1,1){42}}
				\put(90,30){\line(1,0){60}}
				\put(15,15){$x(e_m)$}
				\put(75,15){$y(e_m)$}
				\put(28,28){$\bullet$}
				\put(148,28){$\bullet$}
				\put(60,35){$a$}
				\put(120,35){$a$}
				\put(85,60){0}
				\put(110,60){0}
				\put(180,30){\line(1,0){60}}
				\put(240,30){\line(1,6){10}}
				\put(240,30){\line(1,1){42}}
				\put(240,30){\line(1,0){60}}
				\put(165,15){$x(e_m)$}
				\put(225,15){$y(e_m)$}
				\put(178,28){$\bullet$}
				\put(238,28){$\bullet$}
				\put(298,28){$\bullet$}
				\put(210,35){$i$}
				\put(270,35){$i$}
				\put(235,60){$a$}
				\put(260,60){$a$}
			\end{picture}
			\caption{Change of the status of different edges after $y(e_m)$ becomes infected by $x(e_m)$. Black dots mark infected sites. }
		\end{center}
	\end{figure}
	
	\medskip 
	{\bf Construction of delSI.} There is no rewiring in delSI, each edge will become active at most once. Thus for each undirected edge $e$ we only need two exponential random variables $T_{e,1}$ and $R_{e,1}$, defined in the beginning of Section \ref{sec:evodel}. This allows us to couple evoSI and delSI.
	Also, we use $\mathcal{D}^0_t, \mathcal{D}^a_t$ and $ \mathcal{D}^i_t$ to represent the set of uninfected, active and inactive edges in delSI, respectively.

	\mn
	{\bf Initial step.}  At time 0, a randomly chosen vertex $u_0$ is infected. The edges  ${\cal N}^0(u_0,0)=\{e_{0,1},\ldots, e_{0,k}\}$ are added to the list of active edges.   We have $\mathcal{D}^a_{0}=\{e_{0,1},\ldots, e_{0,k}\}$.  Suppose $e_j$ connects $u_0$ and $y_j$. At time
	$$
	J^d_1 = \min_{1\le j \le k} S_{e_{0,j},1}
	$$
	the first event occurs. The superscript `d' stands for delSI and  $S_{e,1}=\min\{T_{e,1}, R_{e,1}\}$, as defined in \eqref{defs}. Let $i$ be the index that achieves the minimum. 
	
	\mn
	(i) If $R_{e_{0,i},1}<T_{e_{0,i},1}$, then at time $J^a_1$ the edge $e_{0,i}$ is removed from the graph (and hence also from the set $\mathcal{D}^a_{J^d_1}$).
	
	\mn
	(ii) $T_{e_{0,i},1}<R_{e_{0,i},1}$ then at time $J^a_1$ vertex $y_i$ is infected by $x_i$. We move $e_{0,i}$   to  $\mathcal{D}^i_{J^d_1}$ . We add ${\cal N}^0(x_i,J^d_1-)$ to ${\cal D}^a_{J^d_1}$.

	\mn
	{\bf Induction step.} 
	For any active edge $e$ at time $t$, let  $\tau^d_{e,1}$ be the first time that $e$ becomes active in the delSI process. We also let
	$$
	V^d(e,t)=\tau^d_{e,1}+S_{e,1},
	$$ 
	be the time of the next event (infection or rewiring) to affect edge $e$.
	Suppose we have constructed the process up to time $J^d_m$.
	If there are no active edges present at time $J^d_{m}$, the construction is done. Otherwise, we let
	$$
	J^d_{m+1}=\min_{e\in \mathcal{D}^a_{J^d_m}} V(e,J^d_{m}).
	$$
	Let $e_m$ be the edge that achieves the minimum of $V^d(e,J^d_m)$, let $x(e_m)$ be the infected endpoint of $e_m$ and $y(e_m)$ be the susceptible endpoint of $e$.

	\mn
	(i) If $R_{e_m,1}<T_{e_m,1}$, then at time $J^d_{m+1}$ the edge $e_m$ is deleted
	from the graph (and hence also from the set $\mathcal{D}^a_{J^d_{m+1}}$).
	
	\mn
	(ii) If $T_{e_n,1}<R_{e_n,1}$, then at time $J^d_{m+1}$ the vertex $y(e_m)$ is infected by $x(e_m)$.  We move $e_m$ to $\mathcal{D}^i_{J^d_{m+1}}$. 
	
	\begin{itemize}
		\item
		We move all edges $e'$ in ${\cal N}^0_{y(e_m),J^d_{m+1}-}$ to ${\cal D}^a_{J^d_{m+1}}$.  Since $y(e_m)$ has just become infected, the other end of $e'$ must be susceptible at time $J^d_{m+1}$
		
		\item
		We move  all edges $e''$ in ${\cal N}^a_{y(e_m),J^d_{m+1}-}$  to ${\cal D}^i_{J^d_{m+1}}$.  
		Since $y(e_m)$ has just become infected, (i)  the other end of $e''$ must be infected at time $J^d_{m+1}$, and (ii) $e''$ cannot have been inactive earlier.
	\end{itemize}

	\medskip 
	We now prove by induction that avoSI dominates evoSI. 
	
	\begin{lemma}
		All vertices infected in delSI are also infected in evoSI and are infected earlier in avoSI than evoSI.
	\end{lemma}
	
	\begin{proof}
		The induction hypothesis holds for the first vertex since $u_0$ is infected at time 0 in both evoSI and delSI. Suppose the induction holds up to the $k$-st infected vertex in delSI. Assume at time $t$,  $y$ becomes the $(k+1)$-st infected vertex in delSI and $y$ is infected by vertex $x$ through edge $e$. We see from the construction of delSI that this implies $T_{e,1}<R_{e,1}$. Suppose $x$ was infected at time $s<t$ in delSI.
		The induction hypothesis implies that $x$ has also been infected in evoSI at a time $s' \le s$. There are two possible cases for $y$ in evoSI:
		\begin{itemize}
			\item $y$ has already become infected by time $s'+T_{e,1}$.
			\item $y$ was still susceptible right before $s'+T_{e,1}$. In this case, $y$ will be infected at time $s'+T_{e,1}\leq s+T_{e,1}=t.$
		\end{itemize}
		In either case $x$ has been infected by time $t$ in evoSI. This completes the induction step and thus proves Lemma \ref{couple}. 
	\end{proof}
	
	\subsection{The infected sites in delSI}
	
	In the introduction we have noted  that delSI is equivalent  to independent bond percolation. That is, we keep each edge independently with probability $\lambda/(\lambda+\rho)$ and find the component containing the initially infected vertex  (say, vertex 1). To compute the size of the delSI epidemic starting from vertex 1, we apply a standard algorithm, see e.g., \cite{ML}, for computing the size of the component containing 1 in the reduced graph in which edges have independently been deleted with probability $\rho/(\lambda+\rho)$. We call this the {\it exploration process of delSI.} At step 0 the active set ${\cal A}_0=\{1\}$, the unexplored set ${\cal U}_0 =\{2, \ldots n \}$, and the removed set ${\cal R}_0=\emptyset$. Here removed means these sites are no longer needed in the computation. In the SI model sites never enter the removed state, Let $\eta_{i,j} = \eta_{j,i}=1$ if there is an edge connecting $i$ and $j$ in the reduced graph. If $\eta_{i,j}=1$ an infection at $i$ is transmitted to $j$. At step $t$ if ${\cal A}_t \neq\emptyset$ we pick an $i_t \in {\cal A}_t$ and update the sets as follows. 
	\begin{align*}
		{\cal R}_{t+1} & =  {\cal R}_t \cup \{ i_t \}, \\
		{\cal A}_{t+1} & = ( {\cal A}_t - \{ i_t \}) \cup \{ y \in {\cal U}_t : \eta_{i_t,y} = 1 \}, \\
		{\cal U}_{t+1} & =  {\cal U}_t - \{ y \in {\cal U}_t : \eta_{i_t,y} = 1 \}.
	\end{align*}
	When ${\cal A}_t=\emptyset$ we have found the cluster containing 1 in the reduced graph, which will be the final set of infected sites in the SI model.
	The \emph{exploration process of the configuration model} can be similarly defined (just without deletion of edges) and we let $\mathcal{J}_t$ be the set of active sites at step $t$ in  this exploration. We set $R_t=\abs{\mathcal{R}_{t+1}}=t+1$,  $A_t=\abs{\mathcal{A}_{t+1}}$ and $J_t=\abs{\mathcal{J}_{t+1}}$. We make a time shift so that $A_t$ and $J_t$ can be coupled with two random walks with i.i.d. increments (see Lemma \ref{l:atst} below).

	We now study the exploration process of the configuration model itself as well as the delSI process on such graph. 
	Let $\psi_0$ have generating function $G$ defined in \eqref{F_2} and $\zeta_0$ have a  Binomial($D,\lambda/(\lambda+\rho)$) distribution whose generating function is denoted by $G^{\rho}$. The two generating functions are related by
	$$
	G^{\rho}(z)=G\left(\frac{\lambda}{\lambda+\rho}z+\frac{\rho}{\lambda+\rho}\right).
	$$ 
	Recall the definition of the generating function $\hat{G}$ in \eqref{G1f}. We can similarly define $\hat{G}^{\rho}$.
	Let $\{\chi_i,i\geq 1\}$ and   $\{\xi_i, i\geq 1\}$  be independent random variables with generating functions $\hat{G}$  and $\hat{G}^{\rho}$, respectively.
	Define two random walks (for integer-valued $t$):
	\beq
	W_0=\psi_0\quad W_t=W_0+\sum_{r=1}^t (\chi_r-1);
	\qquad S_0=\zeta_0 \quad S_t = S_0 + \sum_{r=1}^{t}  (\xi_r-1).
	\label{RW}
	\eeq
	Let $\tau^W_0=\inf\{t\geq 0:W_t= 0\}$ and set 
	$\bar W_t=W_{t\wedge \tau^W_0}$. We define $\bar S_t$ in a similar way. 
	
	\begin{lemma}\label{l:atst}
		We can couple  $\{J_t,0\leq t\leq n^{1/3}\log n\}$ and $\{\bar W_t,0\leq t\leq n^{1/3}\log n\}$ so that
		\begin{equation}\label{eq:jtwt}
			\lim_{n\to\infty}	\P\left((J_t)^{n^{1/3}\log n}_{t=0}=(\bar W_t)^{n^{1/3}\log n}_{t=0} \right)=1.
		\end{equation}
		Similarly, 	
		there exists a coupling of $\{A_t,0\leq t\leq n^{1/3}\log n\}$ and $\{\bar S_t,0\leq t\leq n^{1/3}\log n\}$ so that
		\begin{equation}\label{eq;atst}
			\lim_{n\to\infty}	\P\left((A_t)_{t=0}^{n^{1/3}\log n}=(\bar S_t)_{t=0}^{n^{1/3}\log n} \right)=1.
		\end{equation}
	\end{lemma}
	Note that $\bar W_t$ and $\bar S_t$ can be viewed as the exploration processes of two-phase branching process $Z_m$ and $\bar Z_m$ (both are defined in Section \ref{britton}), respectively. 
	The proof of Lemma \ref{l:atst} is deferred to the end of this section. 
	For the rest of this section we will always work on the event
	$$
	\{A_t=\bar S_t,0\leq t\leq n^{1/3}\log n \} \cap \{J_t=\bar W_t,0\leq t\leq n^{1/3}\log n \}
	$$ 
	and hence assume that $A_t$ and $J_t$ have independent increments until they hit 0.

	\subsection{Proof of Theorem  \ref{critical_value}(i)} 
	The formula for the critical value of delSI follows from standard results on percolation in random graphs.  Note that in delSI each edge is kept with probability $\lambda/(\lambda+\rho)$. Using \cite[Theorem 3.9]{Jperc} we see that 
	$$\lambda_c(\mbox{delSI})=\frac{\rho m_1}{m_2-2m_1}.$$ Equivalently, we have
	$$
	\alpha_c(\mbox{delSI})=\frac{\rho m_1}{\lambda_c(\mbox{delSI})}=m_2-2m_1. 
	$$
	Recall that Lemma \ref{couple} shows that  the final set of infected individuals in delSI is contained in the analogous set for evoSI with the same parameters so
	$$
	\lambda_c(\mbox{evoSI}) \le \lambda_c(\mbox{delSI}).
	$$ 
	To prove that the two are equal we will show that if $\lambda < \lambda_c(\textrm{delSI})$ then evoSI dies out, i.e., infects only a vanishing portion of the total population as $n\to\infty$.

	To compare the two evolutions, we will first run the delSI epidemic to completion. Once this is done we will randomly rewire the edges deleted in delSI. If the rewiring creates a new infection in evoSI, then we have to continue to run the process. If not, then the infected sites in the two processes are the same. 
	Let ${\cal R}$ be the set of sites that are eventually infected  in delSI, and let ${\cal R'}$ be the set of eventually infected sites in evoSI. Let $R' = |{\cal R}'|$ and $R = |{\cal R}|$. 
	
	To get started we use a result of Janson \cite{Jmax} about graphs with specified degree distributions. He works in the set-up introduced by Molloy-Reed \cite{MoRe1, MoRe2} where the degree sequence $d^n_i$, $1\le i \le n$ is specified and one assumes only that limiting moments exist as well  some other technical assumptions that are satisfied in our case (for a more recent example see \cite{JLW}):
	\begin{equation}\label{jasonthm}
		\frac{1}{n} \sum_{i=1}^n d^n_i \to \mu, \qquad \frac{1}{n} \sum_{i=1}^n d^n_i(d^n_i-1) \to \theta.
	\end{equation}
	The next result is Theorem 1.1 in \cite{Jmax}. The $\mu$ and $\theta$ in this theorem have the same meaning as  \eqref{jasonthm}. The ``whp'' below is short for with high probability, and means that the probability the inequality holds tends to 1 as $n\to\infty$. Let $D^n=d_{u_0}^n$ where $u_0$ is randomly chosen from $\{1,2,\ldots, n\}$. 
	
	\begin{theorem} \label{Janbd}
		Suppose $\mu>0$, $\theta>1$, and $\P( D^n \ge k) \le C k^{1-\gamma}$ for some $\gamma > 3$ and $C<\infty$. Then there is a constant $C_{\ref{maxbd}}$ 
		which depends on $C$ so that the largest component $\mathcal{C}_1$ has
		\beq
		|{\cal C}_1| \le C_{\ref{maxbd}} n^{1/(\gamma-1)} \qquad \hbox{whp.}
		\label{maxbd}
		\eeq
	\end{theorem}

	\noindent
	Theorem \ref{Janbd} can also be applied  to the setting where the degrees are random rather than deterministic. 
	We have assumed that in the original graph $\E D^5 < \infty$, so $\P( D \ge k ) \le k^{-3} \E (D^3)$ and we have
	$$
	\P(D^n\geq k)=\frac{1}{n}\sum_{i=1}^n \P(D_i\geq k)=\P(D_1\geq k)\leq  k^{-3} \E (D^3)\leq Ck^{-3}.
	$$
	It follows that
	\beq
	|{\cal C}_1| \le C_{\ref{maxbd}} n^{1/3}\quad\hbox{whp}.
	\label{C1bd}
	\eeq

	Let $N_d$ be the number of deleted edges in the exploration process of delSI. One vertex is removed from the construction on each step, so whp the number of steps is
	$\le r_0 := C_{\ref{maxbd}} n^{1/3}$. Note that $r_0\ll n^{1/3}\log n$ so  we can assume $J_t$ has independent increments by Lemma \ref{l:atst}. 
	It follows that
	$$
	N_d \leq r_0+\bar W_{r_0}\leq r_0+\psi_0+\sum_{r=1}^{r_0} \chi_r.
	$$
	where we recall that the $\chi_i$ are independent with the distribution $D^*-1$. The inequality comes from the fact that we are counting all the edges even if they are not deleted. 
	Since $\E(D^3)<\infty$,  we have
	$$\var(\chi_i) \le \E(\chi_i^2)=\E[(D^*-1)^2]\leq \E[(D^*)^2]=
	\frac{\E(D^3)}{\E(D)}
	<\infty$$ and $\var(\psi_0)\leq \E\psi_0^2=\E D^2<\infty$. Let $1/3 < a < 1/2$. Using Chebyshev's inequality
	\beq
	\P( N_d \ge r_0\E\chi_1 + \E\psi_0+r_0+n^{a} ) \le (r_0 \E\chi_1^2+\E \psi_0^2)/n^{2a}\leq Cn^{1/3-2a}.
	\label{rewires}
	\eeq
	A similar argument shows that
	\beq
	\P(A_{r_0} + R_{r_0} \ge r_0\E\xi_1 + \E \zeta_0+r_0+ n^{a} ) \le (r_0 \E\xi_1^2+\E \zeta_0^2)/n^{2a} \leq  Cn^{1/3-2a}.
	\label{infecteds}
	\eeq
	Since $a>1/3$, when $n$ is large we can upper bound $r_0\E\chi_1 + \E\psi_0+r_0+n^{a} $ and $ r_0\E\xi_1 + \E \zeta_0+r_0+ n^{a}$  by $2n^{a}$. 
	
	At step $r_0$ we use random variables independent of delSI to randomly rewire the deleted edges.
	Let $Y$ be the number of edges deleted up to time $r_0$ that rewire to the set $\mathcal{A}_{r_0+1}\cup \mathcal{R}_{r_0+1}$.  By construction
	$$
	Y = \text{Binomial}(N_d,(A_{r_0}+ R_{r_0})/n).
	$$
	Using \eqref{rewires}, \eqref{infecteds} and $r_0\le C_{\ref{maxbd}}n^{1/3}$ we see that on a set with probability $\ge 1 - C n^{-(2a-1/3)}$  
	$$
	Y \preceq  \text{Binomial}(2n^{a}, 2n^{a-1}) \equiv \bar Y,
	$$
	where $\equiv$ indicates that the last equality defines $\bar Y$.
	From this  we get 
	\beq
	\P(Y\ge 1 ) \leq \P(\bar Y\geq 1) \leq \E(\bar Y) =\frac{4 n^{2a}}{n} \to 0,
	\label{Y1bd}
	\eeq
	since $a<1/2$. 
	Since $\{ Y = 0\} \subset \{\mathcal{R} = \mathcal{R}'\}$, this shows $\P({\cal R} = {\cal R}') \to 1$ as $n\to\infty$ and completes the proof of (i). To prepare for the proof of (ii) note that the conclusion (the set of infected sites coincide in delSI and evoSI) holds as long as  the number of steps is smaller than $Cn^{1/3}$ even if the epidemic is supercritical.

	\subsection{Proof of Theorem \ref{critical_value}(ii)} 
	
	We need the following ingredient in the proof. 
	
	\begin{lemma} \label{upbd2}
		There is a $\gamma>0$ so that
		$$
		\P(0< A_{\log n} < \gamma \log n | A_0> 0) \le  \frac{C}{\log n}.
		$$
	\end{lemma}
	
	\begin{proof}[Proof of Lemma \ref{upbd2}]
		On the event $\{A_{\log n}>0\}$ we have $A_{\log n}=S_{\log n}$, which is, by definition,
		$$
		\zeta_0+\sum_{r=1}^{\log n}\xi_r.
		$$
		Hence if we take $\gamma=\E(\xi_1)/2$ then we have
		\begin{equation*}
			\begin{split}
				\P(0< A_{\log n} < \gamma \log n | A_0> 0) &\leq 
				\P\left( \sum_{r=1}^{\log n}\xi_r\leq \frac{ \E(\xi_1)\log n }{2} \right)\\
				&\leq  	\P\left( \abs{\sum_{r=1}^{\log n}\xi_r-
					\E(\xi_1)\log n} 
				\geq \frac{\E(\xi_1)\log n }{2} \right)\\
				&\leq \frac{4\var (\xi_1)\log n}{ (\E(\xi_1))^2 \log^2 n }\leq \frac{C}{\log n}.
			\end{split}
		\end{equation*}

	\end{proof}
	Let $B_d$ and $B_e$ be the events that there is a large epidemic in delSI and evoSI respectively.
	We now use Lemma \ref{upbd2} to show the difference between the probabilities of these two events vanishes asymptotically. 
	\begin{lemma}\label{upbd-3}
		Suppose $\lambda > \lambda_c(\mbox{delSI})$. As $n\to\infty$, $\P(B_e) - \P(B_d) \to 0.$
	\end{lemma}	
	
	\begin{proof}[Proof of Lemma \ref{upbd-3}] Clearly $\P(B_d) \leq \P(B_e)$. Let $S_t$ be the random walk defined in \eqref{RW}.
		To  prove Lemma \ref{upbd-3} we need the following lemma.
		Now we return to the proof of Lemma \ref{upbd-3}. 
		Let $F_0 = \{A_{\log n} = 0\}, F_1 = \{0 < A_{\log n}< \gamma \log n\}$ and
		$F_2 = \{A_{\log n} \geq \gamma \log n\}$. Decomposing $B_d$ into three parts
		and using Lemma \ref{upbd2}:
		$$
		\P(B_d) = \sum_{i=0}^2 \P(B_d \mid F_i) P(F_i). 
		$$
		We note that
		\begin{itemize}
			\item $\P(B_d|F_0)=0$ by the definition of $F_0$.
			\item $\P(F_1)$ converges to 0 by Lemma \ref{upbd2}. 
			\item $\P(B_d|F_2)$ converges to 1 by the same argument as the proof of Theorem 2.9(c) in \cite{JLW} (see page 750-752 in \cite{JLW}).
		\end{itemize}
		Therefore we have $\P(B_d)-\P(F_2)\to 0$. 
		As for $\P(B_e)$, we note that by the remark at the end of the proof of (i) one has $\P(\mathcal{A}_{\log n}=\mathcal{A}_{\log n}')\to 1$ where $\mathcal{A}'_{\log n}$ is the number of active sites in evoSI at step $\log n$. Using the decomposition $\P(B_e)=\sum_{i=0}^2 \P(B_e \cap F'_i)$
		where the event $F'_i$ is defined in a similar way to $F_i$ with $A_{\log n}$ replaced by $A'_{\log n}$,
		we see
		\begin{itemize}
			\item $\P(B_e \cap F_1')\leq \P(F_1')\leq \P(F_1)+o(1)=o(1)$.
			\item $\P(B_e \cap F_0')=o(1)$ by definition of $B_e$.
			\item $\P(B_e \cap F_2')=\P(F_2')+o(1)=\P(F_2)+o(1)$.
		\end{itemize}
		It follows that $\P(B_e)-\P(F_2)\to 0$. This implies that $\P(B_e)-\P(B_d)\to 0$.  It remains to compute the limit of $\P(F_2)$. Since $\P(F_1)=o(1)$, we have 
		\begin{equation}\label{eq;f1f2}
			\P(F_2)=\P(F_2\cup F_1)+o(1)=\P(A_{\log n}>0)+o(1).
		\end{equation}
		This completes the proof of Lemma \ref{upbd-3}. 
	\end{proof}
	To compute the limit of  $\P(A_{\log n}>0)$,
	note that due to Lemma \ref{l:atst}, $A_t$ can be coupled with the exploration process of the two-phase branching process $\bar Z_m$ defined in Section \ref{britton}. Therefore we have (recall that $q(\lambda)$ is the survival probability of $\bar Z_m$) $$\P(A_{\log n}>0)=\P(\bar Z_m>0, \forall m)+o(1)=q(\lambda)+o(1).$$ By \eqref{eq;f1f2} we see that $\P(F_2)\to q(\lambda)$ as well. This implies that both $\P(B_e)$ and $\P(B_d)$ converge to $q(\lambda)$ as $n\to\infty$ and completes the proof of Theorem \ref{critical_value}(ii). Note that using the fact that delSI is equivalent to independent bond percolation, the statement $\P(B_d)\to q(\lambda)$ also follows from standard results on percolation in random graphs. See, e.g., 
	\cite[Theorem 3.9]{Jperc}.
	It remains to prove Lemma \ref{l:atst} to complete the proof of Theorem \ref{critical_value}. 
	\begin{proof}[Proof of Lemma \ref{l:atst}]
		We only prove equation \eqref{eq:jtwt} since the other one follows from \eqref{eq:jtwt} and  the fact that  delSI is equivalent to percolation with edge retaining probability $\lambda/(\lambda+\rho)$. 
		The proof consists of two steps. First, we define an empirical version of $W_t$. Let $D_1, \ldots, D_n$ be i.i.d. random variables sampled from the distribution of $D$. Given a sample of $D_1, \ldots, D_n$, let $\psi^n_0$ be sampled from the (random) distribution 
		$$
		\P_n(\psi^n_0=k)=\frac{1}{n}\abs{\{1\leq i\leq n: D_i=k\}},\, \forall  k\geq 0,
		$$
		which is the sample empirical distribution. Let $\chi^n_r,r\geq 1$ have the distribution
		$$
		\P_n(\chi^n_r=k)=\frac{1}{D_1+\ldots+D_n}(k+1)\abs{\{1\leq i\leq n: D_i=k+1\}}\, \forall k\geq 0.
		$$
		Define $$
		W^n_t=\psi^n_0+\sum_{r=1}^t (\chi^n_r-1),\,t\in \N.
		$$
		In other words, $W^n_t$ is a random walk in the random environment given by $D_1,\ldots, D_n$. 
		We define $\bar W^n_t$ to be $W^n_{t\wedge \tau'}$ where $\tau'$ is the first time $W^n_t$ hits zero. 
		Note that the condition $\E(D^5)<\infty$ implies that $\max_{1\leq i\leq n}D_i=o(n^{1/4}\log n)$ whp for any $\ep>0$. Indeed we have
		\begin{equation}\label{eq:dmax}
			\P\left(\max_{1\leq i\leq n}D_i>n^{1/4}\log n\right)\leq n \P(D_1>n^{1/4}\log n)\leq n \cdot n^{-1}\log^{-4} n\E(D^4)\leq C \log^{-4}n.
		\end{equation}
		Consequently, 	for  $n$ large enough, whp  
		$$
		\sqrt{\frac{n}{\max_{1\leq i\leq n}D_i}}> n^{3/8}/\log n.
		$$
		Using Lemma 2.12 in \cite{vdH3}, we see that (recall that $J_t$ is the exploration process of the configuration model starting from a uniformly randomly chosen vertex)
		\begin{itemize}
			\item 	$J_t$ can be coupled with $\bar W^n_t$   with high probability up to time $ n^{3/8}/\log n$.
			\item Whp the subgraph obtained by exploring the neighborhoods of $n^{3/8}/\log n$ vertices is a tree.
		\end{itemize}
		Using $n^{1/3}\log n \ll n^{3/8}/\log n$,
		\begin{equation}
			\lim_{n\to\infty}	\P\left(J_t=\bar W^n_t,0\leq t\leq n^{1/3}\log n \right)=1.
		\end{equation}
		To prove \eqref{eq:jtwt} it remains to show that one can couple $W^n_t$
		and $W_t$ up to step $n^{1/3}\log n$. To this end, we use the characterization of total variation distance in terms of optimal coupling. It is well known that for any two random variables $X$ and $Y$, 
		$$
		d_{\mathrm{TV}}(X,Y)=\inf_{\textrm{all couplings of }X,Y} \P(X\neq Y).
		$$
		See \cite[Theorem 1.14]{VC} for instance. 
		Using ${\bf D_n}$ to denote the degree sequence $D_1,\ldots, D_n$,  it suffices to show that
		\begin{equation}\label{eq;tv2}
			d^{\BD}_{\mathrm{TV}}((\psi^n_0,\{\chi^n_r,1\leq r\leq n^{1/3}\log n\}),(\psi_0, \{\chi_r,1\leq r\leq n^{1/3}\log n\})) \CP 0.
		\end{equation}
		Here the superscript ${\bf D_n}$ indicates that we are considering the quenched law of $\psi^n_0$ and $\xi^n_0$. 
		Since conditionally on $\BD$, $\psi^n_0$ and $\chi^n_r,r\geq 1$ are all independent, 
		\begin{equation}\label{eq:tv}
			\begin{split}
				&		d^{\BD}_{\mathrm{TV}}((\psi^n_0,\{\chi^n_r,1\leq r\leq n^{1/3}\log n\}),(\psi_0, \{\chi_r,1\leq r\leq n^{1/3}\log n\}))\\
				\leq 	&	 	d^{\BD}_{\mathrm{TV}}(\psi^n_0, \psi_0)+
				\sum_{r=1}^{n^{1/3}\log n}	d^{\BD}_{\mathrm{TV}}(\chi^n_r,\chi_r)\\
				\leq &  	d^{\BD}_{\mathrm{TV}}(\psi^n_0, \psi_0)+(\log n)n^{1/3}
				d^{\BD}_{\mathrm{TV}}(\chi^n_1,\chi_1).
			\end{split}
		\end{equation}
		We need the following lemma to control $\BD$. Recall that we let $p_k=\P(D=k)$ and $m_1=\E(D)$. Also recall that we assume $\E(D^5)<\infty$. 
		\begin{lemma}\label{l;d_n}
			For any $\ep>0$, 
			\begin{equation}\label{eq:sumd}
				\P\left(\abs{\sum_{i=1}^n D_i-nm_1}>n^{1/2+\ep}\right)\leq Cn^{-2\ep}.
			\end{equation}
			Let $N_k$ be the cardinality of the set $\{i:D_i=k\}$. We have
			\begin{equation}\label{eq;n_k}
				\P\left(\abs{N_k-np_k}>n^{\ep}(np_k)^{1/2}  \right)\leq n^{-2\ep}.
			\end{equation}
		\end{lemma}
		\begin{proof}[Proof of Lemma \ref{l;d_n}]
			Both \eqref{eq:sumd} and \eqref{eq;n_k} follow from Markov's inequality. For \eqref{eq:sumd} we note that
			$$
			\P\left(\abs{\sum_{i=1}^n D_i-nm_1}>n^{1/2+\ep}\right)\leq \frac{n\mathrm{Var}(D_1)}{n^{1+2\ep}}\leq Cn^{-2\ep}. 
			$$
			For the second inequality, we note that  $\E(1_{\{D_1=k\}}-p_k)^2=p_k(1-p_k)\leq p_k$. 
			It follows that
			$$
			\E(N_k-np_k)^2=\E\left(\sum_{i=1}^n \left(1_{\{D_i=k\}}-p_k\right)\right)^2= n	\E(1_{\{D_1=k\}}-p_k)^2\leq np_k. 
			$$
			Therefore
			\begin{equation}
				\P\left(\abs{N_k-np_k}>n^{\ep}(np_k)^{1/2} \right)\leq 
				\frac{ \E(N_k-np_k)^2}{n^{2\ep}(np_k)}\leq n^{-2\ep}.
			\end{equation}
		\end{proof}
		Let $\ep_1=1/100,\ep_2=1/8+1/100$.
		Consider the event
		\begin{equation}
			\begin{split}
				\Omega_n=&\left\{\max_{1\leq i\leq n}D_i\leq n^{1/4}\log n, 
				\abs{\sum_{i=1}^n D_i-nm_1}\leq n^{1/2+\ep_1}\right\}\\
				&
				\cap 	\left\{ \abs{N_k-np_k}\leq n^{\ep_2}
				(np_k)^{1/2} \mbox{ for }1\leq k\leq  n^{1/4}\log n
				\right\}.
			\end{split}
		\end{equation}
		Lemma \ref{l;d_n} and equation \eqref{eq:dmax} imply that $\P(\Omega_n)\to 1$ by the union bound.
		We now control $d^{\BD}_{\mathrm{TV}}(\chi^n_1,\chi_1)$ on $\Omega_n$.
		For $k\geq n^{1/4}\log n$, we have $\P_n(\chi^n_1=k)=0$ on $\Omega_n$ almost surely. We also have
		$\sum_{k\geq n^{1/4}\log n}\P(\chi_1=k)\leq n^{-1} \E(D^4)\leq Cn^{-1}$.
		Hence we have
		\begin{equation}\label{eq:largek}
			\sum_{k\geq n^{1/4}\log n}\abs{\P_n(\chi^n_1=k)-\P(\chi_1=k)}\leq Cn^{-1}.
		\end{equation}
		For $0\leq k\leq n^{1/4}\log n$, using the definition of $\Omega_n$ we get, for $n$ large, 
		\begin{equation}\label{chi1}
			\begin{split}
				&		\abs{\P_n(\chi^n_1=k)-\P(\chi_1=k)}=
				\abs{\frac{(k+1)N_{k+1}}{\sum_{i=1}^n D_i}-\frac{(k+1)p_{k+1}}{m_1}}	\\
				\leq&\abs{\frac{(k+1)N_{k+1}}{\sum_{i=1}^nD_i}-
					\frac{n(k+1)p_{k+1}}{\sum_{i=1}^n D_i}	
				}+\abs{
					\frac{n(k+1)p_{k+1}}{\sum_{i=1}^nD_i}	- \frac{(k+1)p_{k+1}}{m_1}}\\
				\leq &  \frac{(k+1)n^{\ep_2}(np_k)^{1/2}  }{m_1n/2}+
				\frac{(k+1)p_{k+1}\abs{nm_1-\sum_{i=1}^n D_i}}{nm_1 \sum_{i=1}^n D_i}\\
				\leq& Cn^{\ep_2}n^{-1/2}(k+1)^{-1}  +C(k+1)p_{k+1}n^{-1/2+\ep_1},
			\end{split}
		\end{equation}
		where we have used the fact that $p_k\leq Ck^{-4}$ since $\E(D^4)<\infty$ in the last step. 
		Summing the last line of \eqref{chi1} over $k$ from 0 to $n^{1/4}\log n$ and using the facts
		$$\sum_{k=1}^{n} k^{-1}\sim \log n, \quad \sum_{k=0}^{\infty} (k+1)p_{k+1}=\E(D)< \infty,$$ we get, on $\Omega_n$, 
		\begin{equation}
			\sum_{k=0}^{n^{1/4}\log n} 	\abs{\P_n(\chi^n_1=k)-\P(\chi_1=k)} \leq Cn^{\ep_2} n^{-1/2}\log n+Cn^{-1/2+\ep_1}. 
		\end{equation}
		Inserting the values of $\ep_1=1/100,\ep_2=1/8+1/100$,
		\begin{equation}\label{eq:smallk}
			\sum_{k=0}^{n^{1/4}\log n} 	\abs{\P_n(\chi^n_1=k)-\P(\chi_1=k)} \leq Cn^{1/100}\log n(n^{1/8-1/2}+n^{-1/2}).
		\end{equation}
		Combining \eqref{eq:largek} and \eqref{eq:smallk}, 
		\begin{equation}\label{eq:xi}
			d^{\BD}_{\mathrm{TV}}(\chi^n_1,\chi_1)
			=\frac{1}{2}\sum_{k\geq 0} 	\abs{\P_n(\chi^n_1=k)-\P(\chi_1=k)}
			\leq Cn^{-73/200}\log n.
		\end{equation}
		One can similarly show that $Cn^{-73/200}\log n$ also serves as an upper bound for $	d^{\BD}_{\mathrm{TV}}(\psi^n_0,\psi_0)$.
		Thus  on $\Omega_n$
		$$
		d^{\BD}_{\mathrm{TV}}(\psi^n_0, \psi_0)+(\log n)n^{1/3}
		d^{\BD}_{\mathrm{TV}}(\chi^n_1,\chi_1) \leq Cn^{-73/200+1/3} \log^2 n,
		$$
		which converges to 0 as $n\to\infty$. This together with \eqref{eq:tv} implies \eqref{eq;tv2} and thus completes the proof of Lemma \ref{l:atst}. 
		
	\end{proof}
	\clearp
	
	\section{Upper bound on evoSI}\label{sec:ubsi}

	\subsection{avoSI} \label{sec:avoSI}

	As mentioned in the introduction, in this we will construct a model that serves as an upper bound for evoSI. We first introduce a model \emph{C-evoSI} where `C' stands for `coupled' which means we couple the structure of the graph with the epidemic. The `avo' stands for avoiding infection.
	
	The C-avoSI process is constructed as follows. First recall that in the construction of the configuration model each vertex is assigned a random number of half-edges initially. In the beginning all half-edges attached to the $n$ vertices are unpaired. The half-edges attached to infected nodes are called infected half-edges and those attached to susceptible nodes are susceptible half-edges. Recall that``randomly chosen'' and ``at random''  mean that the distribution of the choice is uniform over the set of possibilities.
	\begin{itemize}
		\item
		At rate $\lambda$ each infected half-edge pairs with a randomly chosen half-edge in the pool of all half-edges excluding itself. If the vertex $y$ associated with that half-edge is susceptible then it becomes infected. Note that if vertex $y$ changes from state $S$ to $I$ then all half-edges attached to $y$ become infected half-edges.
		\item 
		Each infected  half-edge gets removed from the vertex that it is attached to at rate $\rho$ and immediately becomes re-attached to a randomly chosen vertex in the pool of all vertices. 
	\end{itemize}
	
	For the purpose of comparisons it   is convenient to give a reformulation of C-avoSI where the graph has been constructed before the epidemic. We call this process \emph{avoSI}. To describe this process we define the notion of `stable edge'.
	We say that an edge between two vertices $x$ and $y$ is \emph{stable}  if one of the following conditions hold:
	\begin{itemize}
		\item Both $x$ and $y$ are in state $S$.
		\item Either $x$ or $y$ has sent an infection to the other one through this edge. 
	\end{itemize}
	We say that an edge is \emph{unstable} if both conditions fail. Note that an $S-I$ pair is necessarily unstable since the vertex in state $I$ has not sent an infection to the vertex in state $S$. We define the avoSI process as follows: 
	\begin{itemize}
		\item Each infected vertex sends infections to its neighbors at rate $\lambda$. If the neighbor has already been infected then nothing changes. Once a vertex receives an infection, it  stays infected forever. 
		\item A vertex in state $S$ will  break its connection with a vertex in state $I$ at rate $\rho$ and rewire to another randomly chosen vertex. The events for different $S-I$ connections are independent.
		\item For every  unstable $I-I$ edge, each of the two $I's$ will rewire at rate $\rho$ to another uniformly chosen vertex. (This also explains why we call such edges `unstable', since they may evolve.)
	\end{itemize}
	
	\begin{lemma}\label{c-avosi=avo}
		The C-avoSI process and avoSI process running on a configuration model have the same law in terms of the evolution of the set of infected vertices. 
	\end{lemma}
	
	\begin{proof}[Proof of Lemma \ref{c-avosi=avo}]
		It suffices to construct a graph $G$ that has the law of the configuration model such that the set of infected vertices in C-avoSI evolves in the same way as that of the avoSI with initial underlying graph being $G$. Given an outcome of C-avoSI, the graph $G$ can be constructed as follows. Let $H$ be the collection of vertices and half-edges.
		
		\begin{itemize}
			\item
			We assign a unique label to each half-edge in C-avoSI and correspondingly label the half-edges in $H$.
			\item
			Whenever two half-edges combine into one edge in C-evoSI process we pair the two half-edges with the same labels in $H$. It is clear that the pairings of half-edges are done at random.
			We pair the remaining half-edges at random 
			after there is no infected half-edge in the system. This forms the graph $G$. Since the pairings of all half-edges are at random,  we deduce that $G$ itself has the law of $\CM(n,D)$. 
			\item
			Whenever a pairing occurs in C-avoSI, an infected vertex has sent an infection to one of its neighbors(s) in avoSI. 
			
			\item
			Whenever an infected half-edge $h$ attached to $x$ rewires to another vertex $y$ in the C-avoSI, the corresponding edge $e$ in avoSI, which contains $h$ as one of its two half-edges, breaks from $x$ and reconnects to vertex $y$. 
		\end{itemize}
		
		The process we construct on $G$ has the same law as avoSI. Indeed, an infected half-edge can be rewired in C-avoSI if and only if it has not been paired. This exactly corresponds to the notion of `unstable' that we used in the construction of avoSI process. 
		Hence Lemma \ref{c-avosi=avo} follows. 
	\end{proof}
	
	From now on we will not distinguish between the C-avoSI and the avoSI since they are equivalent. The avoSI process stochastically dominates the evoSI process, as shown in the lemma below. 
	
	\begin{lemma}\label{avosi>evosi}
		The final size of infected vertices in avoSI stochastically dominates that in evoSI.
	\end{lemma}

	We couple the evoSI and avoSI as follows. The evoSI is constructed in the same manner as we did in the comparison of evoSI and delSI. See Section \ref{sec:evodel}.  We will use the variables $T_{e,\ell}, R_{e,\ell}, U_{e,\ell}, \ell\geq 1$  (defined at the beginning of Section \ref{sec:evodel}) in the construction of avoSI to couple evoSI and avoSI. 
	
	\medskip
	
	{\bf Construction of avoSI.}
	For avoSI we need four sets to partition the set of all edges. We use
	$\mathcal{A}$ to avoid confusion with the sets used in the construction of evoSI (not to be confused with the $\mathcal{A}_t$ used in Section \ref{sec:pfth2}).
	\begin{itemize}
		\item {\it  Active edges with one end infected} at time $t$,  denoted by $\mathcal{A}^{a,1}_t$,  are the edges at time $t$ that connect an infected vertex and a susceptible vertex. 
		\item {\it  Active edges with both ends infected},  denoted by $\mathcal{A}^{a,2}_t$,  are the unstable edges at time $t$ that connect two infected vertices.
		\item 
		{\it Uninfected edges},  denoted by $\mathcal{A}^0_t$, connect two susceptible vertices. 
		\item  {\it Inactive edges}, denoted by $\mathcal{A}^i_t$, consist of stable infected edges. Once an edge becomes inactive it remains inactive forever.  
	\end{itemize}
	
	We set  $\mathcal{A}_t=  \mathcal{A}^{a,1}_{t} \cup  \mathcal{A}^{a,2}_{t} $.
	The four sets form a partition of all edges. They  are  right-continuous pure jump processes.  At time 0 we randomly choose a vertex $u_0$ to be infected.  $\mathcal{A}^{a,1}_{0}$ consists of the edges with one endpoint at $u_0$.  $\mathcal{A}^0_{0}$ is the collection of all edges in the graph minus the set $\mathcal{A}^{a,1}_{0}$. $\mathcal{A}^i_{0} = \mathcal{A}^{a,2}_{0} =\emptyset$. 
	
	For each undirected edge $e$, we let random variables $T_{e,\ell},R_{e,\ell},U_{e,\ell}$ be the same as those used in the construction of evoSI. We set $S_{e,\ell}=\min\{T_{e,\ell},R_{e,\ell}\}$.  We also let $V'_{e,\ell}$ be independent uniform random variables that take values in the two endpoints of $e$.  
	To make it easier to describe the dynamics, suppose that at time $\tau^a_{e,\ell}$ (the $\ell$-th time $e$ becomes active in avoSI) we have $e=\{x_{e,\ell},y_{e,\ell}\}$ with $x_{e,\ell}$ infected and $y_{e,\ell}$ susceptible.

	The difference between evoSI and avoSI in terms of these clocks is as follows. For any edge $e$ connecting $x$ and $y$, once one endpoint $x_{e,\ell}$ becomes infected, the clocks $T_{e,\ell}$ and $R_{e,\ell}$ start running. If the other endpoint $y_{e,\ell}$ also becomes infected through other edges 
	at (relative) time $w_{e,\ell} < S_{e,\ell}$ (here `relative' means we only count the time after the infection of $x$), then we replace the clocks $T_{e,\ell}$ and $R_{e,\ell}$ by  
	$$
	T'_{e,\ell}=\frac{T_{e,\ell}-w_{e,\ell}}{2} \mbox{\, and \,} R'_{e,\ell}=\frac{R_{e,\ell}-w_{e,\ell}}{2}
	$$
	since the rates are now twice as fast. 
	
	\mn
	To see that this construction give the correct dynamics of avoSI. We note that conditionally on $T_{e,\ell},R_{e,\ell}>w_{e,\ell}$, $(T_{e,\ell}-w_{e,\ell})/2$ and $(R_{e,\ell}-w_{e,\ell})/2$ are independent exponential random variables with parameters $2\lambda$ and $2\rho$, respectively. 
	This corresponds to an unstable $I-I$ pair where each of the two $I's$ attempt to send infection to the other $I$ at rate $\lambda
	$ and rewire from the other $I$ at rate $\rho$. The variable $V'_{e,\ell}$ corresponds to the vertex with small rewiring time.

	\mn
	If $T_{e,\ell}>R_{e,\ell}$ then $T'_{e,\ell}>R'_{e,\ell}$ and vice versa. 
	Using this and the fact that the uniform variable $U_{e,\ell}$ is the same in evoSI and avoSI we get the following proposition:
	\begin{prop}\label{prop:avoevo}
		If edge $e$ is rewired  in evoSI, then $e$ must be rewired  to the same vertex in avoSI (as long as either endpoint of $e$ becomes infected in avoSI). 
		Also, the time it take to $e$ to be rewired in evoSI is not smaller than the time in avoSI.
	\end{prop}
	
	We now formally construct the avoSI process by induction.
	
	\mn
	{\bf Initial step.} 
	At time 0, a randomly chosen vertex $u_0$ is infected. Let ${\cal N}^0(x,t),{\cal N}^{a,1}(x,t), \mathcal{N}^{a,2}(x,t)$ and ${\cal N}^i(x,t)$ be the sets of edges connected to $x$ that 
	belong to the sets $\mathcal{A}^0_t, \mathcal{A}^{a,1}_t, \mathcal{A}^{a,1}_t$ and $\mathcal{A}^i_t$, respectively.  At time 0  all edges in $\{e_{0,1},\ldots, e_{0,k}\}$ are added to the list of active edges so that $\mathcal{A}^{a,1}_{0}=\{e_{0,1},\ldots, e_{0,k}\}$. Suppose $e_j$ connects $u_0$ and $y_j$. At time
	$$
	J^a_1 =  \min_{1\le j \le k} S_{e_{0,j},1}
	$$
	the first event occurs. (The superscript `a' indicates the avoSI model.)  Let $i$ be the index that achieves the minimum. 
	
	\mn
	(i) If $R_{e_{0,i},1}<T_{e_{0,i},1}$, then at time $J^a_1$ vertex $y_i$ breaks its connection with $u_0$ and rewires to $U_{e_{0,i},1}$.  If $U_{e_{0,i},1}$ is susceptible at time $J^a_1$, we move the edge $e_{0,i}$  to $\mathcal{A}^0_{J^a_1}$. On the initial step this will hold unless $U_{e_{0,i},1}=u_0$ in which case nothing has changed.
	
	\mn
	(ii) If $T_{e_{0,i},1}<R_{e_{0,i},1}$ then at time $J^a_1$ vertex $y_i$ becomes infected by $u_0$. We move  $e_{0,i}$ to $\mathcal{A}^i_{J^a_1}$. We move edges in ${\cal N}^0(y_i,J^a_1-)$ to ${\cal A}^a_{J^a_1}$. 
	
	\mn
	{\bf Induction step.}
	For any active edge $e$ at time $t$, let $L^a(e,t) = \sup\{\ell : \tau^a_{e,\ell}\leq  t\}$
	and set $$V^a(e,t)=
	\begin{cases} \tau_{e,L^a(e,t)}+S_{e,L^a(e,t)}
		&\hbox{ if $e \in \mathcal{A}_t^{a,1}$,
		}\\
		\frac{\tau^a_{e,L^a(e,t)}+I(x^*,y^*)+S_{e,L^a(e,t)}}{2}
		&\hbox{ if $e\in \mathcal{A}^{a,2}_t$,}
	\end{cases}
	$$
	where $I(x^*,y^*)$ is the first time that both $x^*=x(e,L(e,t))$ and
	$y^*=y(e,L^a(e,t))$ become active vertices and the second line comes from the computation
	$$
	\tau^a_{e,L(e,t)}+(I(x^*,y^*) -	\tau^a_{e,L^a(e,t)} )+\frac{S_{e,L^a(e,t)} 
		-(I(x^*,y^*) -	\tau^a_{e,L^a(e,t)} )}{2}=\frac{\tau^a_{e,L^a(e,t)}+I(x^*,y^*)+S_{e,L^a(e,t)}}{2}.
	$$
	Then $V^a(e,t)$ is
	the time of the next event (infection or rewiring) to affect edge $e$. 
	Suppose we have constructed the process up to time $J_m$.
	If there are no active edges present at time $J_{m}$, the construction is done. Otherwise, we
	let
	$$
	J^a_{m+1}=\min_{e\in \mathcal{A}^a_{J_m}} V^a(e,J^a_{m}).
	$$
	Let $e_m$ be the edge that achieves the minimum of $V^a(e,J^a_m)$. If $e_m$ only has one endpoint infected at time $J^a_m$ then we
	let $x(e_m)$ be the infected endpoint of $e_m$ and $y(e_m)$ be the susceptible endpoint of $e_m$. To simplify notation let $L_m =L(e_m,J^a_m)$.

	\mn
	(i) If $R_{e_m,L_m}<T_{e_m,L_m}$ and $e_m\in \mathcal{A}^{a,1}_{J^a_m}$, then at time $J^a_{m+1}$ vertex $y(e_m)$ breaks its connection with $x(e_m)$ and rewires to $U_{e_m, L_m}$.   If $U_{e_m,L_m}$ is susceptible at time $J^a_{m+1}$, then $e_m$ is moved to  $\mathcal{A}^0_{J^a_{m+1}}$. Otherwise it remains in $\mathcal{A}^{a,1}_{J^a_{m+1}}$.

	\mn
	(ii) If $R_{e_m,L_m}<T_{e_m,L_m}$ and $e_m\in \mathcal{A}^{a,2}_{J^a_m}$ , then at time $J^a_{m+1}$ vertex $	V'_{e_m,L_m}$ breaks its connection with the other end of $e_m$ and rewires to $U_{e_m, L_m}$.   If $U_{e_m,L_m}$ is susceptible at time $J^a_{m+1}$, then $e_m$ is moved to  $\mathcal{A}^{a,1}_{J^a_{m+1}}$. Otherwise $e_m$ stays in the set $\mathcal{A}^{a,2}_{J^a_{m+1}}$
	
	\mn
	(iii) If $T_{e_m,L_m}<R_{e_m,L_m}$ and  $e_m \in \mathcal{A}^{a,1}_{J^a_m}$, then at time $J^a_{m+1}$ the vertex $y(e_m)$ is infected by $x(e_m)$ and $e_m$ is moved to $\mathcal{A}^i_{J^a_{m+1}}$. 
	
	\begin{itemize}
		\item
		All edges $e'$ in ${\cal N}^0_{y(e_m),J^a_{m+1}-}$ are moved to ${\cal A}^{a,1}_{J^a_{m+1}}$.  Since $y(e_m)$ has just become infected, the other end of $e'$ must be susceptible at time $J^a_{m+1}$.
		
		\item
		All edges $e''$ in ${\cal N}^{a,1}_{y(e_m),J^a_{m+1}-}$ are moved to ${\cal A}^{a,2}_{J^a_{m+1}}$.  Since $y(e_m)$ has just become infected, (i) the other end of $e''$ must be infected at time $J^a_{m+1}$, and (ii) $e''$ cannot have been inactive earlier.
	\end{itemize}

	\mn
	(iv) If $T_{e_m,L_m}<R_{e_m,L_m}$ and  $e_m \in \mathcal{A}^{a,2}_{J^a_m}$, then at time $J^a_{m+1}$,   $e_m$ is moved to $\mathcal{A}^i_{J^a_{m+1}}$. There are no changes for other edges.

	The avoSI process stops when there are no active edges.

	\begin{proof}[Proof of Lemma \ref{avosi>evosi}]
		We now prove by induction that all vertices infected in evoSI are also infected in avoSI and actually they are infected earlier in avoSI than evoSI.
		
		\mn
		The induction hypothesis holds for the first vertex since initially $u_0$ is infected in both evoSI and avoSI. Suppose the induction holds up to the $k$-th infected vertex in evoSI. Assume at time $t$,  $y$ becomes the $(k+1)$-th infected vertex in evoSI.
		We assume that $y$ is infected by vertex $x$ through edge $e$. Note that $e$ has possibly gone through a series of rewirings before connecting vertex $y$. We assume that $e$ connects vertices $x_{\ell}$ and $y_{\ell}$ after the $(\ell-1)$-th rewiring. 
		We also assume when $x$ infects $y$ through $e$, $e$ has been  rewired $r$ times. This implies that
		\begin{equation}\label{tell}
			T_{e,\ell}>R_{e,\ell} \mbox{ for } 1\leq \ell \leq r \mbox{ and }T_{e,r+1}<R_{e,r+1}.
		\end{equation}
		We let $m(x)=\inf\{i: x\in \{x_k,y_k\} \hbox{ for all }i \le k \le r+1 \}$. 
		We now divide the analysis into two cases: $m=1$
		and $1<m\leq r+1$. 
		
		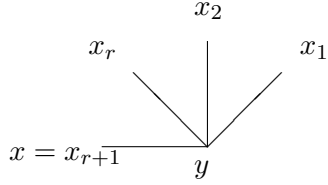
\begin{figure}[h]
			\begin{center}
				\begin{picture}(140,100)
					\put(80,40){\line(1,1){28}}
					\put(80,40){\line(0,1){40}}
					\put(80,40){\line(-1,1){28}}
					\put(80,40){\line(-1,0){40}}
					\put(75,30){$y$}
					\put(115,75){$x_1$}
					\put(35,75){$x_{r}$}
					\put(75,90){$x_2$}
					\put(5,35){$x=x_{r+1}$}
				\end{picture}
				\caption{Illustration of the case $m=1$. If $y$ is infected before time $t$ in avoSI then the induction step holds true. 
					If not, then since $x_1, \ldots, x_{r+1}$ are the rewirings in evoSI,  they are also infected in avoSI and the rewiring occurs as does the infection of $x$ by $x_{r+1}$.}
				\label{fig:case1}
			\end{center}
		\end{figure}
		
		\mn
		{\bf Case 1.} $m=1$ so that $y=y_k$ for all $1\leq k\leq r+1$. If we assume that $y$ has not been infected by time $t$,  then $x_1, \ldots x_{r}$ must have been infected at the time that the rewiring occurred
		and  $x_{r+1}=x$ infected $y$ at time $t$ in evoSI. By the induction hypothesis we see that $x_1,\ldots, x_{r+1}$ are also infected in avoSI. 	By Proposition \ref{prop:avoevo} and \eqref{tell}, $e$  breaks its connection with $x_1$ and reconnects to $x_2$, then to $x_3$ and after $r$ rewirings to $x_{r+1}$. If $y$ is already infected before $x_{r+1}$ sends an infection to it than we are done. Otherwise since $T_{e,\ell+1}<R_{e,\ell+1}$ we see that $x_{r+1}$ will send an infection to $y$ in avoSI as well. In any case we have proved that $x$ will also be infected in avoSI and is infected earlier. For a picture see Figure \ref{fig:case1}.

		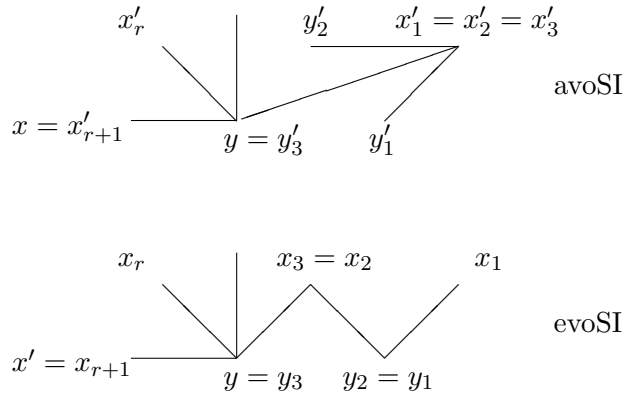
\begin{figure}[h!]
			\begin{center}
				\begin{picture}(220,190)
					\put(80,40){\line(1,1){28}}
					\put(80,40){\line(0,1){40}}
					\put(80,40){\line(-1,1){28}}
					\put(80,40){\line(-1,0){40}}
					\put(75,30){$y=y_3$}
					\put(120,30){$y_2=y_1$}
					\put(95,75){$x_3=x_2$}
					\put(170,75){$x_1$}
					\put(35,75){$x_r$}
					\put(-5,35){$x'=x_{r+1}$}
					\put(108,68){\line(1,-1){28}}
					\put(136,40){\line(1,1){28}}
					\put(200,50){evoSI}
					\put(80,130){\line(0,1){40}}
					\put(80,130){\line(-1,1){28}}
					\put(80,130){\line(-1,0){40}}
					\put(75,120){$y=y_3'$}
					\put(130,120){$y_1'$}
					\put(105,165){$y_2'$}
					\put(140,165){$x_1'=x_2'=x_3'$}
					\put(35,165){$x_r'$}
					\put(-5,125){$x=x_{r+1}'$}
					\put(164,158){\line(-3,-1){82}}
					\put(164,158){\line(-1,0){56}}
					\put(136,130){\line(1,1){28}}
					\put(200,140){avoSI}
				\end{picture}
				\caption{Illustration of the case $m>1$. We assume that at the time of the rewiring
					$x_1$, $y_1$ and $x_2$ are infected. Since one flips coins to determine 
					the end that rewires, the sequence of of edges $(x_k',y_k')$ in avoSI is different
					from the edges in evoSI. However, thanks to the use of $U_{e,\ell}$ to determine the new endpoint, the second rewiring brings the edge to $y$, and there is a correspondence between
					the vertices in the two processes indicated by the drawing.}
				\label{fig:case2}
			\end{center}
		\end{figure}
		
		\mn
		{\bf Case 2.} If $m>1$, then again by Proposition \ref{prop:avoevo}, the induction hypothesis and \eqref{tell} we see that in the avoSI picture, $e$ will be  rewired at least $r$ times and $y$ becomes an endpoint of $e$ exactly after $m-1$ rewirings. After this point we can  repeat the analysis in the case of $m=1$ to deduce that $x$ is also infected no later than $t$ in avoSI.
		For a picture see Figure \ref{fig:case2}.
	\end{proof}

	We now show that avoSI and delSI actually have the same critical value (and thus also have the same critical value as evoSI by Theorem \ref{critical_value}).
	\begin{lemma}\label{critical_avosi}
		Theorem \ref{critical_value} still holds if we replace evoSI by avoSI. 
	\end{lemma}
	\begin{proof}
		As we will see, Lemma \ref{critical_avosi}  follows by repeating the proof of Theorem \ref{critical_value}.
		Since avoSI stochastically dominates evoSI (in terms of final epidemic size) and evoSI dominates delSI, avoSI must also dominate delSI. 
		We claim that if we run avoSI on a tree and no edge is rewired to  vertices that are infected up to time  $t$, then  there are no unstable $I-I$
		pairs up to time $t$. This immediately implies that the  evolution of the avoSI process is equal to that of the evoSI process starting from the same initially infected vertex up to time $t$. 
		To prove the claim, suppose there is an unstable $I-I$ pair connecting vertex $x$ and $y$, then there must be two  infection paths that lead to the infections of $x$ and $y$. Here an infection path for $x$ is just a sequence of vertices $u_0\to u_1 \to \cdots \to x$  where the former vertex in the chain infects the latter. Since the edge between $x$ and $y$ is unstable, we see that if we consider the union of the two infection paths together with the edge $(x,y)$ then we get a cycle of infected vertices. This contradicts with the assumption that the original graph is a tree and no edges are rewired to vertices infected by $t$ (so that rewirings will not help create any cycle of infected vertices). 
		
		As we mentioned in the proof of Lemma  \ref{l:atst}, by Lemma 2.12 in \cite{vdH3}, whp the subgraph obtained by exploring the neighborhoods of $n^{3/8}/\log n$ vertices of any fixed vertex is a tree. The proof of Theorem \ref{critical_value}(i) implies that with high probability no edge is rewired to infected vertices up to step $O(n^{1/3})$ (i..e, up to the exploration of the neighborhoods of $O(n^{1/3})$ vertices).  Hence the condition of the above claim is satisfied.
		Using the conclusion of the claim we see that whp we can couple the avoSI process and evoSI process such that they coincide up to step $O(n^{1/3})$. Therefore the proof of Theorem \ref{critical_value} also applies to the comparison of  avoSI and delSI and hence Lemma \ref{critical_avosi} follows. 
	\end{proof}
	
	\subsection{Tightness of $\{\S_{t,k}/n,t\geq 0\}_{n\geq 1}$} \label{sec:tight}

	We first consider $\tilde S_{t,k}$, the number of susceptible vertices with $k$ half-edges at time $t$ in the  original avoSI process (i.e., without the time change). We have the following equation 
	
	\begin{equation}\label{eqStk1}
		d\tilde S_{t,k}=-\left(\lambda\tilde X_{I,t}\frac{k\tilde S_{t,k}}{\tilde X_t-1}\right)dt+\left(1_{\{k\geq 1\}}\rho\tilde X_{I,t} 
		\frac{\tilde S_{t,k-1}}{n}\right)dt-\left(\rho\tilde X_{I,t} 
		\frac{\tilde S_{t,k}}{n}\right)dt+d\tilde M_{t,k}.
	\end{equation}
	
	\noindent To explain the terms
	\begin{enumerate}
		\item 
		At rate $\lambda \tilde X_{I,t}$ infections occur. The infected half-edge attaches to a susceptible vertex with $k$ half-edges, which we call an $S^k$, with probability $k\tilde S_{k,t}/(\tilde X_t-1)$. The $-1$ in the denominator is because the half-edge will not connect to itself.
		\item 
		At rate $\rho \tilde X_{I,t}$ rewirings occur. If $k\ge 1$, the half-edge gets attached to an $S^{k-1}$ with probability $\tilde S_{k-1,t}/n$, promoting it to an $S^k$.
		\item 
		If the rewired half-edge gets attached to an $S^{k}$, which occurs with probability $\tilde S_{k,t}/ n$, it is promoted to an $S^{k+1}$ and an $S^k$ is lost.
		\item
		If $Z_t$ is a Markov chain with generator $L$ then Dynkin's formula implies
		$$
		f(Z_t) - \int_0^t Lf(Z_s) \,ds \quad\hbox{is a martingale.}
		$$
		See Chapter 4, Proposition 1.7 in \cite{EK}. Fortunately, we do not need an explicit formula for the martingale. All that is important is that when $f(Z_t) =\tilde S_{t,k}$,  $\tilde M_{\cdot,k}$ has jumps equal to $\pm 1$. 
	\end{enumerate}
	
	The equation for $\S_{t,k}$ can be then obtained by multiplying the first three terms in the right hand side of \eqref{eqStk1} by the time change, leading to
	
	\begin{align}
		d\S_{t,k}=&-\left(\lambda \X_{I,t}\frac{\X_t-1}{\lambda \X_{I,t}}\frac{k\S_{t,k}}{\X_t-1}\right)dt
		+\left(1_{\{k\geq 1\}}\rho \X_{I,t} \frac{\X_t-1}{\lambda \X_{I,t}}\frac{\S_{t,k-1}}{n}\right)dt
		\nonumber\\
		&-\left(\rho \X_{I,t} \frac{\X_t-1}{\lambda \X_{I,t}}\frac{\S_{t,k}}{n}\right)dt +d\M_{t,k}.
		\label{eqStk}
	\end{align}
	Here $\M_{t,k}$ is a time-changed version of the previous martingale so it is also a margingale with jumps $\pm 1$.
	
	Canceling common factors and dividing both sides of \eqref{eqStk} by $n$  we get
	
	\begin{align}
		d\left(\frac{\S_{t,k}}{n}\right) & =-\left(k\frac{\S_{t,k}}{n}\right) dt
		+\left(1_{(k\geq 1)} \frac{\rho }{\lambda}\frac{\X_t-1}{n}\frac{\S_{t,k-1}}{n}\right)dt
		\nonumber\\
		&-\left(\frac{\rho }{\lambda}\frac{\X_t-1}{n}\frac{\S_{t,k}}{n}\right)dt+d\left(\frac{\M_{t,k}}{n}\right).
		\label{eqStk2}
	\end{align}
	
	We now show that for all  fixed $k\geq 0$ and $T>0$,
	\beq
	\sup_{0\leq t \leq \T}\abs{\M_{t,k}}/n\CP 0.
	\label{term4}
	\eeq
	To do this we note that the expected value of quadratic variation  of $\M_{\t,k}$ evaluated at time $T$, which is also equal to $\E(\M_{\T,k}^2)$, is bounded above by the expectation of total number of jumps in the whole avoSI process, which is equal to 
	$$
	\E\left(X_0/2+\sum_{j=1}^{X_0} N_j\right).
	$$ 
	Here we have a factor of 2 in the denominator  because each pairing event takes two half-edges and
	$N_j$ is the number of times that half-edge  $j$ gets transferred to another vertex. Note that an infected half-edge gets rewired before being paired with probability at most $\rho/(\lambda+\rho)$ and susceptible half-edge cannot get rewired unless the vertex it is attached to becomes infected. Thus, $N_j$ is stochastically dominated by a Geometric($\rho/(\lambda+\rho)$) distributed random variable, so that for all $j$, $\E(N_j)\le C$ for some constant $C$. Therefore by $L^2$ maximal inequality applied to the submartingale $\abs{\M_{\t,k}}$,  we obtain that
	\begin{equation}\label{mskctl}
		\E\left(\sup_{0\leq t \leq \T} \M_{t,k}^2\right)\leq 4\E(\M_{\T,k}^2  )\leq Cn,
	\end{equation}
	where $C$ is a constant whose value is unimportant.
	
	Since $\S_{0,k}/n\leq 1$, we know   $\{\S_{0,k}/n\}_{n\geq 1}$ is a tight sequence of random variables. To establish tightness of  $\{\S_{t,k}/n,t\geq 0\}_{n\geq 1}$  we need to show
	for any fixed $\ep,\delta>0$, there is a $\theta>0$ and integer $n_0$ so that for $n\geq n_0$
	\begin{equation}\label{tightness}
		\P\left (\sup_{\abs{t_1-t_2}\leq \theta, t_1,t_2\leq T  } \abs{\S_{t_1\wedge \gamma_n, k}-\S_{t_2\wedge\gamma_n,k}}/n \geq \delta \right)
		\leq \ep.
	\end{equation}
	Assuming \eqref{tightness} for the moment, 
	we see that  $\{\S_{t\wedge \gamma_n,k}/n,t\geq 0\}$, as an element of $\D$, the space of right continuous paths with left limits, satisfies condition (ii) of
	Proposition 3.26 in \cite{JS}. Consequently,  
	$\{\S_{t\wedge \gamma_n,k}/n,t\geq 0\}_{n\geq 1}$ is a tight sequence. 
	It remains to prove \eqref{tightness}. We note that there exists a constant $C$ so that
	\beq
	\P(\X_0/n>C)\leq \ep/3,
	\label{X0bd}
	\eeq
	since $\E(\X_0)=\E(\sum_{i=1}^n D_i)= nm_1$.  Hence using $\S_{t,k-1}+\S_{t,k}\leq n$  and $\X_t\leq \X_0 $ with \eqref{eqStk2} 
	\begin{align*}
		&\P\left(\sup_{\abs{t_1-t_2}\leq \theta, t_1,t_2\leq T  } \abs{\S_{t_1\wedge \gamma_n, k}
			-\S_{t_2\wedge\gamma_n,k}}/n \geq \delta \right)
		\\
		\leq & \P\left( k\theta>\frac{\delta}{4} \right) + 2\P\left(\frac{\rho}{\lambda} \frac{X_0}{n}\theta\geq \frac{\delta}{4}\right)
		+\P\left(2 \sup_{0\leq t \leq \T} \abs{\M_{t,k}/n} \geq \frac{\delta}{4} \right).
	\end{align*}
	Using \eqref{X0bd} and \eqref{mskctl}, we see that if we pick $\theta$ small and $n$ large then the last line is $\le \ep$.
	This proves \eqref{tightness} and thus completes the proof of  tightness of the sequence $\{\S_{t\wedge \gamma_n,k}/n,t\geq 0\}_{n\geq 1}$.

	\subsection{Convergence of $\{\S_{t,k}/n,t\geq 0\}_{n\geq 1}$} \label{sec:convStk}
	Note that the evolution for $\X_t$ has the same transition rates as the time-changed SIR dynamics defined in \cite{JLW} so their equation (3.4)  also holds true in avoSI, which gives for any fixed $T$,
	\begin{equation}\label{x_t}
		\sup_{0\leq t \leq \T}\abs{	\frac{\X_t}{n}- m_1 \exp(-2t)} \CP 0. 
	\end{equation}
	In order to upgrade $\sup_{0\leq t\leq \T}$ to $\sup_{0\leq t\leq \gamma_n}$, we note that for any $\ep>0$, we can pick a sufficiently large $T$ so that $m_1\exp(-2T)<\ep$. 
	Then by the monotonically decreasing property of $\X_t$,
	\begin{align}
		\label{Tgam}
		&	\P\left(\gamma_n>T, \sup_{T< t\leq \gamma_n} \abs{\frac{\X_t}{n}-m_1\exp(-2t)}>3\ep\right)
		\leq \P(\gamma_n>T,\X_T/n>2\ep) \\
		\leq & \P\left(\sup_{0\leq t \leq \T}\abs{	\frac{\X_t}{n}- m_1 \exp(-2t)}>\ep\right)\leq \ep, &
		\nonumber
	\end{align}
	for $n$ large enough. Thus we deduce 
	$$
	\P\left(	\sup_{0\leq t\leq  \gamma_n}\abs{\frac{\X_t}{n}-m_1\exp(-2t)}>\ep\right)\leq 2\ep
	$$ 
	for $n$ sufficiently large, which proves the first equation of \eqref{3lim}.
	
	By the tightness of $\{\S_{t\wedge \gamma_n,k}/n,t\geq 0\}_{n\geq 1}$, we see for any subsequence of $\S_{t,k}/n$ we can extract a further subsequence that converges in distribution to a  process $\s_{t,k}$ with continuous sample path. By the Skorokhod representation theorem we can assume the convergence is actually in the almost sure sense and we can also assume that 
	$\X_{t}/n$  converges a.s. to $m_1\exp(-2t)$.
	Having established tightness, a standard argument implies that we can show the convergence of $\S_{t,k}/n$ by establishing that the limit  $\s_{t,k}$ is independent of the subsequence. First consider the case $k=0$. The first two terms on the right-hand side of \eqref{eqStk2} are 0, so using \eqref{term4} and first equation of \eqref{3lim} we see that any subsequential limit $\s_{t,0}$ has to satisfy the equation
	
	\begin{equation}\label{s0t}
		\s_{t,0}=-\frac{\rho}{\lambda}m_1 \int_0^t \exp(-2z)\s_{z,0}\, dz.
	\end{equation}
	
	\noindent Since
	$z \to \exp(-2z)$ is Lipschitz continuous  this equation has a unique solution.
	
	Repeating this process for $k\geq 1$ we see that any subsequential limit $\s_{t,k}$ of  $\S_{t,k}/n$ satisfies the differential equation
	
	\begin{equation}\label{stk}
		\s'_{t,k}=-k\s_{t,k}-\alpha \exp(-2t)\s_{t,k}+1_{\{k\geq 1\}}\alpha \exp(-2t)\s_{t,k-1},
	\end{equation}
	
	\noindent
	where $\alpha=\rho m_1/\lambda$.
	This system of equations can be solved explicitly. First we rewrite the equations as 
	$$
	\s_{t,k}'+[k+\alpha \exp(-2t)]\s_{t,k}= 1_{(k\ge 1)}\alpha \exp(-2t)\s_{t,k-1}.
	$$
	Define
	\begin{equation}
		g_{t,k}=\exp\left(kt+ (\alpha/2) (1-\exp(-2t))\right)\s_{t,k},
	\end{equation}
	then 
	\begin{align*}
		g_{t, 0}'  &= \alpha \exp(-2t) g_{t,0} + \exp((\alpha/2)(1-\exp(-2t)) \s'_{t,0} \\
		& = \alpha \exp(-2t) \exp((\alpha/2) (1-\exp(-2t)))\s_{t,0} \\
		& + \exp(\alpha/2)(1-\exp(-2t))[-\alpha \exp(-2t)] \s_{t,0} = 0.
	\end{align*}
	Let $A_{k,t} = \exp(kt+ (\alpha/2) (1-\exp(-2t)))$.
	An almost identical calculation for $k \ge 1$ gives
	\begin{align*}
		g_{t,k}'  &= [k +\alpha \exp(-2t)] g_{t,k} + A_{k,t} \s'_{t,k} \\
		& = [k+\alpha \exp(-2t)]A_{t,k} \s_{t,k} \\
		& + A_{t,k} \{ [-k -\alpha \exp(-2t)] \s_{t,k} + \alpha \exp(-2t) \s_{t,k-1}\},
	\end{align*}
	so we have
	$$
	g'_{t,k}= A_{t,k}\alpha \exp(-2t) \s_{t,k-1} = \alpha \exp(-t)g_{t,k-1}.
	$$
	
	Making the change of variable $s=\alpha (1-\exp(-t))$ and letting $h_{s, k}= g_{t,k}$, we see that  $h_{s, 0}$ is constant in $s$ and
	$$
	h'_{s,k}=h_{s, k-1}, k\geq 1,
	$$
	from which we see $h_{s,k}$ is a polynomial of degree $k$ in $s$ and for all $\ell\leq k$ the $\ell$-th derivative of $h_{s,k}$ at $s=0$ equals 
	$h_{0,k-\ell}$. From this we obtain that for all $k$,
	$$
	h_{s,k}=\sum_{\ell=0}^k \frac{h_{0,k-\ell}}{l!}s^{\ell}.
	$$
	The initial conditions are $g_{0,k}=h_{0,k}=s_{0,k}=p_k=\P(D=k)$.
	It follows that
	\begin{equation}
		h_{s,k}= \sum_{\ell=0}^k \frac{p_{k-\ell}s^{\ell} }{\ell!},
	\end{equation}
	and hence using definitions of $\s$, $g$, and $h$
	\begin{align}
		g_{t,k}&=h_{s,k}= \sum_{\ell=0}^k \frac{p_{k-\ell}}{\ell!} (\alpha (1-w))^k,
		\label{gtkeq}\\
		\s_{t,k}&=\exp\left(-\frac{\alpha}{2}(1-w^2)\right)w^k\sum_{\ell=0}^k \frac{p_{k-\ell}}{\ell!} (\alpha (1-w))^{\ell},
		\label{stkeq2}
	\end{align}
	where $w=w(t)=\exp(-t)$.
	
	\subsection{Summing the $\s_{t,k}$}\label{sec:sum_stk}
	
	We pause to record the following fact which we will use later. From the explicit expression for $\s_{t,k}$ in \eqref{stkeq2},
	dropping the factor $\exp(-\alpha(1-w^2)/2)\leq 1$  and writing $k=(k-\ell) + \ell$, 
	\begin{align}
		\sup_{t\geq 0} \sum_{k \geq K}k\s_{t,k}\leq  &
		\sup_{0\leq w\leq 1} \sum_{k \geq K} \sum_{\ell=0}^{k} \frac{(k-\ell)p_{k-\ell}}{\ell!}(\alpha (1-w))^{\ell} w^k
		\nonumber \\
		&+\sup_{0\leq w\leq 1} \sum_{k \geq K} \sum_{\ell=1}^{k} \frac{\ell p_{k-\ell}}{\ell!}(\alpha (1-w))^{\ell} w^k.
		\label{splitup}
	\end{align}
	The $\ell=0$ term in the first sum is bounded by
	$$
	\sum_{k \ge K} k p_k.
	$$
	The remainder of the two sums is bounded by
	\begin{equation}\label{remaindsum}
		\sup_w w^K (1-w) \left[ \sum_{k \ge \ell, \ell \ge 1} \left( \frac{\alpha^\ell (k-\ell) p_{k-\ell}}{\ell!}
		+  \frac{\alpha^\ell \ell p_{k-\ell}}{\ell!} \right) \right].
	\end{equation}
	Interchanging the order of summation in the double sum and letting $m=k-\ell$, the double sum in \eqref{remaindsum} is bounded by
	$$
	\sum_{\ell=1}^\infty \ell \cdot \frac{\alpha^\ell}{\ell!} \left( \sum_{m=0}^\infty (m p_m +  p_m) \right) \le C_{\alpha,D}.
	$$
	Combining our calculations leads to
	\beq
	\limsup_{K\to \infty} \sup_{t\geq 0} \sum_{k \geq K}k\s_{t,k} 
	\leq C \limsup_{K\to\infty} \left(\sum_{k\geq K}kp_k+\sup_w w^K(1-w)\right)= 0.
	\label{limKstk=0}
	\eeq

	We can use this bound to show that
	$\sum_{k=0}^{\infty}\S_{t,k}/n$ converges to $\sum_{k=0}^{\infty}\s_{t,k}$
	as well as $\sum_{k=0}^{\infty}k\S_{t,k}/n$ converges to $\sum_{k=0}^{\infty}k\s_{t,k}.$
	Since the proofs are similar, we only prove the second result. We fix a large number $K$ and 
	observe that $\sum_{k\geq K} k\S_{t,k} $ satisfies the equation
	\begin{align}
		d\left(\sum_{k\geq K} k\S_{t,k} \right)= &
		-\left(\sum_{k\geq K} k^2 \S_{t,k}\right) dt
		+\rho \X_{I,t} \frac{X_t-1}{\lambda \X_{I,t}}\frac{K\S_{t,K-1}}{n}dt
		\nonumber \\
		&+\sum_{k\geq K} \rho \X_{I,t} \frac{\X_t-1}{\lambda \X_{I,t}}\frac{\S_{t,k}}{n}dt
		+d\hat{M}_{t,K}.
		\label{kstk}
	\end{align}
	Here $\hat{M}_{t,K}$ is a martingale that satisfies
	\begin{equation}\label{2moment}
		\sum_{t>0} (\hat{M}_{t,k}-\hat{M}_{t-,k})^2\leq 3
		\sum_{\ell=1}^n Q_{\ell}^2,
	\end{equation}
	where $Q_{\ell}$ is the number of half-edges that vertex $\ell$ has before it becomes infected. This follows from the observation that there are two sources for the jump of $\sum_{k \geq K}k\S_{t,k}$: 
	
	\begin{itemize}
		\item A susceptible vertex $\ell$ with at least $K$ half-edges gets infected. Then
		$\sum_{k \geq K}k\S_{t,k}$ drops by the number of half-edges of vertex $\ell$, which is $Q_{\ell}$. Each vertex can contribute to this type of jumps at most once.
		\item A half-edge of an infected vertex gets transferred to a susceptible vertex of degree at least $K-1$. Then  $\sum_{k \geq K}k\S_{t,k}$ increases by either $\le K$ (if the vertex gaining a half-edge had $K-1$ half-edges before) or 1 (if the vertex gainning a half-edge had at least $K$ half-edges before).
		Vertex $\ell$ can contribute at most $Q_{\ell}$ times to  jumps of size 1 and at most once to jumps of size $K$.  
	\end{itemize}
	
	Initially vertex $\ell$ has $D_\ell$ half-edges. As time grows the half-edges of other vertices might be transferred to vertex $\ell$, the number of which  is dominated by 
	\begin{equation}\label{V}
		V = \mbox{Binomial}(W ,1/n ) \quad\hbox{where $W= \sum_{m=1}^nD_{m}$}.
	\end{equation}
	Thus we have
	\begin{equation}\label{EQ_l^2}
		\E[Q_{\ell}^2 ]\leq \E[(D_{\ell} + V)^2]\leq 2\E D^2+ 2\E V^2.
	\end{equation}
	Conditioning on the value of $W$ we have
	\begin{equation}\label{EV^2}
		\E V^2 = (1/n)(1-1/n) \E W+ \E (W/n)^2 \leq C.
	\end{equation}
	It follows from \eqref{2moment} and \eqref{EV^2} that
	\begin{equation}\label{mhatzk}
		\E\left(\sup_{0\leq t\leq \T} \abs{\hat{M}_{t,k}}^2\right)\leq 4\E(\hat{M}_{\T,k}^2  )\leq Cn.
	\end{equation}
	
	Writing \eqref{kstk} as an integral equation and dropping the negative term $-(\sum_{k\geq K} k^2 \S_{t,k})$ we see 
	\begin{align*}
		\sum_{k\geq K} k\S_{\z,k} & \leq 	\sum_{k\geq K} k\S_{0,k}+ \frac{\rho}{\lambda}K \int_0^z \frac{\X_{u\wedge \gamma_n}}{n} \S_{u\wedge \gamma_n,K-1}du\\
		&+\frac{\rho}{\lambda}\int_0^z \frac{\X_{u\wedge \gamma_n}}{n} \sum_{k \geq K}\S_{u\wedge \gamma_n,k}+\hat{M}_{\z,k}.
	\end{align*}
	Take $\sup_{0\leq z\leq t}$  on both sides we have
	\begin{equation}\label{kSzk}
		\begin{split}
			\sup_{0\leq z\leq \t} \sum_{k\geq K} k\S_{z,k}
			&\leq \sum_{k\geq K} k\S_{0,k}+\frac{\rho}{\lambda}K\int_0^t \sup_{0\leq u\leq z\wedge \gamma_n}
			\frac{\X_{u}}{n} \S_{u,K-1}dz \\
			+\frac{\rho}{\lambda} &\int_0^t \sup_{0\leq u\leq z\wedge \gamma_n}
			\frac{\X_{u} }{n} \sum_{k \geq K}\S_{u,k}dz
			+\sup_{0\leq z \leq \t}\hat{M}_{z,k}.
		\end{split}
	\end{equation}
	
	Dividing both sides of \eqref{kSzk} by $n$, taking the square and using $(a+b+c+d)^2 \le 4(a^2+b^2+c^2+d^2)$ we have

	\begin{align}
		\left(\sup_{0\leq z\leq \t}\frac{\sum_{k\geq K} k\S_{z,k}}{n}  \right)^2  
		&\le 4 \left(\sum_{k\geq K} \frac{k\S_{0,k}}{n} \right)^2 
		+4 \left(\frac{\rho}{\lambda}K\right)^2 t \int_0^t \left( \sup_{0\leq u\leq \z}\frac{\X_{u}}{n} \S_{u,K-1} \right)^2 dz 
		\nonumber\\
		+4 \left(\frac{\rho}{\lambda} \right)^2t & \int_0^t \left( \sup_{0\leq u\leq z\wedge \gamma_n} \frac{\X_{u} }{n} \sum_{k \geq K}\S_{u,k} \right)^2 dz
		+4 \left( \sup_{0\leq z \leq \t}\hat{M}_{z,k} \right)^2, 
		\label{step1}
	\end{align}
	where we have also used the Cauchy-Schwarz inequality to conclude that for any function $g=g(u)$,
	$$ 
	\left(\int_0^t \left[\sup_{0\leq u \leq z\wedge \gamma_n} g(u) \right] \,  dz\right)^2
	\le t \int_0^t \left[ \sup_{0\leq u \leq z\wedge \gamma_n} g^2(u)\right] \, dz.
	$$
	
	If we use $\EE$ to denote the conditional expectation with respect to the $\sigma$-algebra generated by $\X_0$, then
	for any $\ep>0$, using equation \eqref{mhatzk} we can find a constant (depending on $\ep$) $C_{\ref{step-5}}>0$ so that
	\begin{equation}\label{step-5}
		\P\left(\EE\left(\sup_{0\leq z \leq \t}\hat{M}^2_{z,k}\right)>C_{\ref{step-5}} n\right)\leq \ep. 
	\end{equation}
	We then take another constant $C_{\ref{step-4}}$ such that 
	\begin{equation}\label{step-4}
		\P(\X_0/n>C_{\ref{step-4}}) \leq \ep.
	\end{equation}

	Taking the conditional expectation of \eqref{step1} with respect to $\X_0$ we see on the event 
	$$
	\Omega_0 = \{\X_0/n\leq C_{\ref{step-4}}\}\cap \left\{   \EE(\sup_{0\leq z \leq \t}\hat{M}^2_{z,k})\leq C_{\ref{step-5}}n  \right\}
	$$ 
	which has probability $\geq 1-2\ep$, there exists a constant $C_{\ref{step2}}$ such that 
	\begin{equation}
		\begin{split}
			&\EE\left(\sup_{0\leq z\leq \t}\frac{\sum_{k\geq K} k\S_{z,k}}{n}  \right)^2  \leq
			C_{\ref{step2}} t	 \int_0^t 	\EE\left[\left(\sup_{0\leq u\leq z\wedge \gamma_n}\frac{\sum_{k\geq K} k\S_{u,k}}{n} \right)^2\right] dz\\
			&+C_{\ref{step2}} \left(K^2t\EE\left[\left( \sup_{0\leq z\leq \t}\frac{S_{z,K-1}}{n} \right)^2\right]+\frac{1}{n}+\EE\left[\left(\sum_{k\geq K} \frac{k\S_{0,K}}{n} \right)^2\right]\right).
		\end{split}
		\label{step2}
	\end{equation}
	
	If we let 
	$$
	\phi(t) = \EE\left(\sup_{0\leq z\leq \t}\frac{\sum_{k\geq K} k\S_{z,k}}{n}  \right)^2,
	$$
	and $B$  equals the second line in \eqref{step2}, then for $0\leq t\leq T$ we have 
	$$
	\phi(t) \leq  C_{\ref{step2}} T\int_0^t \phi(s) \, ds + B.
	$$
	Gronwall's inequality gives
	\begin{equation}\label{gronwallfinal}
		\mbox{if }\phi(t) \le \alpha(t) + \int_0^t \beta(s) \phi(s) \, ds \mbox{, then }\phi(t) \le \alpha(t) \exp\left(\int_0^t \beta(s) \, ds \right),
	\end{equation}
	provided that $\beta(t)\ge 0$ and $\alpha(t)$ is nondecreasing.
	So applying \eqref{gronwallfinal} we have $$\phi(t) \le B \exp(C_{\ref{step2}} T^2)$$ on $\Omega_0$, that is,
	\begin{equation}\label{tail exp}
		\begin{split}
			&\EE\left(\sup_{0\leq t\leq \T}\frac{\sum_{k\geq K} kS_{t,k}}{n} \right)^2 \\
			& \leq C_{\ref{step2}}\exp(C_{\ref{step2}}T^2)\left(K^2t\EE\left(\sup_{0\leq t\leq \T}\frac{S_{t,K-1}}{n} \right)^2 + \frac{1}{n}
			+\EE\left(\sum_{k\geq K} \frac{k\S_{0,k}}{n} \right)^{\sqz 2}\ \right).
		\end{split}
	\end{equation}
	To control the first term on the right
	we use the convergence of $\S_{t,K-1}/n$ to $\s_{t,K-1}$ in probability as well as the bounded convergence theorem (since $\S_{t,K-1}/n\leq 1$) to obtain that
	\begin{equation}\label{step10}
		\limsup_{n\to\infty} \E\left(\sup_{0\leq t\leq \T}\frac{\S_{t,K-1}}{n} \right)^2 \leq  \left( \sup_{0\leq t\leq T}\s_{t,K-1}\right)^2.
	\end{equation}
	Using \eqref{step10} and \eqref{limKstk=0}, if we  first pick a large $K$,
	then for all $n$ sufficiently large all three terms on right hand side of \eqref{tail exp}
	smaller than $\ep^2/3$ with probability at least $1-\ep$, which in turn 
	implies that there is a set $\Omega_1$ with $\P(\Omega_1)\geq 1-3\ep$, so that on $\Omega_1$ 
	$$
	\EE\left(\sup_{0\leq t\leq \T}\frac{\sum_{k\geq K} kS_{t,k}}{n} \right)^2  \leq \ep^2.  
	$$
	It follows that
	$$
	\E\left(1_{\Omega_1} \EE\left(\sup_{0\leq t\leq \T}\frac{\sum_{k\geq K} kS_{t,k}}{n} \right)^{\sqz 2}  \right)
	=\E\left( 1_{\Omega_1} \sup_{0\leq t\leq \T}\frac{\sum_{k\geq K} kS_{t,k}}{n}   \right)^{\sqz 2}\leq \ep^2. 
	$$
	Using $\P(\Omega_1)\geq 1-3\ep$ and the Chebyshev's inequality, we see that with probability $\geq 1-4\ep$,
	$$
	\sup_{0\leq t\leq \T}\frac{\sum_{k\geq K} kS_{t,k}}{n}   \leq \ep.
	$$
	Fixing $t$ and $\ep$ and using the triangle inequality, we get
	\begin{equation}\label{sum_conv}
		\begin{split}
			\sup_{0\leq t\leq \T} \abs{\frac{\sum_{k\geq 0}k\S_{t,k}  }{n}-\sum_{k\geq 0}k \s_{t,k}} &\leq 	\sup_{0\leq t\leq \T} \abs{\frac{\sum_{k=0}^Kk\S_{t,k}  }{n}-\sum_{k=0}^Kk \s_{t,k}}\\
			&+\sup_{0\leq t\leq \T} \frac{\sum_{k\geq K}k\S_{t,k} }{n}+\sup_{0\leq t\leq T}
			\sum_{k\geq K} k\s_{t,k}. 
		\end{split}
	\end{equation}
	By first choosing $K$ large enough and then $n$ large enough we can make both the first and second terms on the right hand side of \eqref{sum_conv}  smaller than $\ep$ with probability at least $1-4\ep$. The third term can also be made smaller than $\ep$ using \eqref{limKstk=0}.
	
	Since $\ep$ is arbitrary we see that
	$$
	\sup_{0\leq t\leq T\wedge \gamma_n} \abs{\frac{\sum_{k=0}^{\infty}k\S_{t,k} }{n}-\sum_{k=0}^{\infty}k\s_{t,k}}\CP 0. 
	$$
	To find $\sum_{k=0}^{\infty} \s_{t,k}$ and $\sum_{k=0}^{\infty}k\s_{t,k}$,
	recall that we set $w=w(t)=\exp(-t)$ and $G(w) = \E(w^D)$. The limit of the fraction of susceptible nodes $\s_t$  satisfies
	\begin{equation}
		\begin{split}
			\s_t&=\sum_{k=0}^{\infty} \s_{t,k}= \exp(-(\alpha/2)(1-w^2))\sum_{r,\ell\geq 0}\frac{p_r}{\ell!}   (\alpha (1-w))^\ell w^{r+\ell}\\
			&=\exp(-(\alpha/2)(1-w)^2) G(w).
		\end{split}
	\end{equation}
	
	The limit of (scaled) number of susceptible half-edges satisfies
	$$
	\x_{S,t}=\sum_{k=0}^{\infty} k \s_{t,k}= \exp(-\frac{\alpha}{2}(1-w^2))
	\sum_{r,\ell\geq 0}(r+\ell)\frac{p_r}{\ell!} (\alpha (1-w))^\ell w^{r+\ell}.
	$$
	The double sum equals
	$$
	w \sum_{r\ge 0} r p_r w^{r-1} \sum_{\ell\ge 0} \frac{(\alpha(1-w)w)^\ell}{\ell!} 
	+ \alpha w(1-w) \sum_{r\ge 0}  p_r w^{r} \sum_{\ell\ge 1} \frac{(\alpha(1-w)w)^{\ell-1}}{(\ell-1)!},
	$$
	so we have 
	\beq
	\x_{S,t}=\exp\left(-(\alpha/2)(1-w)^2\right)w(G'(w)+\alpha (1-w)G(w)).
	\label{xSteq}
	\eeq

	{\bf Extension to time $\gamma_n$.} 
	We have proved the second and third statements of \eqref{3lim}  with $0\leq t\leq \gamma_n$ replaced by $0\leq t\leq \T$ for any fixed $T$. To upgrade this to $0\leq t\leq \gamma_n$, note that 
	$\sum_{k=0}^{\infty}k\S_{t,k}\leq \X_t$. Picking
	a large $T$ satisfying $(m_1+\alpha)\exp(-T)\leq \ep$ and
	re-using equation \eqref{Tgam} we obtain that
	\begin{align}
		&\P\left(\gamma_n>T, \sup_{T \leq t  \leq \gamma_n}
		\abs{	\frac{\sum_{k=0}^{\infty} k\S_{t,k}}{n}-\exp\left(-\frac{\alpha}{2}(w-1)^2 \right)w(G'(w)+\alpha (1-w)G(w)) }\geq 3\ep \right)
		\nonumber\\
		\leq &\P \left(\gamma_n>T,\sup_{T\leq t\leq \gamma_n}\sum_{k=0}^{\infty} k\S_{t,k}/n>2\ep\right) 
		\leq \P(\gamma_n>T,\X_T/n>2\ep )\leq \ep. \label{SL1'}
	\end{align}
	
	This proves the third equation of \eqref{3lim}. The proof of the second equation in \eqref{3lim} is slightly more  complicated. Again we fix a large $T$ such that $(m_1+\alpha)\exp(-T)\leq \ep$
	and
	\begin{equation}\label{largeT}
		\abs{\exp(-\alpha/2)G(0)-\exp\left(-\alpha/2 \left(\exp(-T)-1\right)^2\right)G(\exp(-T))}\leq \ep.
	\end{equation}
	Equation \eqref{largeT} and the fact that
	$ \exp\left(-\alpha/2 \left(\exp(-t)-1\right)^2\right)G(\exp(-t))$ is decreasing in $t$ imply that
	\begin{align}\label{largeT'}
		\sup_{t,t'>T} &\Bigl|   \exp\left(-\alpha/2 \left(\exp(-t)-1\right)^2\right)G(\exp(-t)) \\
		& -\exp\left(-\alpha/2 \left(\exp(-t')-1\right)^2\right)G(\exp(-t'))\Bigr| \leq \ep. \nonumber 
	\end{align}
	We first estimate the term $\S_{T,0}$ for large $T$. 
	Using the weaker version (i.e., with $\sup_{0\leq t\leq \T}$) of the second equation of \eqref{3lim}  we see that for $n$ large enough,
	\begin{equation}\label{ST0}
		\P\left(\gamma_n>T,  \abs{\S_{T,0}/n- \exp\left(-\alpha/2 \left(\exp(-T)-1\right)^2\right)G(\exp(-T))}  >\ep\right)\leq \ep. 
	\end{equation}
	On the event $\{\gamma_n>T\}$, $
	\sup_{T\leq t\leq \gamma_n} \abs{\S_{t,0}-\S_{T,0}}$ can be bounded by
	$\X_T$, since in order to lose a susceptible vertex of degree 0 there must be a half-edge transferred to it. It follows that
	\begin{equation}\label{ST>0}
		\P\left(\gamma_n>T, \sup_{T\leq t\leq \gamma_n} \abs{\S_{t,0}-\S_{T,0}}/n>2\ep\right)
		\leq \P\left(\gamma_n>T,\X_T/n>2\ep\right)\leq \ep. 
	\end{equation}
	Combining equations  \eqref{ST0}, \eqref{ST>0} and \eqref{largeT'} we see that
	\begin{equation}\label{0thterm}
		\P\left(\gamma_n>T, \sup_{T<t\leq \gamma_n}\abs{\S_{t,0}/n- \exp\left(-\alpha/2 \left(\exp(-t)-1\right)^2\right)G(\exp(-t)) }>4\ep \right)\leq 3\ep. 
	\end{equation}
	Using
	$\sum_{k \ge 1}\S_{t,k}\leq \sum_{k=0}^{\infty} k\S_{t,k}$ and equation \eqref{SL1'}
	we see that
	\begin{equation}\label{stk>1}
		\P\left(\gamma_n>T,\sup_{T\leq t\leq \gamma_n}\sum_{k \ge 1} \S_{t,k}/n>2\ep\right)\leq \ep.
	\end{equation}
	It follows from  \eqref{stk>1} and \eqref{0thterm} that 
	$$
	\P\left(\gamma_n>T, \sup_{T<t\leq \gamma_n}\abs{\sum_{k=0}^{\infty} \S_{t,k}/n-a_S \exp\left(-\alpha/2 (w(t)-1)^2\right)G(w(t)) }> 6\ep   \right)\leq 4\ep,
	$$
	which  proves the second equation of \eqref{3lim} and concludes the proof of \eqref{3lim}. 
	
	\subsection
	{Proof of Theorem \ref{epsize}}
	\label{sec:pfth6}
	
	Recall for all $t\leq \gamma_n$ we have $\X_t\geq \X_{S,t}$ and at $\gamma_n$ we have $\X_t=\X_{S,t}$ since $\gamma_n$ is the time that we run out of infected half-edges and the dynamics stop.  Note that by the definition of $f$ in \eqref{defoff} 
	we have
	$$
	\exp(f(w))=\frac{\x_t}{\x_{S,t}}.
	$$
	We  can rewrite $f$ as
	$$
	f(w)  = \log(m_1) + \log(w)- \log( G'(w) + \alpha(1-w) G(w) )  + \frac{\alpha}{2}(w-1)^2. 
	$$
	Taking the derivative we have
	\begin{equation}
		\label{1stderivative}
		f'(w)=\frac{1}{w}-\frac{G''(w)-\alpha G(w)+\alpha(1-w) G'(w)}{ G'(w)+\alpha (1-w)G(w)} +\alpha(w-1).
	\end{equation}
	To evaluate $f'(1)$, recall that $G(1)=1$, $G'(1)=m_1$ and $G''(1)=\E[D(D-1)] = m_2-m_1$. 
	The terms with $1-w$ vanish at $w=1$. Using $\alpha = \rho m_1/\lambda$ from  \eqref{alpha}, it follows that
	$$
	f'(1)=1-\frac{m_2 - m_1 - \rho m_1/\lambda}{m_1} = - \left( \frac{m_2-2m_1}{m_1} - \frac{\rho}{\lambda} \right).
	$$
	
	Theorem \ref{critical_value} tells us that in the supercritical case, we have $\lambda>(\rho m_1)/(m_2-2m_1)$,
	and hence  $f'(1)<0$, which implies that $f$ is positive on $(1-\delta,1)$ for some $\delta>0$.
	Theorem \ref{critical_value} also shows that when $\eta>0$ is small, 
	\begin{equation}
		\lim_{n\to\infty}\P(\I_{\infty}/n>\eta)=q(\lambda)>0,
	\end{equation}
	where $q(\lambda)$ is the survival probability of the two-phase branching process $\bar Z_m$ (defined in Section \ref{britton}).
	Let $t_{\eta}<\delta$ be some small number depending on $\eta$ such that
	\begin{equation}\label{teta}
		1-\exp\left(-\frac{\alpha}{2}(\exp(-t_{\eta})-1)^2\right)G(\exp(-t_{\eta}))<\eta/4.
	\end{equation}
	Conditionally on $\I_{\infty}/n>\eta$  for some small $\eta$, the second equation of \eqref{3lim} implies that with high probability 
	$\gamma_n$ is also bounded from below by  $t_{\eta}$ (depending on $\eta$),  since otherwise we would have 
	\begin{equation*}
		\begin{split}
			&	\limsup_{n\to\infty} \P(\gamma_n<t_{\eta},\I_{\gamma_n}/n>\eta)\\
			\leq & 	\limsup_{n\to\infty} \P\left(\gamma_n<t_{\eta}, \S_{\gamma_n}/n>
			\exp\left(-\frac{\alpha}{2}(\exp(-\gamma_n)-1)^2\right)G(\exp(-\gamma_n))
			-\eta/4,\I_{\gamma_n}/n>\eta\right)\\
			\leq &\limsup_{n\to \infty} \P\left( 
			\S_{\gamma_n}/n>\exp\left(-\frac{\alpha}{2}(\exp(-t_{\eta})-1)^2\right)G(\exp(-t_{\eta}))-\eta/4,\I_{\gamma_n}/n>\eta	\right)=0,
		\end{split}
	\end{equation*}
	where the last equality is due to \eqref{teta} and the fact that $\S_{\gamma_n}+\I_{\gamma_n}=n$. We have also used the definition of $\gamma_n$ so that $\I_{\infty}=\I_{\gamma_n}$ since no more vertices can be infected after $\gamma_n$.

	Recalling the definition of  the $\sigma$ in statement of Theorem \ref{epsize} we see that for any $\ep>0$, $$
	\inf_{t_{\eta}<t<-\log (\sigma+\ep) } \frac{\x_t}{\x_{S,t}}>1,
	$$
	which implies that for some $\ep'>0$ and all $t_{\eta}<t<-\log (\sigma+\ep)$,
	\begin{equation}\label{infxtxst}
		\x_t-\x_{S,t}>\ep'. 
	\end{equation}
	The first and third equations of \eqref{3lim} and  \eqref{infxtxst} imply that
	\begin{equation*}
		\begin{split}
			&\limsup_{n\to\infty} \P\left(t_{\eta}<\gamma_n<-\log(\sigma+\ep) \right) \\
			\leq & \limsup_{n\to\infty} \P\left(t_{\eta}<\gamma_n<-\log(\sigma+\ep),(\X_{\gamma_n}-\X_{S,\gamma_n})/n>\ep'/2 \right) =0,
		\end{split}
	\end{equation*}
	where we have used the fact that $\X_{\gamma_n}=\X_{S,\gamma_n}$. Since 
	we have already shown conditionally on $\I_{\infty}/n>\eta$ whp $\gamma_n>t_{\eta}$, we see that
	$$
	\lim_{n\to\infty}\P(\gamma_n>-\log (\sigma+\ep)|\I_{\infty}/n>\eta)=1.
	$$ 
	
	Using the monotonically decreasing property of $\S_t=\sum_{k=0}^{\infty}\S_{t,k}$ and the second equation of  \eqref{3lim},
	we see that conditionally on $\I_{\infty}/n>\eta$, for any $\ep>0$ whp 
	$$
	\S_{\gamma_n}/n \leq \exp\left(-\frac{\alpha}{2}(\sigma+\ep-1)^2\right)G(\sigma+\ep)+\ep.
	$$
	Hence for any $\ep>0$,
	$$
	\lim_{n\to\infty} \P(I_{\infty}/n>\nu-\ep |I_{\infty}/n>\eta)=1,
	$$
	where $\nu=1-\exp(-\alpha/2(\sigma-1)^2)G(\sigma)$.

	For the other direction, ($\star$) implies that we can pick $\delta'>0$ so that $f(w)<0$ on $(\sigma-\delta',\sigma)$. \eqref{3lim} thus implies with high probability $\gamma_n$ cannot be larger than $-\log(\sigma-\delta')$ since otherwise we would have $\X_{-\log(\sigma-\delta')}<\X_{S,-\log (\sigma-\delta')}$, which is impossible. Since $\delta'$ can be taken  arbitrarily small we conclude that for any $\ep>0$ 
	$$
	\lim_{n\to\infty} \P(\I_{\infty}/n<\nu+\ep)=1.
	$$ 
	
	\subsection{\bf Proof of Theorem  \ref{Q1}} \label{sec:pfth7}
	
	From \eqref{1stderivative} we see that as $\lambda\to\lambda_c$ we have $f'(1)\to 0$. The second derivative of $f$ is given by
	$$
	f''(w)=-\frac{1}{w^2}-\frac{G'''+\alpha(1-w) G''-2\alpha G'}{G'+\alpha (1-w)G}
	+\left(\frac{G''-\alpha G+\alpha (1-w)G' }{G'+\alpha (1-w)G}\right)^2+\alpha .
	$$
	Recall $\mu_k$ denotes the  factorial moment $\E[D(D-1) \cdots(D-k+1)]$.
	We have assumed that  $\E(D^5)<\infty$ so 
	$$
	G'''(1)=\E[D(D-1)(D-2)] =\mu_3<\infty. 
	$$
	Since the terms with $1-w$ vanish at $w=1$, inserting the values of $G(1)$, $G'(1)$, $G''(1)$ and $G'''(1)$ we get
	\begin{align*}
		f''(1) &=-1-\frac{G'''(1)-2\alpha G'(1)}{G'(1)}+ \left( \frac{ G''(1) - \alpha G(1)}{ G'(1)} \right)^2 + \alpha  \\
		& = - 1 - \frac{\mu_3}{m_1} + 2\alpha + \left( \frac{ m_2-m_1 - \alpha}{ m_1} \right)^2 + \alpha.
	\end{align*}
	Theorem \ref{critical_value} tells us that $\alpha_c = m_2-2m_1$ so 
	$$
	- 1 + \left(\frac{m_2-m_1 - \alpha_c}{m_1}\right)^2 = 0. 
	$$
	From this, we see that at $\alpha_c$
	$$
	f''(1)=\frac{-\mu_3+3m_1m_2-6m_1^2}{m_1}.
	$$
	Using $\mu_1=m_1, \mu_2=m_2-m_1$  this can be written  as
	\begin{equation}\label{f''1delta}
		f''(1) = -\frac{\mu_3}{\mu_1} + 3(\mu_2 - \mu_1) \equiv \Delta.
	\end{equation}

	By Theorem \ref{epsize} it suffices to prove that in the  case $\Delta<0$, $\sigma$  converges to 1 as $\lambda\to\lambda_c$. Equation \eqref{f''1delta} and the assumption $\Delta<0$ imply that for $\lambda$ close to $\lambda_c$, in a (non-shrinking) neighborhood of 1, $f''(w)$ has to be bounded from above by some negative constant. Since $f'(1)$ converges to 0 as $\lambda\to \lambda_c$ we conclude
	that for any fixed $w<1$ and all $\lambda$ sufficiently close to $\lambda_c$ one can find $\hat{w} \in (w,1)$ so that $f(\hat{w})<0$. Using the definition of $\sigma$ we see  $\sigma>\hat{w}>w$.
	Letting $w\to 1$, we see that $\sigma$ has to converge to 0 as $\lambda\to\lambda_c$ and thus $\nu$ converges to 0. Hence we have a continuous phase transition. 
	\clearp

	\section{Lower bound on evoSI} \label{sec:lbSI}

	\subsection{AB-avoSI} \label{sec:ABavoSI}
	
	Roughly speaking, the avoSI process serves as an upper bound because certain $I-I$ pairs can rewire, which may leads to additional infections. To get a lower bound, we need to find a way to ensure rewired $I-I$ edges will not transmit infections. This motivates the AB-avoSI process defined as follows.  For each half-edge $h$ we give it two indices:
	
	\begin{itemize}
		\item
		The infection index  $A(h,t)=0$ if $h$ has not been infected by time $t$. 
		If $i$ first becomes an infected half-edge at time $s$, then we set $A(h,t)=s$ for all
		$t\geq s$.

		\item
		The rewiring index $B(h,t)=0$ if the half-edge $h$ has not  rewired by time $t$.
		If $h$ gets rewired at time $s$, then we update the value of $B(h,s)$ to be $s$, no matter whether $h$ has been rewired before or not. In other words, if we let
		$\tau_m(h)$ be the time when $h$ is rewired for the $m$-th time (possibly $\infty$) with $\tau_0(h)=0$,
		then $B(h,t)=\tau_m(h)$
		for $\tau_{m}(h)\leq t<\tau_{m+1}(h)$. 
	\end{itemize}
	
	\mn
	We define the C-AB-avoSI process as follows. As in Section \ref{sec:avoSI}, C is for coupled.
	
	\begin{itemize}
		\item
		At rate $\lambda$ each infected half-edge $h_1$ pairs with a randomly chosen half-edge. Suppose $h_1$ gets paired with half-edge $h_2$ at time $t$. If $h_2$ is susceptible and $B(h_2,t)<A(h_1,t)$ then the vertex associated with half-edge $h_2$ becomes infected. Otherwise $h_1$ will not pass infection to the vertex associated with $h_2$. The reader will see the reason for this condition in the proof of Lemma \ref{abavosi<evosi}.
		Note that if vertex $y$ associated with $h_2$ changes from state $S$ to $I$ then all half-edges attached to $y$ become infected.
		\item 
		Each infected  half-edge gets removed from the vertex that it is attached to at rate $\rho$ and immediately becomes re-attached to a randomly chosen vertex. 
	\end{itemize}
	Similarly to the relation between C-avoSI and avoSI, one can also define the AB-avoSI such that the C-AB-avoSI has the same law as the AB-avoSI on the configuration model. The construction of the graph $G$ for AB-avoSI follows the same route of avoSI (see the proof of Lemma \ref{c-avosi=avo}). Given the graph $G$,  we view each (full) edge as being composed of two half-edges and assign the two indices $A(\cdot, t)$ and $B(\cdot, t)$ as defined above to every half-edge. The evolution of AB-avoSI is then similar to avoSI, except that each time when infected vertex $x$ tries to infect ssuceptible vertex $y$ through edge $e$, we will compare $A(h_1,t)$ and $B(h_2,t)$,
	where $h_1$ is the half-edge of $e$ with one end at $x$ and $h_2$ is the other half-edge of $e$ with one end at $y$.  If $A(h_1,t)>B(h_2,t)$ then the infection will pass through. Otherwise the infection will not pass through, which means $x$ has made an attempt but $y$ remains uninfected. 
	No matter whether the infection passes through or not, we let this $S-I$ edge be deemed stable and it will not get rewired later on. 
	Also from this point on, $x$ will never pass infection to $y$ through $e$. 
	As a comparison, $S-I$ edges in avoSI are always unstable and are subject to potential rewiring.

	One can show that evoSI stochastically dominates AB-avoSI. 
	\begin{lemma}\label{abavosi<evosi}
		There exists a coupling of avoSI and evoSI such that if a vertex is infected in AB-avoSI then it is also infected in evoSI. 
	\end{lemma}

	\begin{proof}[Proof of Lemma \ref{abavosi<evosi}]
		Following the proof of Lemma \ref{avosi>evosi}, we see it suffices to show $I-I$ rewirings will not create additional infections in AB-avoSI. 
		Assume at some time $t_1$ an edge $e$ between two infected vertices $x$ and $y$ is broken from $y$ and reconnects to $z$. For a picture see Figure \ref{fig:ABrewire}. 
		
		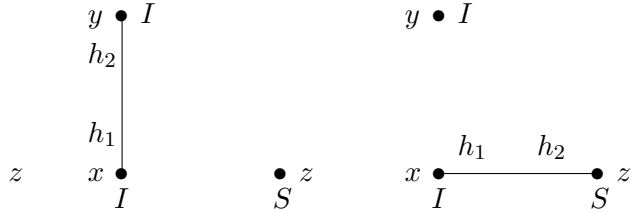
\begin{figure}[h]
			\begin{center}
				\begin{picture}(230,100)
					\put(50,80){$\bullet$}
					\put(50,20){$\bullet$}
					\put(110,20){$\bullet$}
					\put(170,80){$\bullet$}
					\put(170,20){$\bullet$}  	
					\put(230,20){$\bullet$}  	
					
					\put(53,22){\line(0,1){60}}
					\put(172,23){\line(1,0){60}}
					
					\put(40,20){$x$}
					\put(40,80){$y$}
					\put(10,20){$z$}
					\put(60,80){$I$}
					\put(50,10){$I$}
					\put(110,10){$S$}
					\put(120,20){$z$}
					\put(160,20){$x$}
					\put(160,80){$y$}
					\put(240,20){$z$}
					\put(180,80){$I$}
					\put(170,10){$I$}
					\put(230,10){$S$}
					\put(40,65){$h_2$}
					\put(40,35){$h_1$}
					\put(180,30){$h_1$}
					\put(210,30){$h_2$}
				\end{picture}
			\end{center}
			\caption{In the transition from the first drawing to the second, $x$ rewires its connection from $y$ to $z$. Note that $h_1$ and $h_2$ are the two half-edges comprising the edge $e$, which is not drawn in the figure. Vertex $x$ cannot pass an infection to $z$ throught $h_1$ and $h_2$ since $B(h_2,t)>A(h_1,t)$.}
			\label{fig:ABrewire}
		\end{figure}

		We would like to show that after time $t_1$ no infections can pass through $e$ to create additional infected vertices. Denote the half-edge attached to $x$ by $h_1$ and the other half-edge by $h_2$. Since $x$ and $y$ must be infected before time $t_1$, we see that,  according to the definition of the infection index, for all $t\geq t_1$,
		\begin{equation}\label{ah1h2}
			A(h_1,t)<t_1, A(h_2,t)<t_1. 
		\end{equation}
		We next consider the rewiring index. Since the half-edge $h_2$ is rewired at time $t_1$, using the definition of rewiring index we see that for all $t\geq t_1$,
		\begin{equation}\label{bh2}  	B(h_2, t)\geq t_1. 
		\end{equation}
		Combining \eqref{ah1h2} and \eqref{bh2} we see that
		$$
		B(h_2,t)>A(h_1, t).
		$$
		This implies that infections cannot pass from the vertex associated to $h_1$ to the vertex associated to $h_2$. Now we consider the other direction, i.e., from $h_2$ to $h_1$. Note that $x$ has already been infected.  There are two possible cases:
		\begin{itemize}
			\item   If $h_1$ stays with $x$ forever then the vertex associated with $h_1$ is always infected after time $t_1$ (and hence there is no additional infected vertex). 
			\item  If $h_1$ is rewired to some other vertex $x'$ at time $t_2>t_1$ then for all $t>t_2$,
			$$
			B(h_1,t)\geq t_2>t_1>A(h_2,t),
			$$
			which implies that after time $t_2$ infections cannot go from the vertex associated with $h_2$ to the vertex associated with $h_1$. 
		\end{itemize}
		In both cases there will be no additional infected vertices stemming from the rewiring of $e$. Thus we have completed the proof of Lemma \ref{abavosi<evosi}.
	\end{proof}
	\noindent
	We end this section by showing AB-avoSI dominates delSI. This is used in the proof Lemma \ref{avocrv}. 
	\begin{lemma}\label{abdel}
		There exists a coupling of AB-evoSI and delSI such that
		if a vertex is infected in AB-evoSI then it is also infected in delSI. 
	\end{lemma}
	\begin{proof}
		This can be proved in a similar way to the proof of Lemma \ref{couple}. We construct 
		AB-avoSI using the variables $\{T_{e,\ell},R_{e,\ell},V'_{e,\ell},\ell\geq 1\}$ as we did in the proof of Lemma \ref{avosi>evosi} except that for AB-avoSI certain infections may not pass through depending on the relative size of infection index and rewiring index. We also construct delSI using $T_{e,1}$ and $R_{e,1}$  as we did in the proof of Lemma \ref{couple}. We can then prove Lemma \ref{abdel} by repeating the induction argument used in the proof of Lemma \ref{couple}. Note that if edge $e$ is initially present between $x$ and $y$ and 
		$T_{e,1}<R_{e,1}$ then either end of $e$ can be infected by the other end because the rewiring indices for both half-edges of $e$ are equal to 0 at the time of infection. 
		Therefore we don't need to worry about infections not being transmitted successfully in AB-evoSI. 
	\end{proof}
	\clearp

	\subsection{Moment bounds}\label{sec:moment}

	Let $\check{X}_t$ and $\check{X}_{I,t}$ be the number of total half-edges and infected half-edges in the AB-avoSI process.
	As in the analysis of avoSI,  we multiply the original transition rates by $(\check{X_t}-1)/(\lambda \check{X}_{I,t})$. We will use a hat to denote the quantities after the time change. The evolution equation for $\hat{X}_t$ has the same form  as avoSI and hence the first equation of \eqref{3lim} also holds for AB-avoSI. Now we consider the evolution of the number of susceptible half-edges $\hat{X}_{S,t}$. We need a bit more notation to describe this. For half-edge $i$, we let $I(i,t)=1$ if $i$ is an infected half-edge at time $t$ (which also means it hasn't been paired) and $I(i,t)=0$ otherwise. We can define $S(i,t)$ similarly. We also let $S(i,k,t)=1$ if $i$ is attached to a susceptible vertex with $k$ half-edges at time $t$. 
	Finally, we let $v(j,t)$ be the vertex that half-edge $j$ is attached to at time $t$ and
	$D(j,t)$ be the number of half-edges attached to  $v(j,t)$ at time $t$.
	
	We first write down the equation for $\hat{S}_{t,k}$. To reduce the size of 
	formulas we let
	\beq
	G_{i,j}=\{I(i,t)=1,  A(i,t)\leq B(j,t) \} \label{defGij}.
	\eeq
	Reasoning as in the derivation of \eqref{eqStk2} gives
	\begin{equation}\label{newStk}
		\begin{split}
			d\hat{S}_{t,k}&=-k\hat{S}_{t,k}\, dt
			+1_{\{k\geq 1\}} \frac{\rho}{\lambda}\frac{\hat{S}_{t,k-1}}{n}(\hat{X}_t-1) \, dt
			- \frac{\rho}{\lambda}\frac{\hat{S}_{t,k}}{n}(\hat{X}_t-1)\, dt\\
			&+\frac{1}{\hat{X}_{I,t}}\left( \sum_{i,j=1}^{\X_0} 
			1_{G_{i,j}} 1_{\{ S(j,k+1,t)=1 \}}  \right)dt+d\hat{M}_{t,k},
		\end{split}
	\end{equation} 
	where $\hat{M}_{t,k}$ is a martingale.
	Summing \eqref{newStk} over $k$ from 0 to $\infty$ and noting that the second and third term cancel, we get
	\begin{equation}\label{newSt}
		d\hat{S}_t=-\hat{X}_{S,t}\, dt+\frac{1}{\hat{X}_{I,t}}
		\left( \sum_{i,j=1}^{\hat X_0} 1_{G_{i,j}}  1_{\{ S(j,t)=1\} } \right)dt+dM_{1,t}.
	\end{equation}
	Multiplying both sides of \eqref{newStk} by $k$ and summing over $k$, we get
	\begin{equation}\label{newXst}
		\begin{split}
			d\hat{X}_{S,t}&=-\sum_{k=0}^{\infty} k^2\hat{S}_{t,k}dt
			+\frac{\rho}{\lambda}\frac{\hat{S}_t}{n}(\hat{X}_t-1)dt\\
			&+\frac{1}{\hat{X}_{I,t}}
			\left( \sum_{i,j=1}^{\hat X_0} 1_{G_{i,j}} (D(j,t)-1) 1_{\{ S(j,t)=1\}} \right)dt +dM_{2,t}
		\end{split}
	\end{equation}
	for some martingale term $M_{2,t}$.
	
	Analogously,  if we multiply both sides of \eqref{newStk} by $k^2$, $k^3$ and $k^4$, respectively, then we get
	\begin{equation}\label{newk^2Stk}
		\begin{split}
			&d\left(\sum_{k=0}^{\infty} k^2\hat{S}_{t,k}\right)=-\sum_{k=0}^{\infty} k^3\hat{S}_{t,k}dt+2\frac{\rho}{\lambda}\frac{\hat{X}_{S,t}}{n}(\hat{X}_t-1)+\frac{\rho}{\lambda} \frac{\hat{S}_t}{n}(\hat{X}_t-1) \\
			&+\frac{1}{\hat{X}_{I,t}}\left( \sum_{i,j=1}^{\hat X_0} 
			1_{G_{i,j}} (D(j,t)-1)^2 1_{\{ S(j,t)=1\}} \right)dt+dM_{3,t},
		\end{split}
	\end{equation}
	\begin{equation}\label{newk^3stk}
		\begin{split}
			&d\left(\sum_{k=0}^{\infty} k^3\hat{S}_{t,k}\right)=-\sum_{k=0}^{\infty} k^4\hat{S}_{t,k} dt
			+3\frac{\rho}{\lambda}\frac{\sum_{k=0}^{\infty} k^2\hat{S}_{t,k}}{n}(\hat{X}_t-1)
			+3\frac{\rho}{\lambda}\frac{\hat{X}_{S,t} }{n}(\hat{X}_t-1)\\&+ 
			\frac{\rho}{\lambda}\frac{\hat{S}_t}{n}(\hat{X}_t-1)
			+\frac{1}{\hat{X}_{I,t}}\left( \sum_{i,j=1}^{\X_0} 
			1_{G_{i,j}} (D(j,t)-1)^3 1_{\{ S(j,t)=1\}} \right)dt+dM_{4,t},
		\end{split}
	\end{equation}
	and
	\begin{equation}\label{newk^4stk}
		\begin{split}
			d\left(\sum_{k=0}^{\infty} k^4\hat{S}_{t,k}\right)&=-\sum_{k=0}^{\infty} k^5\hat{S}_{t,k} dt
			+4\frac{\rho}{\lambda}\frac{\sum_{k=0}^{\infty} k^3\hat{S}_{t,k}}{n}(\hat{X}_t-1)
			+6\frac{\rho}{\lambda}\frac{\sum_{k=0}^{\infty} k^2\hat{S}_{t,k}}{n}(\hat{X}_t-1)\\
			&+4\frac{\rho}{\lambda}\frac{\hat{X}_{S,t} }{n}(\hat{X}_t-1)+ 
			\frac{\rho}{\lambda}\frac{\hat{S}_t}{n}(\hat{X}_t-1)\\
			&		+\frac{1}{\hat{X}_{I,t}}\left( \sum_{i,j=1}^{\X_0} 
			1_{G_{i,j}} (D(j,t)-1)^4 1_{\{ S(j,t)=1\}} \right)dt+dM_{5,t}.
		\end{split}
	\end{equation}

	The main result of this section is the following lemma. 
	
	\begin{lemma}\label{lbstep1}
		There exists two constants $C_{\ref{s12}},C_{\ref{s13}}>0$, such that for every $\ep>0$ whp we have 
		\begin{equation}\label{s11}
			\sum_{j=1}^5 \left(\sup_{0\leq t\leq \gamma_n\wedge 1} 
			\abs{M_{j,t}}\right)\le n^{5/6},
		\end{equation}
		\begin{equation}\label{s12}
			\sup_{0\leq t\leq  \gamma_n\wedge 1} \sum_{k=0}^{\infty} (k+1)^4\hat{S}_{t,k} \le C_{\ref{s12}}n,
		\end{equation}
		\begin{equation}\label{s13}
			\sum_{j=1}^{\hat X_0}D(j,t)^2 1_{\{S(j,t)=1, B(j,t)>0\}}\leq n(C_{\ref{s13}}t+\ep),
			\quad\hbox{for all $0\leq t\leq \gamma_n\wedge 1$}
		\end{equation}
	\end{lemma}
	\begin{proof}[Proof of Lemma \ref{lbstep1}]
		We first prove equation \eqref{s11}. The proof is based on analyzing the quadratic variation of $M_{1,t}, \ldots, M_{5,t}$. Since the proofs for the five quantities are similar we only give the details for $M_{5,t}$, the martingale associated with $\sum_{k=0}^{\infty} k^4\hat{S}_{t,k}$.
		Let $Q_x$ be the number of half-edges that vertex $x$ originally has plus the half-edges that has been rewired to vertex $x$ before  $x$ becomes infected. Let $D_x(t)$ be the number of half-edges that $x$ has at time $t$. We necessarily have $D_x(t) \leq Q_x$ as long as $x$ is susceptible at time $t$.
		Note that the jumps of $\sum_{k=0}^{\infty} k^4\hat{S}_{t,k}$ have the following three sources.
		
		\begin{itemize}
			\item An infected half-edge pairs with a susceptible half-edge attached to a vertex $x$ with $D_x(t)$ half-edges and passes the infection to $x$. This decreases $\sum_{k=0}^{\infty} k^4\hat{S}_{t,k}$ by  $D_x(t)^4$. Such type of jumps can occur at most once for each susceptible vertex.
			\item An infected half-edge pairs with a susceptible half-edge attached to a vertex $x$ with $D_x(t)$ half-edges but does not pass the infection.  This decreases $\sum_{k=0}^{\infty} k^4\hat{S}_{t,k}$ by  
			$$
			D_x(t)^4-(D_x(t)-1)^4\leq  15 Q_x^3.
			$$
			Such type of jumps can occur at most  $Q_x$ times for vertex $x$.  
			\item An infected half-edge is rewired to a susceptible vertex $x$  with $D_x(t)$ half-edges. This gives an increase of 
			$$
			(D_x(t)+1)^4-D_x(t)^4 \leq 15Q_x^3
			$$ 
			to $\sum_{k=0}^{\infty} k^4\hat{S}_{t,k}$. Such type of jumps can happen at most $Q_x$ times for vertex $x$.
		\end{itemize}
		It follows from the above analysis that the quadratic variation of $M_{5,t}$
		is bounded by
		\beq
		\sum_{x=1}^n (Q_x^4)^2+\sum_{x=1}^n 15(Q_x^3)^2Q_x
		+\sum_{x=1}^n 15(Q_x^3)^2Q_x\leq C_{\ref{qvbd}}\sum_{x=1}^n Q_x^8.
		\label{qvbd}
		\eeq
		
		Using the Burkholder-Davis-Gundy inequality (see, e.g., \cite[Theorem 7.34]{Kl} with $p=5/4$),
		\beq
		\E\left[\sup_{0\leq t\leq  \gamma_n\wedge 1}\abs{M_{5,t}}^{5/4}   \right]
		\leq C_{\ref{BDG}} \E\left[\left( \sum_{x=1}^n Q_x^8 \right)^{5/8} \right]. 
		\label{BDG}
		\eeq
		To bound the right-hand side we need the following well-known fact:
		for any $p\ge 1$ and positive numbers $a_1,\ldots, a_m$,
		\begin{equation}\label{anineq}
			\left(\sum_{i=1}^m a_i^p \right)^{\sqz 1/p}  \leq \sum_{i=1}^m a_i.
		\end{equation}
		In words the $L^p$ norm  under the counting measure  decreases as $p$ increases.
		This is easily seen to be true by noting that
		\begin{equation*}
			\sum_{i=1}^m \frac{a_i}{	\left(\sum_{j=1}^m a_j^p \right)^{\sqz 1/p}  } =\sum_{i=1}^m \left( \frac{a_i^p}{\sum_{j=1}^m a_j^p} \right)^{1/p} \geq \sum_{i=1}^m  \frac{a_i^p}{\sum_{j=1}^m a_j^p}=1,
		\end{equation*}
		since 
		$$
		\frac{a_i^p}{\sum_{j=1}^m a_j^p} \in [0,1] \mbox{ and } 1/p \leq 1.
		$$
		Applying \eqref{anineq} with $p=8/5$ and $a_i=Q_i^8$ gives that
		\begin{equation}\label{q58}
			\left(\sum_{x=1}^n Q_x^8 \right)^{5/8}\leq \sum_{x=1}^n Q_x^5. 
		\end{equation}
		
		As we argued in the proof of \eqref{EQ_l^2}, if $V=\mbox{Binomial}(\hat X_0,1/n)$, then  $Q_x$ is dominated by $D_x+V$. To bound the fifth moment of the sum we note that if $Y$ and $Z$ are nonnegative random variables,
		\begin{equation}\label{y+z^5}
			\E(Y+Z)^5 \le \E( 2 \max\{Y,Z\} )^5 \le 32[ \E Y^5 + \E Z^5].
		\end{equation}
		We claim that
		the 5-th moment of Binomial($m,p$) is bounded by $$mp+\cdots +(mp)^5.$$ To see this, let $Y_1,\ldots, Y_m$ be i.i.d. Bernoulli variable with mean $p$. 
		\begin{equation}
			\begin{split}
				\E\left[\left(\sum_{i=1}^m Y_i\right)^5\right]=\sum_{i_1,\ldots, i_5}\E(Y_{i_1}\cdots Y_{i_5})=\sum_{i_1,\ldots, i_5} p^{\# \textrm{ of distinct elements among }i_1,\ldots, i_5}.
			\end{split}
		\end{equation}
		Note that the number of ordered tuples $(i_1,\ldots, i_5)$ such that 
		there are $\ell$ distinct elements among them is bounded by $m^{\ell}$. We conclude that
		\begin{equation}
			\E\left[\left(\sum_{i=1}^m Y_i\right)^5\right]\leq \sum_{\ell=1}^5 m^{\ell}p^{\ell},
		\end{equation}
		which verifies the claim. 
		Using this claim we have that 
		\beq
		\E(V^5)\leq \sum_{\ell=1}^{5}\frac{\E(\hat X_0^{\ell})}{n^{\ell}}
		\leq C_{\ref{binmom}},
		\label{binmom}
		\eeq
		so that by \eqref{y+z^5},
		\begin{equation}
			\sum_{x=1}^n \E(Q_x^5) \leq 32\sum_{i=1}^n (\E(D_i^5)+\E(V^5)) 
			\leq C_{\ref{m4tctl}}n.
			\label{m4tctl}
		\end{equation}
		
		Equations \eqref{BDG}, \eqref{q58} and \eqref{m4tctl} imply that
		\begin{equation}
			\P\left(\sup_{0\leq t\leq \gamma_n\wedge 1}\abs{M_{5,t}}>n^{5/6}\right)=
			\P\left(\sup_{0\leq t\leq \gamma_n\wedge 1}\abs{M_{5,t}}^{5/4}>
			(n^{5/6})^{5/4} \right)\leq \frac{Cn}{n^{25/24}}\to 0.
		\end{equation}
		Using the reasoning that led to \eqref{qvbd}, we can deduce the same bounds (with different constants) for $M_{1,t},M_{2,t},M_{3,t},M_{4,t}$ and hence equation \eqref{s11} follows by using Markov's inequality. 
		
		To prove equation \eqref{s12}, 
		define the event $\Omega_n$ to be	
		\begin{equation}\label{defomegan}
			\left\{\abs{\sum_{k=0}^{\infty} k^i\hat S_{0,k}-n\sum_{k=0}^{\infty} k^ip_k }\leq n,
			\sup_{0 \leq t\leq \gamma_n\wedge 1} \abs{M_{i,t}}\leq n, 	\quad\hbox{for $i=1,2,3,4$} \right\}.
		\end{equation}
		The assumption that $\E(D^5)<\infty$ and equation \eqref{s11} imply that $\P(\Omega_n)\to 1$ as $n\to\infty$. 
		By the definition of $G_{i,j}$ in \eqref{defGij}, 
		\begin{equation*}
			\frac{1}{\hat{X}_{I,t}}\left( \sum_{i,j=1}^{\hat X_0} 
			1_{G_{i,j}} (D(j,t)-1)^2 1_{\{ S(j,t)=1\}} \right) \leq  	\frac{1}{\hat{X}_{I,t}}\left( \sum_{i,j=1}^{\hat X_0} 
			1_{\{I(i,t)=1\}} (D(j,t)-1)^2 1_{\{ S(j,t)=1\}} \right),
		\end{equation*}
		which is bounded by
		\begin{equation}\label{s14}
			\begin{split}
				\sum_{j=1}^{\hat{X}_0} D(j,t)^21_{\{S(j,t)=1\}}
				&=  \sum_{r=1}^n \sum_{j=1}^{\hat{X}_0}1_{\{v(j,t)=r\}}  1_{\{r \textrm{ is susceptible at }t\}}  D_r(t)^2\\
				&= \sum_{r=1}^n  1_{\{r \textrm{ is susceptible at }t\}}  D_r(t)^3=\sum_{k=0}^{\infty} k^3 \hat{S}_{t,k}.
			\end{split}
		\end{equation}
		Here $D_r(t)$ is the number of half-edges that vertex $r$ has at time $t$. 
		Using \eqref{newk^2Stk} and \eqref{s14} we obtain that
		\begin{equation}\label{s15}
			\sum_{k=0}^{\infty} k^2\hat{S}_{t,k}-\sum_{k=0}^{\infty} k^2\hat{S}_{0,k}\leq  \int_0^t \left(2\frac{\rho}{\lambda}\frac{\hat{X}_{S,u}}{n}(\hat{X}_u-1)+\frac{\rho}{\lambda} \frac{\hat{S}_u}{n}(\hat{X}_u-1)\right)dt+M_{3,t}. 
		\end{equation}
		On the event $\Omega_n$, we have that $\sum_{k=0}^{\infty} k^2 \hat{S}_{0,k}\leq (m_2+1)n$, $\hat{X}_{S,u}\leq \hat{X}_0\leq (m_1+1)n$ and $M_{3,t}\leq n$. Therefore, using \eqref{s15} we see that, there exsits a constant $C_{\ref{s16}}$ such that
		\begin{equation}\label{s16}
			\sum_{k=0}^{\infty} k^2 \hat S_{t,k}\leq (m_2+1)n +\int_0^t \left(\frac{2\rho}{\lambda}(m_1+1)^2n+ \frac{\rho}{\lambda}(m_1+1)n\right)du+n \leq C_{\ref{s16}}n
		\end{equation}
		for all $0\leq t\leq \gamma_n\wedge 1$. 
		Analogously  to the proof of \eqref{s14}, we can show 
		\begin{equation}\label{s100}
			\frac{1}{\hat{X}_{I,t}}\left( \sum_{i,j=1}^{\hat X_0} 
			1_{G_{i,j}} (D(j,t)-1)^3 1_{\{ S(j,t)=1\}} \right)\leq \sum_{k=0}^{\infty} k^4\hat S_{t,k}.
		\end{equation}
		Using \eqref{s100} and \eqref{newk^3stk},
		\begin{align*}
			\sum_{k=0}^{\infty} k^3 \hat{S}_{t,k}  -\sum_{k=0}^{\infty} k^3 \hat{S}_{0,k} \leq \int_0^t   \biggl(&
			3\frac{\rho}{\lambda}\frac{\sum_{k=0}^{\infty} k^2\hat{S}_{u,k}}{n}(\hat{X}_u-1)\\
			&+3\frac{\rho}{\lambda}\frac{\hat{X}_{S,u} }{n}(\hat{X}_u-1)
			+	\frac{\rho}{\lambda}\frac{\hat{S}_u}{n}(\hat{X}_u-1) \biggr)du+M_{4,t}.
		\end{align*}
		On the event $\Omega_n$, using \eqref{s16}, we see that, for all $0\leq t\leq \gamma_n \wedge 1$,
		\begin{equation}\label{s17}
			\begin{split}
				\sum_{k=0}^{\infty} k^3 \hat{S}_{t,k}\leq &(m_1+1)n+n+\\
				&\int_0^t \left(\frac{3\rho}{\lambda}C_{\ref{s16}}(m_1+1)n+\frac{3\rho}{\lambda}(m_1+1)^2n+\frac{\rho}{\lambda}(m_1+1)n \right)\, du \leq C_{\ref{s17}}n,
			\end{split}
		\end{equation}
		for some constant $C_{\ref{s17}}>0$. Proceeding in a similar fashion and using \eqref{newk^4stk} we can show that on $\Omega_n$ there exists a constant $C_{\ref{s18}}>0$, such that whp for all $0\leq t\leq \gamma_n \wedge 1$,
		\begin{equation}\label{s18}
			\sum_{k=0}^{\infty} k^4\hat{S}_{t,k}\leq C_{\ref{s18}}n.
		\end{equation}
		Equation \eqref{s12} follows from \eqref{s18} and the fact that $\P(\Omega_n)\to 1$  since
		$$
		\sum_{k=0}^{\infty} (k+1)^4 \hat{S}_{t,k}\leq 16\sum_{k=0}^{\infty} (k^4+1)\hat{S}_{t,k}\leq 16\sum_{k=0}^{\infty} k^4\hat{S}_{t,k}+16 \sum_{k=0}^{\infty} \S_{t,k}\leq 16\sum_{k=0}^{\infty} k^4 \hat{S}_{t,k}+16n.
		$$
		
		We now turn to the proof of equation  \eqref{s13}. 
		Set $$H(t)=\sum_{j=1}^{\hat X_0}(D(j,t)-1)^2 1_{\{S(j,t)=1, B(j,t)>0\}}.$$
		Note that $H(0)=0$.  Using Dynkin's formula,
		$$
		H(t)=\int_0^t h(s)\, ds+M_{6,t}, 
		$$
		where $M_{6,t}$ is a martingale associated with 
		$H(t)$ and $h(t)$ is the rate of change of $H(t)$. We now control $h(t)$ and $M_{6,t}$ by analyzing the jumps of $H(t)$ ($H(t)$ is a pure jump process). Note that there are three types of jumps:
		\begin{itemize}
			\item An infected half-edge pairs with a half-edge attached to susceptible vertex $x$ and makes $x$ infected. This does not increase $H(t)$ and thus makes a non-positive contribution to $h(t)$. The absolute value of the jump size of $H(t)$ is bounded by $D_x(t)^3$ where $D_x(t)$ is the number of half-edges that $x$ has at time $t$. For each vertex $x$ such jumps can happen at most once.
			\item An infected half-edge pairs with a half-edge attached to susceptible vertex $x$ but $x$ stays susceptible after the pairing. This does not increase $H(t)$ and thus makes a non-positive contribution to $h(t)$. The absolute value of the jump size of $H(t)$ is bounded above by $$
			D_x(t)^2+\abs{(D_x(t)-1)^2-(D_x(x)-2)^2}D_x(t)\leq 3(D_x(t)+1)^2. 
			$$
			To see this, note that the loss of a half-edge $j$ attached to $x$ makes a twofold contribution to $H(t)$. First, $j$ is no longer a half-edge so $H(t)$ has to decrease by $(D(j,t)-1)^2=(D_x(t)-1)^2$. Second, the for each of the remaining $D_x(t)-1$ half-edges attached to $x$, its contribution to $H(t)$ changes from $(D_x(t)-1)^2$ to $(D_x(t)-2)^2$. 
			
			For each vertex $x$ such type of jumps can occur at most $Q_x$ times. 
			\item An infected half-edge is rewired to a susceptible vertex $x$. This increases $H(t)$ by at most 
			$$
			D_x^2(t)+		\abs{(D_x(t)^2-(D_x(t)-1)^2)}D_x(t)\leq 3(D_x(t)+1)^2\leq 3(Q_x+1)^2. 
			$$
			The rate that $x$ receives a rewired half-edge is equal to 
			$$
			\frac{\hat{X}_0}{\lambda \hat{X}_{I,t}} \rho \hat{X}_{I,t} \frac{1}{n}=\frac{\rho \hat{X}_0}{\lambda n}\leq \frac{(m_1+1)\rho}{\lambda}
			$$
			on the event $\Omega_n$ (defined in \eqref{defomegan}). Here the factor of $1/n$ comes from Poisson thinning since  each half-edge is a rewired to a uniformly chosen vertex independently. 
			For each vertex $x$ such type of  jumps can occur at most $Q_x$ times. 
		\end{itemize}
		Therefore  on $\Omega_n$ we have 
		$$
		h(t)\leq \frac{(m_1+1)\rho}{\lambda}
		\sum_{x=1}^n 3(D_x(t)+1)^21_{\{x \textrm{ is susceptible at }t\}}\leq \frac{(m_1+1)\rho}{\lambda}\sum_{k=0}^{\infty} 3(k+1)^4 \hat{S}_{t,k}.
		$$
		Equation \eqref{s12} shows that with high probability $\sum_{k=0}^{\infty} (k+1)^4 \hat{S}_{t,k} \leq C_{\ref{s12}}n$. Since $\Omega_n$ also holds with high probability, we deduce that
		\begin{equation}\label{s131}
			\lim_{n\to\infty}\P\left(\int_0^t h(s)ds\leq C_{\ref{s131}}nt, \forall 0\leq t\leq \gamma_n\right)=1.
		\end{equation}
		The above analysis of the jumps of $H(t)$ also implies that the quadratic variation of $M_{6,t}$ is bounded by 
		$$
		(Q_x^3)^2+ (3(Q_x+1)^2)^2Q_x+(3(Q_x+1)^2)^2Q_x\leq 20 (Q_x+1)^6.
		$$
		We now bound the $2/3$-th moment of the quadratic variation.
		Applying \eqref{anineq} with $p=3/2$ and $a_i=(Q_i+1)^6$ gives that
		$$
		\left(\sum_{i=1}^n (Q_i+1)^6 \right)^{2/3}\leq \sum_{i=1}^n (Q_i+1)^4. 
		$$	Using this and the Burkholder-Davis-Gundy inequality (\cite[Theorem 7.34]{Kl} with $p=4/3$),
		$$
		\E\left[\sup_{0\leq t\leq  \gamma_n\wedge 1}\abs{M_{6,t}}^{4/3}   \right]
		\leq C' \E\left[\left( \sum_{i=1}^n (Q_i+1)^6 \right)^{2/3} \right] \leq C'\E\left(\sum_{i=1}^n (Q_i+1)^4\right)\leq  Cn.
		$$
		This implies that 
		\begin{equation}\label{s132}
			\P\left(\sup_{0\leq t\leq \gamma_n\wedge 1}\abs{M_{6,t}}>n^{4/5}\right)=
			\P\left(\sup_{0\leq t\leq \gamma_n\wedge 1}\abs{M_{6,t}}^{4/3}>n^{4/5*4/3}\right)\leq \frac{Cn}{n^{16/15}}\to 0.
		\end{equation}
		Equation \eqref{s13} follows from equations \eqref{s131} and \eqref{s132}. 
	\end{proof}

	\subsection{Rough upper and lower  bounds for $\hat{X}_{I,t}$ and $\hat{I}_t$}\label{sec:rough}

	Define 
	\begin{equation}\label{et}
		E(t)=\frac{1}{\hat{X}_{I,t}}\left( \sum_{i,j=1}^{\hat{X}_0} 1_{G_{i,j}} D(j,t) 1_{\{ S(j,t)=1\}} \right)
	\end{equation}
	where $G_{i,j}=\{ I(i,t)=1, A(i,t) \le B(j,t) \}$ was defined in \eqref{defGij}. 
	\begin{lemma}\label{lbstep2}
		There exists two constants $\lambda_0,t_0\in (0,1)$, such that for all $\lambda<\lambda_c+\lambda_0$, $t<t_0$ and  $\ep>0$, the following four inequalities hold whp for $0\leq t\leq \gamma_n\wedge t_0$:
		\begin{equation}\label{lbstep11}
			\hat{X}_{I,t}\leq \left(\frac{2\rho m_1}{\lambda_c^2}(\lambda-\lambda_c)t+m_1\Delta t^2+\ep\right)n, 
		\end{equation}
		
		\beq\label{lbstep12}
		\hat{X}_{I,t} \geq \left(\frac{\rho m_1}{2\lambda_c^2}(\lambda-\lambda_c)t+\frac{m_1\Delta}{4}t^2-\ep\right)n 
		-\int_0^t E(u)\, du, 
		\eeq
		
		\begin{equation}\label{lbstep13}
			\hat{I}_t\leq (2m_1t+\ep)n,
		\end{equation}
		
		\begin{equation}\label{lbstep14}
			\hat{I}_t\geq \left(\frac{m_1t}{2}-\ep\right)n-\int_0^t E(u)\, du.
		\end{equation}
	\end{lemma}
	
	\begin{proof}[Proof of Lemma \ref{lbstep2}]
		We first prove equation \eqref{lbstep11}.
		Using Lemma \ref{lbstep1} we see that there exist a constant $C$ such that whp the absolute value of the right hand sides of  \eqref{newXst}--\eqref{newk^3stk} are all upper bounded by $Cn$ for $t\leq \gamma_n \wedge 1$. Denote this event by $G_1(n)$. We have $\P(G_1(n))\to 1$ as $n\to\infty$. Denote 
		the event 
		$$ 
		\left\{\abs{\sum_{k=0}^{\infty} k^i\hat S_{0,k}-n\sum_{k=0}^{\infty} k^ip_k }\leq\ep n,
		\quad\hbox{for $i=1,2,3$}\right\}
		$$ 
		by $G_2(n)$.
		We have $\P(G_2(n))\to 1$ as $n\to\infty$ since $\E(D^3)<\infty$. 
		Therefore we get 
		$$
		\lim_{n\to\infty}\P(G_1(n)\cap G_2(n))=1,
		$$
		which implies that whp for all $0\leq t\leq \gamma_n\wedge 1$, we have
		\beq\label{level1}
		\abs{\sum_{k=0}^{\infty} k^3\S_{t,k}-n\sum_{k=0}^{\infty} k^3 p_k }
		+\abs{\hat{X}_{S,t}-n\sum_{k=0}^{\infty} kp_k} \leq 4n(Ct+\ep).
		\eeq
		The definition of $E(t)$ in \eqref{et} implies that
		\begin{equation*}
			E(t)\leq \sum_{j=1}^{\hat X_0}D(j,t)1_{\{S(j,t)=1,B(j,t)>0\}}.
		\end{equation*}
		Therefore by \eqref{s13} we see that whp for all $ t\leq \gamma_n\wedge 1$,
		\begin{equation}\label{etctl1}
			E(t)\leq n(Ct+\ep). 
		\end{equation}
		Now it follows from the integral form of \eqref{newSt}, \eqref{level1} and \eqref{etctl1} that whp
		\begin{equation}\label{level2st}
			\abs{\hat{S}_t-n\left(1-m_1t\right)}\leq n(Ct^2+\ep).
		\end{equation}
		Similarly to the proof of \eqref{level2st},  whp for all $ 0\leq t\leq \gamma_n\wedge 1$
		\begin{equation}\label{level2k^2}
			\abs{\sum_{k=0}^{\infty} k^2\S_{t,k}-n\left(\sum_{k=0}^{\infty} k^2d_k+\left(-\sum_{k=0}^{\infty} k^3p_k+2\frac{\rho}{\lambda}m_1^2+\frac{\rho}{\lambda}m_1\right)t \right)} \leq n(Ct^2+\ep).
		\end{equation}
		The evolution equation for $\hat X_{t}$ in AB-avoSI has the same form as avoSI, i.e.,
		\begin{equation}\label{xtnew}
			\hat{X}_t-\hat X_0=-\int_0^t 2(\hat{X}_u-1)du+M_{0,t}.
		\end{equation}
		Hence, as in the case of avoSI, equation (3.4) in \cite{JLW} holds for AB-avoSI, which implies 
		\begin{equation}\label{xtnew2}
			\sup_{0\leq t\leq \gamma_n\wedge 1}\abs{\frac{\hat{X}_t}{n}-m_1\exp(-2t)}\CP 0. 
		\end{equation}
		Combining \eqref{newXst} and \eqref{xtnew} and using $\hat{X}_{I,t}=\hat{X}_t-\hat{X}_{S,t}$, we get
		\begin{equation}\label{xitnew}
			\begin{split}
				\hat{X}_{I,t}&=\hat{X}_{I,0}+\int_0^t\left(-2(\hat{X}_u-1)+\sum_{k=0}^{\infty} k^2\hat S_{u,k}-\frac{\rho}{\lambda}\frac{\hat S_u}{n}(\hat{X}_u-1)  \right) du
				+(M_{0,t}-M_{2,t}) \\
				&-\frac{1}{\hat{X}_{I,t}}\int_0^t \left( \sum_{i,j=1}^{\X_0} 1_{G_{i,j}} (D(j,u)-1) 1_{\{ S(j,u)=1\}} \right)du. 
			\end{split}
		\end{equation}
		Dropping the term in the second line of \eqref{xitnew} and bounding $\sup_{0\leq t\leq \gamma_n\wedge 1}\abs{\hat X_{I,0}+M_{0,t}-M_{2,t}}$
		by $\ep n$ (which holds with high probability by \eqref{s11}),
		\begin{equation}
			\hat{X}_{I,t}\leq \ep n+\int_0^t\left(-2(\hat{X}_u-1)+\sum_{k=0}^{\infty} k^2\hat S_{u,k}-\frac{\rho}{\lambda}\frac{\hat S_u}{n}(\hat{X}_u-1)  \right) du.
		\end{equation}
		Using \eqref{xtnew2}, \eqref{level2st} and \eqref{level2k^2} to approximate $\hat{X}_t$, $\hat{S}_t$ and $\sum_{k=0}^{\infty} k^2\hat{S}_{t,k}$ up to the first order, respectively, we obtain that whp,
		\begin{equation}\label{xt-xst}
			\begin{split}
				\hat{X}_t-\hat{X}_{S,t}\leq & n\int_0^t\left(-2m_1\exp(-2u)+ \sum_{k=0}^{\infty} k^2p_k+\left(-\sum_{k=0}^{\infty} k^3p_k+2\frac{\rho}{\lambda}m_1^2+\frac{\rho}{\lambda}m_1\right)u 
				\right. \\
				&\left. -\frac{\rho}{\lambda}(1-m_1u)m_1\exp(-2u)  
				+Cu^2\right)du+2n\ep. 
			\end{split}
		\end{equation}
		We would like to expand the integrand of \eqref{xt-xst} in powers of $u$. Since 
		$\rho/\lambda_c = (m_2-2m_1)/m_1$, the constant term is
		\beq
		-2m_1+m_2-\frac{\rho}{\lambda}m_1 
		= m_1 \left( \frac{\rho}{\lambda_c} -\frac{\rho}{\lambda} \right) 
		= \frac{m_1\rho}{\lambda_c^2}(\lambda-\lambda_c)+O((\lambda-\lambda_c)^2).
		\label{cterm}
		\eeq
		for $\lambda>\lambda_c$. Therefore for $\lambda$ sufficiently close to $\lambda_c$ we have 
		\beq
		-2m_1+m_2-\frac{\rho}{\lambda}m_1\leq \frac{2m_1\rho}{\lambda_c^2}(\lambda-\lambda_c).
		\label{ctermbd}
		\eeq
		Note that
		$e^{-2u} = 1 - 2u +2u^2 + \ldots$, so the coefficient in front of $u$ is  
		\beq
		4m_1 + \left(-m_3+\frac{\rho}{\lambda}2m_1^2+\frac{\rho}{\lambda}m_1\right)
		+\frac{\rho}{\lambda}m_1^2 + \frac{\rho}{\lambda} 2m_1
		= 4m_1 - m_3 + \frac{\rho}{\lambda} (3m_1^2 + 3m_1).
		\label{uterm}
		\eeq
		At $\lambda=\lambda_c$, this coefficient is equal to
		\begin{equation}\label{m1delta}
			4m_1-m_{3}+(3+3m_1)(m_2-2m_1)=-m_3+3m_2-2m_1+3m_2m_1-6m_1^2.
		\end{equation}
		We claim that the quantity in \eqref{m1delta} is exactly equal to $m_1\Delta.$
		Indeed, from the definition of $\Delta$ we see that
		$$
		\Delta=-\frac{\mu_3}{\mu_1}+3(\mu_2-\mu_1)=
		-	\frac{m_3-3m_2+2m_1}{m_1}+3(m_2-2m_1)
		$$
		so that 
		$$
		m_1\Delta= -m_3+3m_2-2m_1+3m_2m_1-6m_1^2.
		$$
		Using \eqref{uterm} with the equations that follow, we see that for $\lambda$ close to $\lambda_c$ we have
		\beq
		4m_1-m_3+2\frac{\rho}{\lambda}m_1^2+\frac{\rho}{\lambda}m_1+\frac{\rho}{\lambda}m_1(m_1+2)\leq \frac{3m_1\Delta}{2}. 
		\label{utermbd}
		\eeq
		Using \eqref{ctermbd} and \eqref{utermbd} in \eqref{xt-xst}, we see that for some constant 
		$C'>0$ and 
		all $t\leq m_1\Delta/(2 C')$,
		\begin{equation}\label{xt-xst2}
			\begin{split}
				\hat{X}_{I,t}&\leq n\left(\int_0^t\left( \frac{2\rho}{\lambda_c^2}(\lambda-\lambda_c)+\frac{3m_1\Delta }{2}u+C'u^2\right)du+2\ep\right)\\
				&\leq \left(\frac{2\rho}{\lambda_c^2}(\lambda-\lambda_c)t+\frac{3m_1\Delta}{4} t^2+ \frac{C't^3}{3} +2\ep\right)n\\
				&\leq \left(\frac{2\rho}{\lambda_c^2}(\lambda-\lambda_c)t+m_1\Delta t^2+2\ep\right)n.
			\end{split}
		\end{equation}
		This proves \eqref{lbstep11} since $\ep$ is arbitrary. 
		
		\medskip
		{\bf The proof of \eqref{lbstep12}} is parallel to the proof of \eqref{lbstep11}, except that we now replace the second line of \eqref{xitnew} by
		$E(u)$ (defined in \eqref{et}). We can do this because
		$$
		\frac{1}{\hat{X}_{I,t}}\left( \sum_{i,j=1}^{\hat X_0} 1_{G_{i,j}}
		(D(j,t)-1) 1_{\{ S(j,t)=1\}} \right)\leq E(t),
		$$
		which is true by the definition of $E(t)$ in \eqref{et}.
		
		\medskip
		Since $\hat{I}_t=n-\hat{S}_t$, equation \eqref{newSt} implies 
		\begin{equation}\label{newIt}
			d\hat{I}_t=\hat{X}_{S,t}\, dt-\frac{1}{\hat{X}_{I,t}}
			\left( \sum_{i,j=1}^{\hat X_0} 1_{G_{i,j}}  1_{\{ S(j,t)=1\} } \right)dt-dM_{1,t}.
		\end{equation}
		Using $\hat{I}_0=1$ and the inequality
		$$
		\frac{1}{\hat{X}_{I,t}}\left( \sum_{i,j=1}^{\hat X_0} 1_{G_{i,j}}  
		1_{\{ S(j,t)=1\}} \right) \leq E(t),
		$$
		which follows from the fact that $D(j,t)\geq 1$, we see that
		$$
		\abs{\hat{I}_t-\left(1+\int_0^t\hat{X}_{S,u}du\right)}\leq \int_0^t E(u)du+\abs{M_{1,t}}. 
		$$
		{\bf  The rest of proofs for \eqref{lbstep13} and \eqref{lbstep14}} are parallel to \eqref{lbstep11} and \eqref{lbstep12}.
		We omit further details. 
	\end{proof}
	Let
	\beq
	L(t)=\sum_{j=1}^{\hat X_0}D(j,t)1_{\{S(j,t)=1, B(j,t)>0\}}.
	\label{defLt}
	\eeq
	Since $ \sum_{i=1}^{\hat X_0} 1_{\{I(i,t)=1\}} = \hat X_{I,t}$ we have
	\begin{equation}\label{et<lt}
		\begin{split}
			E(t)&=\frac{1}{\hat{X}_{I,t}}\left( \sum_{i,j=1}^{\hat X_0} 1_{\{I(i,t)=1\}} D(j,t) 1_{\{ S(j,t)=1\}} 1_{\{A(i,t)\leq B(j,t) \} } \right)\\
			&\leq \frac{1}{\hat{X}_{I,t}}\left( \sum_{i,j=1}^{\hat X_0} 1_{\{I(i,t)=1\}} D(j,t) 1_{\{ S(j,t)=1\}} 1_{\{ B(j,t)>0 \} } \right) =L(t).
		\end{split}
	\end{equation} 
	The definition of $L(t)$ and the fact $D(j,t)\leq D(j,t)^2$ imply that
	$$
	L(t)\leq \sum_{j=1}^{\hat X_0}D(j,t)^2 1_{\{S(j,t)=1,B(j,t)>0\}}.
	$$
	Using the fact $D(j,t)\leq D(j,t)^2$ again, \eqref{s13} implies that
	\begin{equation}\label{L1}
		\lim_{n \to \infty}\P(L(t)\leq n(C_{\ref{s13}}t+\ep), \forall 0\leq t\leq \gamma_n\wedge 1)=1.
	\end{equation}
	Combining \eqref{et<lt} and \eqref{L1} we see that for some constant $C_{\ref{eubound1}}>0$,
	\begin{equation}\label{eubound1}
		\lim_{n\to\infty}\P\left(\int_0^t E(u)du\leq (
		C_{\ref{eubound1}}
		t^2+\ep)n, \forall 0\leq t\leq \gamma_n \wedge t_0\right)=1. 	
	\end{equation}
	The bound provided in \eqref{eubound1} is not enough for our purpose (though we will also use it in the proof of Theorem  \ref{Q1AB}). We will prove refined bounded in the next section. 
	
	\clearp
	
	\subsection{More refined bounds}\label{sec:refined}

	Equation \eqref{lbstep12} implies that we can get a lower bound for $\hat{X}_{I,t}$ if we can upper bound the term $\int_0^t E(u)du$. To this end, we 	let $b$ be some number in $(0,1)$ to be determined. We can decompose $E(t)$ into two parts:
	\begin{equation}\label{ede}
		\begin{split}
			E_1(t) &:=\frac{1}{\hat{X}_{I,t}}
			\left( \sum_{i,j=1}^{\hat X_0} 1_{\{I(i,t)=1\}} D(j,t) 1_{\{ S(j,t)=1\}} 
			1_{\{A(i,t)<bt \} } \right),\\
			E_2(t)	&=\frac{1}{\hat{X}_{I,t}}\left( \sum_{i,j=1}^{\hat X_0} 1_{\{I(i,t)=1\}} D(j,t) 1_{\{ S(j,t)=1\}} 1_{\{B(j,t)>bt\}} \right).
		\end{split}
	\end{equation}
	Since $G_{i,j} = \{ I(i,t)=1, A(i,t) \le B(j,t)\}$, we have that  $E(t) \le E_1(t) + E_2(t)$.
	Now we set \begin{equation}\label{xll}
		\begin{split}
			X(I,b,t)&:=\sum_{i=1}^{\X_0}1_{\{A(i,t)\leq bt, I(i,t)=1\}} \le \hat X_{I,t}, \\
			L(b,t)&:=\sum_{j=1}^{\X_0}D(j,t)1_{\{S(j,t)=1, B(j,t)>bt\}} \le L(t).
		\end{split}
	\end{equation}
	Recalling the definition of $L(t)$ in \eqref{defLt}, we see that
	\begin{equation}\label{e1e2}
		E_1(t)=\frac{X(I,b,t)}{\hat{X}_{I,t}}L(t), \qquad E_2(t)=L(b,t).
	\end{equation}
	In the next two lemmas we give bounds on $L(b,t)$ and $X(I,b,t)$. 
	\begin{lemma}\label{ctlet2}
		There exists a constant $C_{\ref{lbtctl}}$ so that
		for any fixed $b\in (0,1)$ and any $\ep>0$, 
		whp for all $0\leq t\leq \gamma_n\wedge 1$,
		\begin{equation}
			L(b,t)\leq n(C_{\ref{lbtctl}}(1-b)^{1/2}t+\ep).
			\label{lbtctl}
		\end{equation}
	\end{lemma}
	\begin{proof} Using the Cauchy-Schwartz inequality,
		\begin{equation}\label{lbtctl3}
			\begin{split}
				L(b,t)=&\sum_{j=1}^{\hat X_0}D(j,t)1_{\{S(j,t)=1, B(j,t)>bt\}} \\
				\leq &
				\left(\sum_{j=1}^{\hat X_0} D(j,t)^21_{\{S(j,t)=1,B(j,t)>0\}} \right)^{1/2}
				\left(\sum_{j=1}^{\hat X_0} 1_{\{bt\leq B(j,t)\leq t\}} \right)^{1/2}.
			\end{split}
		\end{equation}
		The first term in the second line of \eqref{lbtctl3} has already been controlled by equation \eqref{s13}, i.e., 
		\begin{equation}\label{Nt-1}
			\lim_{n \to \infty}	\P\left(\sum_{j=1}^{\hat X_0}D(j,t)^2 
			1_{\{S(j,t)=1, B(j,t)>0\}} \leq n(C_{\ref{s13}}t+\ep), \forall 0\leq t\leq \gamma_n\wedge 1 \right)=1. 
		\end{equation}
		Let $N(t)$ be the number of rewiring events that occur by time $t$. Then we have
		\begin{equation}\label{Nt0}
			\sum_{j=1}^{\hat X_0} 1_{\{bt\leq B(j,t)\leq t\}} 
			\leq N(t)-N(bt).
		\end{equation}
		Now we write down the  evolution equation for  $N(t)$ 
		$$
		N(t)=\int_0^t q(u)\, du+M_{7,t},
		$$
		where 
		\begin{equation}\label{qt}
			q(t)=\rho \hat{X}_{I,t}\frac{\hat{X}_t-1}{\lambda \hat{X}_{I,t}}
			\leq\frac{\rho \hat X_0}{\lambda} 
		\end{equation}
		and $M_{7,t}$ is some martingale. The assumption $\E(D^5)<\infty$ implies that
		the event $\Omega^*_n=\{\hat X_0\leq 2m_1n\}$ has probability tending to 1 as $n\to\infty$. On $\Omega^*_n$, using \eqref{qt} we have
		$$
		q(t)\leq 2m_1\rho n/\lambda.
		$$
		It follows that
		\begin{equation}\label{Nt1}
			\lim_{n\to\infty}\P(q(t)\leq 2m_1\rho n/\lambda, \forall t\geq 0)=1. 
		\end{equation}
		Note that $N(t)$ is a pure jump process with jump size equal to 1. 
		It follows that the expected value of quadratic variation of $M_{7,t}$  up to time 1 is upper bounded by 
		$$
		\E\left(\sup_{t>0} q(t)\right)\leq \frac{\rho}{\lambda}\E(\hat{X}_0)\leq \frac{m_1\rho}{\lambda}.
		$$
		This implies that for any $\ep>0$,
		\begin{equation}\label{Nt2}
			\lim_{n\to\infty}	\P\left(\sup_{0 \leq t\leq \gamma_n \wedge 1} \abs{M_{7,t}}\leq \ep n\right)=1.  
		\end{equation}
		From the definition of $N(t)$, we see that for $t\leq 1$,
		$$
		\abs{N(t)-N(bt)}=\abs{\int_{bt}^t q(u)du\, 
			+M_{7,t}-M_{7,bt}}\leq \abs{\int_{bt}^t q(u)du}+2\sup_{0\leq t\leq \gamma_n\wedge 1}
		\abs{M_{7,t}}.
		$$
		Thus by \eqref{Nt1} and \eqref{Nt2},
		\begin{equation}\label{Nt3}
			\lim_{n\to\infty}\P(\abs{N(t)-N(bt)}\leq Cn((1-b)t+\ep),\forall  0\leq t\leq \gamma_n)=1. 
		\end{equation}
		Combining  \eqref{Nt0} and \eqref{Nt3},
		\begin{equation}\label{Nt5}
			\lim_{n\to\infty}	
			\P\left( \sum_{j=1}^{\hat X_0} 1_{\{bt\leq B(j,t)\leq t\}} \leq 
			Cn((1-b)t+\ep) 	\right)=1.
		\end{equation}
		Equation \eqref{lbtctl} now follows from \eqref{lbtctl3}, \eqref{Nt-1} and \eqref{Nt5}.
	\end{proof}
	Let $t_0$ and $\lambda_0$ be the two constants given in the statement of Lemma \ref{lbstep2}. 
	Based on the calculations that led to \eqref{fRHS} we let
	\begin{equation}\label{Utdef}
		\begin{split}
			U(t) = & C_{\ref{fRHS}} \biggl[ \exp\left(-\frac{C_{ \ref{Ddef} }(1-b)}
			{\lambda-\lambda_c+t} \right) (\lambda-\lambda_c) t + (1-b)t\\
			& + \exp\left(-\frac{C_{ \ref{Ddef} }(1-b)}
			{\lambda-\lambda_c+t} \right) t^2 + (\lambda-\lambda_c)t^2+t^3 \biggr]
			+C_{\ref{RHS2}}\sqrt{\ep}.
		\end{split}
	\end{equation}
	\begin{lemma} \label{ctlet3}
		For any $\lambda<\lambda_c+\lambda_0$,
		whp  for all $0\leq t\leq t_0\wedge \gamma_n$ we have 
		\beq
		X(I,b,t) \leq U(t)n.
		\label{xibt1}
		\eeq
	\end{lemma}
	
	\begin{proof}
		We define the events $H_k(i,t),1\leq k\leq 4$ for $1\leq i\leq \hat X_0$ ($i$ is any half-edge) as follows.
		\begin{align*}
			H_1(i,t)&=\{I(i,bt)=1, i \mbox{ didn't get rewired or paired in }[bt,t]\},\\
			H_2(i,t)&=\{I(i,bt)=1,  i \mbox{ got rewired to an infected vertex at its first rewiring in }[bt,t] \},\\
			H_3(i,t)&=\{I(i,bt)=1,  i \mbox{ got rewired to a susceptible vertex at its first rewiring in }[bt,t]\\
			& \mbox{ and that vertx later became infected in }[bt,t]\},\\
			H_4(i,t)&=\{0<A(i,t)\leq bt,S(i,bt)=1, v(i,bt) \mbox{ got infected in }[bt,t]\}.
		\end{align*}

		We claim that 
		$$
		\{A(i,t)\leq bt, I(i,t)=1\} \subset \cup_{k=1}^4 H_k(i,t). 
		$$
		Indeed, either $I(i,bt)=1$ or $S(i,bt)=1$ must hold. The case of $S(i,bt)=1$ corresponds to $H_4(i,t)$. On the other hand, if $I(i,bt)=1$ and $I(i,t)=1$, then there are three possible cases: $i$ didn't get rewired in $[bt,t]$, $i$ was rewired to an infected vertex or $i$ was rewired to a susceptible vertex which later became infected. The first case case corresponds to $H_1(i,t)$ while the second and third case are covered in $H_2(i,t)$ and $H_3(i,t)$, respectively.  
		
		Let $H_k(t)$ be the number of half-edges $i$ for which $H_k(i,t)$ occurs.
		It follows from the claim and the definition of $X(I,b,t)$ that
		\beq\label{xibt}
		X(I,b,t)\leq \sum_{k=1}^4 H_k(t).
		\eeq
		
		\mn
		{\bf We first estimate $H_1(t)$.} The rate for  a half-edge to be paired at time $u$ is $(\hat{X}_u-1)/(\lambda \hat{X}_{I,u})$. Hence conditionally on $\hat X_{I,s}$, $bt\leq s\leq t$, 
		$H_1(t)$ is stochastically dominated by a 
		$$
		\textrm{Binomial}\left(\hat{X}_{I,bt},\  \exp\left(-\int_{bt}^t \frac{\hat{X}_u}{\lambda \hat{X}_{I,u}}du\right)\right).
		$$
		random variable. 
		Here we have a binomial distribution because the Poisson clocks on different infected half-edges are independent of each other. 
		Let $C_{\ref{defl1n}}=\max\{2\rho m_1/\lambda_c^2,m_1\Delta,1\}$. Then the event
		\begin{align}
			L_1(n) =& \{\hat{X}_{I,t}\leq C_{\ref{defl1n}}n((\lambda-\lambda_c)t+t^2+\ep),
			\nonumber\\
			&\hat{X}_t\geq nm_1\exp(-2t)/2, \forall 0\leq t\leq \gamma_n\wedge t_0\}
			\label{defl1n}
		\end{align}
		has  probability tending to 1 as $n\to\infty$ by \eqref{lbstep11} of Lemma \ref{lbstep2} and \eqref{xtnew2}.
		On the event $L_1(n)$, we have that
		\beq
		\int_{bt}^t \frac{\hat{X}_u}{\hat{X}_{I,u}}du\geq
		\int_{bt}^t \frac{m_1\exp(-2u)/2}{C_{\ref{defl1n}}((\lambda-\lambda_c)u+u^2+\ep)}du
		\geq \frac{C_{\ref{Xintbd}}e^{-2t} (1-b)t}{(\lambda-\lambda_c)t+t^2+\ep}.
		\label{Xintbd}
		\eeq
		Thus on $L_1(n)$, $H_1(t)$ is stochastically dominated by
		\begin{equation}
			W_0(t):=	\textrm{Binomial}
			\left( C_{\ref{defl1n}}n((\lambda-\lambda_c)t+t^2+\ep), 
			\exp\left(	-\frac{C_{\ref{Xintbd}}e^{-2t} (1-b)t}{(\lambda-\lambda_c)t+t^2+\ep)} \right)\right).
			\label{ber0}
		\end{equation}
		For  $\sqrt{\ep}\leq t\leq t_0$, we have
		$$
		\frac{(1-b)t}{(\lambda-\lambda_c)t+t^2+\ep}
		\geq \frac{(1-b)t}{2((\lambda-\lambda_c)t+t^2)}
		=\frac{1-b}{2(\lambda-\lambda_c+t)}.
		$$
		Hence $W_0(t)$ is stochastically dominated by 
		\begin{equation}
			W_1(t):=\textrm{Binomial}\left( C_{\ref{defl1n}}n((\lambda-\lambda_c)t+t^2+\ep),
			\exp\left(-\frac{C_{\ref{Xintbd}}e^{-2t} (1-b)}{2(\lambda-\lambda_c+t)}  \right)\right)
			\label{ber1}
		\end{equation}
		for $\sqrt{\ep}\leq t\leq t_0$. Define
		\begin{equation}\label{defU0t}
			U_0(t)=C_{\ref{defl1n}}((\lambda-\lambda_c)t+t^2+\ep)\exp\left(
			-\frac{C_{\ref{Xintbd}} e^{-2t}(1-b)}{2(\lambda-\lambda_c+t)}  \right).
		\end{equation}
		For $t>\sqrt{\ep}$, there exists a constant $C_{\ref{U0tlb}}=C_{\ref{U0tlb}}(\lambda_0,t_0,\ep)$ depending on $\lambda_0$, $t_0$ and $\ep$ such that 
		\beq
		U_0(t)\geq C_{\ref{U0tlb}}.
		\label{U0tlb}
		\eeq
		
		We need a large deviations bound for sums of Bernoulli random variables. 
		
		\begin{lemma}\label{chernoff}
			Consider $n$ i.i.d. Bernoulli random variables $Y_1,\ldots, Y_n$. Let $\mu=\sum_{k=1}^n\E(Y_i)$. Then we have
			\begin{equation}\label{cher1}
				\P\left(\sum_{k=1}^n Y_k\geq 3\mu
				\right)\leq \exp(-\mu).
			\end{equation}
		\end{lemma}
		
		\begin{proof}[Proof of Lemma \ref{chernoff}]
			By \cite[Theorem 2.3.1]{V}, we have
			\begin{equation*}
				\P\left(\sum_{k=1}^n Y_k\geq 3\mu\right)\leq \exp(-\mu)\left(\frac{e\mu}{3\mu}\right)^{3\mu} \leq \exp(-\mu). 
			\end{equation*}
		\end{proof}
		
		Using  Lemma \ref{chernoff}, we have, for $t\geq \sqrt{\ep}$,
		\begin{equation}
			\P\left(W_1(t)>3U_0(t)n \right)\leq \exp(-U_0(t)n)
			\leq \exp(-C_{\ref{U0tlb}}n).
			\label{ber2}
		\end{equation}
		Using the definition of $L_1(n)$ in \eqref{defl1n} and the fact that for $\sqrt{\ep}\leq t\leq t_0$, $W _1(t)$ dominates $W_0(t)$ which in turn dominates $H_1(t)$ (see \eqref{ber0}), we get
		\begin{equation}
			\P\left(\{H_1(t)>3U_0(t)n\}\cap L_1(n) \right)
			\leq \exp(-C_{\ref{U0tlb}}(\lambda_0,t_0,\ep)n).
		\end{equation} 
		For $0\leq t\leq \sqrt{\ep}$, on the event $L_1(n)$,
		\begin{equation}
			\hat{X}_{I,t}\leq C_{\ref{defl1n}}n((\lambda-\lambda_c)t+t^2+\ep)
			\leq C_{\ref{defl1n}}n((\lambda-\lambda_c)\sqrt{\ep}+\ep+\ep)
			\leq C_{\ref{ber5}}n\sqrt{\ep}. 
			\label{ber5}
		\end{equation}
		for sufficiently small $\ep$. Since $H_1(t)\leq \hat{X}_{I,bt}$ (by the definition of $H_1(t)$),
		$$
		H_1(t)\leq C_{\ref{ber5}}n\sqrt{\ep} \quad\hbox{ for $t\leq \sqrt{\ep}$.}
		$$
		Thus if we define  
		\beq
		U_1(t)=U_0(t)+C_{\ref{ber5}}\sqrt{\ep},
		\label{defU1t}
		\eeq
		then we have 
		\begin{equation}\label{ber3}
			\P\left(\{H_1(t)>3 U_1(t)n\}\cap L_1(n) \right)
			\leq \exp(-C_{\ref{U0tlb}}n)
		\end{equation}
		for all $0\leq t\leq t_0$. 
		Setting $t^{\ell}_{0}=t_0\ell /n^{3/2},0\leq \ell\leq n^{3/2}$,
		we get 
		\begin{equation}\label{xibt2}
			\P\left(\left( \cup_{\ell=0}^{n^{3/2}-1} 
			\{H_1(t_0^{\ell})>3U_1(t_0^{\ell})n\}\right)\cap L_1(n) 
			\right) \leq n^{3/2}\exp(-C_{\ref{U0tlb}}n).
		\end{equation}
		Denote  the oscillation of $H_1(t)$ in $[t_0^{\ell},t_0^{\ell+1}]$ by
		$\omega(H_1(t),t_0^{\ell},t_0^{\ell+1})
		$. Here, the oscillation of any function (deterministic or random) $g(t)$ in an interval $[a,b]$ is defined to be 
		$$
		\sup_{a\leq t_1\leq t_2\leq b}\abs{g(t_1)-g(t_2)}.
		$$
		Consider the event that there is at most six pairings  (i.e., a half-edge pairs with another half-edge) and rewirings (i.e., a half-edge is rewired to another vertex) occurring in $[bt_0^{\ell},bt_0^{\ell+1}]\cup [t_0^{\ell},t_0^{\ell+1}]$ and denote it by $\Omega^{\ell}$. On $\Omega^{\ell}$ we have 
		\begin{equation}\label{w6qi}
			\omega(H_1(t),t_0^{\ell},t_0^{\ell+1})\leq 6\max_{1\leq i\leq n}Q_i,
		\end{equation}
		where  $Q_i$ is the number of half-edges that vertex $i$ has before it becomes infected. By \eqref{w6qi}, Markov's inequality and \eqref{m4tctl} (together with $Q_i^5\geq Q_i^4$) we have
		\begin{equation}\label{h11}
			\P(\{\omega(H_1(t),t_0^{\ell},t_0^{\ell+1})\geq \ep n\}\cap \Omega^{\ell})\leq 
			\P\left(6\max_{1\leq i\leq n}Q_i \geq \ep n\right) \leq 
			\frac{6^4\E(\sum_{i=1}^n Q_i^4)}{\ep^4n^4}\leq \frac{C}{\ep^4 n^3}.
		\end{equation}
		Now we control the probability of $(\Omega^{\ell})^c$. Note that the rate for a rewiring or pairing to occur is 
		equal to $(\rho \hat{X}_{I,t}+\lambda\hat{X}_{I,t})(\hat{X}_t-1)/(\lambda\hat{X}_{I,t})$ which is bounded by $(\lambda+\rho)\hat X_0/\lambda$. On the event $\Omega^*_n:=\{\hat X_0\leq 2m_1n\}$ (which holds with high probability) this quantity is upper bounded by $2(\rho+\lambda)m_1n/\lambda$. Using this we get
		\begin{equation}\label{h12}
			\P((\Omega^{\ell})^c)\cap \Omega_n^*) \leq C(n\cdot n^{-3/2})^6\leq Cn^{-3}. 
		\end{equation}
		Combining \eqref{h11} and \eqref{h12} we get
		\begin{equation}
			\P(\{\omega(H_1(t),t_0^{\ell},t_0^{\ell+1})>\ep n\}\cap \Omega_n^*)\leq \frac{C}{\ep^4 n^3}.
		\end{equation}
		By the union bound for probabilities,
		\begin{equation}\label{xibt3}
			\P\left(\left(\cup_{\ell=0}^{n^{3/2}-1} \{w(H_1(t),t_0^{\ell},t_0^{\ell+1})>\ep n\}\right) \cap \Omega_n^*\right)\leq \frac{C}{\ep^4 n^{3/2}}.
		\end{equation}
		Combining \eqref{xibt2},\eqref{xibt3} and the facts  
		$\P(L_1(n))\to 1,\P(\Omega^*_n)\to 1$, we get
		\begin{equation}
			\lim_{n\to\infty}\P\left( 
			\left(\cap_{\ell=0}^{n^{3/2}-1} 
			\{w(H_1(t),t_0^{\ell},t_0^{\ell+1})\leq \ep n\}\right)
			\cap \left( \cap_{\ell=1}^{n^{3/2}-1} 
			\{H_1(t_0^{\ell})\leq 3 U_1(t_0^{\ell})n\}\right)
			\right)=1.
		\end{equation}
		On the event 
		$$
		\left(\cap_{\ell=0}^{n^{3/2}-1} \{w(H_1(t),t_0^{\ell},t_0^{\ell+1})\leq \ep n\}\right)
		\cap \left( \cap_{\ell=1}^{n^{3/2}} 
		\{H_1(t_0^{\ell})\leq 3 U_1(t_0^{\ell})n\}\right),
		$$
		we necessarily have
		$H_1(t)\leq 3U_1(t)n +\ep n\leq 4U_1(t)n$
		for all $0\leq t\leq t_0\wedge \gamma_n$.  Hence we get
		\begin{equation}\label{ctlh1}
			\lim_{n\to\infty}	\P(H_1(t)\leq 4U_1(t)n,\forall 0\leq t\leq t_0\wedge \gamma_n)=1. 
		\end{equation}
		
		\mn
		{\bf Now we turn to the control of $H_2(t)$.} 
		We set the event 
		$$
		L_2(n):=\{\hat{I}_t\leq (2m_1t+\ep)n,\forall 0\leq t\leq \gamma_n\wedge t_0\}.
		$$
		By equation \eqref{lbstep13} we have
		\begin{equation}\label{ItbdH2t}
			\lim_{n\to\infty} \P(L_2(n))=1.
		\end{equation}
		On  $L_2(n)$, $H_2(t)$ is stochastically dominated by
		\begin{equation}\label{w_2}
			W_2(t):=\textrm{Binomial}(C_{\ref{w_2}}n((\lambda-\lambda_c)t+t^2+\ep), C_{\ref{w_2}}'(t+\ep)).
		\end{equation}
		for some constants $C_{\ref{w_2}}$ and $C'_{\ref{w_2}}$. 
		Now we define
		\beq
		U_2(t)=C_{\ref{w_2}} ((\lambda-\lambda_c)t+t^2+\ep)
		C'_{\ref{w_2}}(t+\ep)+C_{\ref{ber5}}\sqrt{\ep}.
		\label{defU2t}
		\eeq
		Following the proof of \eqref{ctlh1}, one can derive  analogous inequalities to \eqref{ber3} and \eqref{xibt3} for $H_2(t)$. Combining these two inequalities we obtain that
		\begin{equation}\label{ctlh2}
			\lim_{n\to\infty}	\P(H_2(t)\leq 4U_2(t),\forall 0\leq t\leq t_0\wedge \gamma_n)=1. 
		\end{equation}
		We omit further details. 
		
		\mn
		{\bf It remains to control $H_3(t)$ and $H_4(t)$.}  For any vertex $x$,
		let $R(x)$ be the indicator function of the event that vertex $x$ has received at least one rewired edge when $x$ first becomes infected and let $Q_x$ be the number of half-edges that $x$ has just before it becomes infected. Let $R(x,t)$ be the indicator of the event that $x$ has received at least one rewired half-edge by time $t$.
		Then we have, by the definitions of $H_3(t)$ and $H_4(t)$,
		\begin{equation}\label{h3h4}
			H_3(t)+H_4(t)\leq \sum_{x=1}^n R(x,t) Q_x 1_{\{x \textrm{ was infected in }[bt,t]\}}.
		\end{equation}
		Denote the right hand side of \eqref{h3h4} by $N(bt,t)$, then
		we can decompose $N(bt,t)$ into a drift part and a martingale part for any fixed $t$:
		\begin{equation}\label{nbtt}
			N(bt,t)=\int_{bt}^t \bar h_t(u) \, du +M_{8,t}.
		\end{equation}
		Let $D_x(u)$  be the number of half-edges of vertex $x$ at time $u$ and $D(j,u)$
		the number of half-edges that $v(j,u)$ has at time $u$ (recall that $v(j,u)$ is the vertex that half-edge $j$ is attached to at time $u$).  The process  $N(bt,u)$, $bt\leq u\leq t$ has a positive jump whenever a susceptible vertex with at least one rewired half-edge gets infected. The probability that $x$ is infected (given an infection event occurs) is equal to
		$D_x(u)/(\hat{X}_u-1)$ and the contribution to $N(bt,t)$ is equal to $R(x,u)D_x(u)$.
		Thus  $\bar h_t(u)$ satisfies
		\begin{equation}\label{barh}
			\begin{split}
				\bar h_t(u)&\leq \lambda \hat{X}_{I,u} \frac{\hat{X}_u-1}{\lambda \hat{X}_{I,u}} \cdot \frac{\sum_{x=1}^n R(x,u)D^2_x(u)1_{\{x \textrm{ is susceptible at time }u\}}}{\hat{X}_u-1}\\
				&\leq \sum_{j=1}^{\hat X_0}
				D(j,u)^21_{\{S(j,t)=1,B(j,t)>0\}},
			\end{split}
		\end{equation}
		where the second inequality  follows from changing the order of summation:
		\begin{equation}\label{d^2ju}
			\begin{split}
				\sum_{v=1}^n D_x^2(u)1_{\{x \textrm{ is susceptible at time }u\}} R(x,u)
				&\leq \sum_{x=1}^n D_x^2(u)1_{\{x \textrm{ is susceptible at time }u\}} \sum_{j=1}^{\hat X_0} 1_{\{v(j,u)=x,B(j,u)>0\}}\\
				&=	 \sum_{j=1}^{\hat X_0} \sum_{x=1}^n 1_{\{v(j,u)=x\}} 	D(j,u)^21_{\{S(j,u)=1,B(j,u)>0\}}\\
				&	=	 	 \sum_{j=1}^{\hat X_0}
				D(j,u)^21_{\{S(j,u)=1,B(j,u)>0\}}.
			\end{split}
		\end{equation}
		In the first step  of \eqref{d^2ju}  we used the definition of $R(x,u)$ so that 
		$$
		\sum_{j=1}^{\hat X_0} 1_{\{v(j,u)=x,B(j,u)>0\}}\geq R(x,u).
		$$

		We denote the event
		$$L_3(n)=
		\left\{
		\sum_{j=1}^{\X_0}D(j,t)^2 1_{\{S(j,t)=1, B(j,t)>0\}}\leq n(C_{\ref{s13}}t+\ep), \forall 0\leq t\leq \gamma_n\wedge 1
		\right\}. $$
		Then by \eqref{s13} we have $\P(L_3(n))\to 1$ as $n\to\infty$. On $L_3(n)$,
		using \eqref{barh} and setting $C_{\ref{barh2}}=C_{\ref{s13}}$, we see that for all $0\leq t\leq t_0$,
		\begin{equation}\label{barh2}
			\bar h_t(u)\leq n(C_{\ref{barh2}}t+\ep), \forall bt\leq u\leq t.
		\end{equation} 
		The definition of $N(bt,t)$ as the right hand side of \eqref{h3h4} implies that
		we can upper bound the quadratic variation of   $M_{8,t}$  by  $\sum_{i=1}^n Q_i^2$ where $Q_i$ is the number of half-edges that vertex $i$ has before it becomes infected. Using this and the Burkholder-Davis-Gundy inequality we have
		\begin{equation}\label{m7t}
			\E\left(M_{8,t}^4\right)
			\leq C \E\left[\left(\sum_{i=1}^n Q_i^2\right)^2\right]\leq C\E\left(n\sum_{i=1}^n Q_i^4\right)
			\leq C'n^2.
		\end{equation}
		The second inequality in \eqref{m7t} is due to the Cauchy-Schwartz inequality
		$$
		\left(\sum_{i=1}^n (Q_i^2)^2\right)\left(\sum_{i=1}^n 1^2\right)\geq \left(\sum_{i=1}^n Q_i^2\right)^2,
		$$
		and the third inequality follows from \eqref{m4tctl}. 
		Using \eqref{m7t} we have
		\begin{equation}\label{m7t2}
			\P(\abs{M_{8,t}}\geq \ep n)\leq \frac{\E(M_{8,t}^4)}{\ep^4 n^4}\leq \frac{C}{\ep^4 n^2}.  
		\end{equation}
		Now using \eqref{h3h4}, \eqref{nbtt}, \eqref{barh2} and \eqref{m7t2}, we get, for any $0\leq t\leq t_0$,
		\begin{equation}
			\P(L_3(n)\cap \{H_3(t)+H_4(t)\geq (1-b)t(C_{\ref{barh2}}t+\ep)n+\ep n \})\leq \frac{C}{\ep^4 n^2}.
		\end{equation}
		Define 
		\begin{equation}\label{defU3t}
			U_3(t)=(1-b)t(C_{\ref{barh2}}t+\ep)+\ep.
		\end{equation}
		Now we can repeat the proof of \eqref{ctlh1} (i.e., divide $[0,t_0]$ into $n^{3/2}$ intervals and use a union bound) to get
		\begin{equation}\label{ctlh3h4}
			\lim_{n\to\infty}	\P(H_3(t)+H_4(t)\leq 2U_3(t)n,\forall 0\leq t\leq  t_0   \wedge \gamma_n )=1. 
		\end{equation}
		Combining \eqref{ctlh1}, \eqref{ctlh2} and \eqref{ctlh3h4}, we have that with high probability
		$$
		X(I,b,t) \le 5(U_1(t)+U_2(t)+U_3(t))n.
		$$
		Using \eqref{defU0t}, \eqref{defU1t}, \eqref{defU2t}, and \eqref{defU3t}, the right-hand side is
		\begin{align}
			5 \biggl[& C_{\ref{defl1n}}((\lambda-\lambda_c)t+t^2+\ep)
			\exp\left(-\frac{C_{\ref{Xintbd}}e^{-2t}(1-b)}
			{2(\lambda-\lambda_c+t)} \right)
			\nonumber\\
			& + (C_{\ref{w_2}} (\lambda-\lambda_c)t+t^2+\ep)
			C'_{\ref{w_2}}(t+\ep)+2 C_{\ref{ber5}}\sqrt{\ep}
			\label{RHS}\\
			&+ (1-b)t(C_{\ref{barh2}}t+\ep)+\ep \biggr]n.
			\nonumber
		\end{align}
		To make the computation easier to write we note that when $t \le t_0\leq 1$,
		\beq
		\exp\left(-\frac{C_{\ref{Xintbd}}e^{-2t}(1-b)}
		{2(\lambda-\lambda_c+t)}  \right)
		\le \exp\left(-\frac{C_{ \ref{Ddef} }(1-b)}
		{\lambda-\lambda_c+t} \right)
		:= F.
		\label{Ddef}
		\eeq
		Expanding the terms in \ref{RHS} and putting the terms  with $\ep$ or $\sqrt{\ep}$ together, we bound \eqref{RHS} by
		\begin{equation}	\label{RHS2}
			\begin{split}
				5 \biggl[& C_{\ref{defl1n}}((\lambda-\lambda_c)t+t^2) F
				+(C'_{\ref{w_2}} (\lambda-\lambda_c)t+t^2)C'_{\ref{w_2}}t\\
				&+ C_{\ref{barh2}} (1-b)t \biggr]n +C_{\ref{RHS2}}\sqrt{\ep}n.
			\end{split}
		\end{equation}
		We also used the fact that $\ep\leq \sqrt{\ep}$ in \eqref{RHS2}. 
		Sorting the terms by powers of $t$ we get
		\begin{align}
			&5 \biggl[ C_{\ref{defl1n}}F (\lambda-\lambda_c)t + C_{\ref{barh2}} (1-b)t 
			\label{RHS3}\\
			&  +C_{\ref{defl1n}}F t^2
			+C_{\ref{w_2}}C'_{\ref{w_2}} [(\lambda-\lambda_c)t^2+t^3] \biggr]n
			+C_{\ref{RHS2}}\sqrt{\ep}n.
			\nonumber
		\end{align}
		Simplifying constants we have
		\begin{align}
			X(I,b,t) \le & C_{\ref{fRHS}} \biggl[ \exp\left(-\frac{C_{ \ref{Ddef} }(1-b)}
			{\lambda-\lambda_c+t} \right) (\lambda-\lambda_c) t 
			\label{fRHS} \\
			& + \exp\left(-\frac{C_{ \ref{Ddef} }(1-b)}
			{\lambda-\lambda_c+t} \right) t^2 + (\lambda-\lambda_c)t^2+ (1-b)t^2+t^3 \biggr]n
			+C_{\ref{RHS2}}\sqrt{\ep}n,
			\nonumber
		\end{align}
		which completes the proof of Lemma \ref{ctlet3}.
	\end{proof}
	
	\clearp
	
	\subsection{Completing the proof of Theorem \ref{Q1AB}}\label{sec:complete}
	
	\begin{proof}[Proof of Theorem \ref{Q1AB}] We set $\gamma_n=\inf\{t>0:\hat{X}_{I,t}=0\}$. 	We now condition on a large outbreak so that $\hat I_{\infty}>\eta n$ for some fixed 	$\eta>0$. By \eqref{lbstep13}, we see that, conditionally on $\hat I_{\infty}>\eta n$, with high probability $\gamma_n>\ep$ for some $\ep>0$. That is, 
		\begin{equation}\label{epeta}
			\lim_{n\to\infty}\P(\gamma_n>\ep|\hat{I}_{\infty}/n>\eta)=1.
		\end{equation}
		Let $t_0$ and $\lambda_0$ be given by the statement of Lemma \ref{lbstep2}. 
		We let $\lambda_1<\lambda_0,t_1<t_0$ be two constants (independent of $\ep$)  and $\ep_1,\ep_2,\ep_3,\ep_4,\ep_5,\ep_6$ be some small numbers (depending on $\ep$) to be determined. 
		Recall the definition of $U(t)$ in \eqref{Utdef}.  We set $C_{\ref{RHS2}}\sqrt{\ep}$ in \eqref{Utdef} to be $\ep_4$. In other words, 
		\begin{equation}\label{defUt2}
			\begin{split}
				U(t)&=	 	
				C_{\ref{fRHS}} \left(\exp\left(-\frac{C_{ \ref{Ddef} }(1-b)}{\lambda-\lambda_c+t} \right)(\lambda-\lambda_c)t+\right.\\
				&	 	\left.	\left(\exp\left(-\frac{C_{ \ref{Ddef} }(1-b)}{\lambda-\lambda_c+t} \right)+(\lambda-\lambda_c)+(1-b) \right)t^2+t^3\right)+\ep_4,
			\end{split}
		\end{equation}
		
		Previous results imply that the following inequalities hold whp on $0 \le t \le \gamma_n \wedge t_0$. The numbers on the left give the formula numbers for these statements.
		\begin{align*}
			&\eqref{lbstep12} \quad \hat{X}_{I,t}\geq \left(\frac{\rho m_1(\lambda-\lambda_c)}{2\lambda_c^2}t+\frac{m_1\Delta}{4}t^2-\ep_2\right)n-\int_0^t E(u) du. \\
			&\eqref{lbstep14} \quad \hat{I}_t\geq \left(\frac{m_1t}{2}-\ep_5\right)n-\int_0^t E(u)\, du.\\
			&\eqref{eubound1}\quad \int_0^t E(u)du\leq (C_{\ref{eubound1}} t^2+\ep_1)n.\\
			& \eqref{L1} \quad  L(t)\leq n(C_{\ref{s13}}t+\ep_1).\\
			&\eqref{lbtctl} \quad  L(b,t)\leq n( C_{\ref{lbtctl}}(1-b)^{1/2}t+\ep_3).	\\
			& \eqref{xibt1} \quad  X(I,b,t) \leq U(t)n.
		\end{align*}
		
		Let $\Omega_n$ be the event that all of the last six formulas together with the event $\{\gamma_n>\ep\}$ hold.
		Combining \eqref{epeta} and the fact that $\liminf_{n\to\infty}\P(\hat{I}_{\infty}/n>\eta)>0$ (since $\lambda>\lambda_c$), 
		\begin{equation}
			\lim_{n\to\infty}	\P(\Omega_n|\hat{I}_{\infty}/n>\eta)=1.
		\end{equation}
		Now we define
		\begin{equation}\label{lasttau}
			\tau=\inf\left\{\ep \leq t\leq t_0\wedge \gamma_n: \int_0^t E(u)\, du> \left(\frac{\rho m_1(\lambda-\lambda_c) }{4\lambda_c^2}t+
			\frac{m_1\Delta}{8}t^2+\ep_6\right)n
			\right\},
		\end{equation}
		where $\inf\emptyset$ here is set to be $t_0 \wedge \gamma_n$. 
		We want to select the parameters $\lambda_1$ and $t_1$ so that whenever $\lambda-\lambda_c<\lambda_1$ and the outcome is in $\Omega_n$, we have 
		\begin{align}
			t_1 \le & \tau < \gamma_n \label{timebds},\\
			\hat{I}_{t_1}\geq 	\left(\frac{m_1t_1}{2}-\ep_5\right)n & -\int_0^{t_1} E(u)du>\frac{m_1t_1}{8}n.
			\label{final1}
		\end{align}
		This implies $$\frac{\hat{I}_{\infty}}{n}>\frac{m_1t_1}{8},$$
		which proves \eqref{delta>0AB}, as desired. We now divide the proof of
		\eqref{timebds} and \eqref{final1}
		into five steps:
		\begin{itemize}
			\item In Step 1, we choose appropriate $\ep_1$, $\ep_2$ and $\ep_6$ to ensure $\gamma_n>\tau>\ep$.
			\item In Step 2, we show that, under  conditions \eqref{cond2a} and \eqref{cond2b} below, there exists a constant $t_1$ such that $\tau\geq t_1$.
			\item In Step 3, we show that \eqref{cond2a} and \eqref{cond2b} can be satisfies by choosing appropriate values of the parameters involved. The first three steps combined give \eqref{timebds}.
			\item In Step 4, we prove \eqref{final1}. 
			\item Finally, we summarize the choices of  the parameters. 
		\end{itemize}
		
		\mn
		{\bf Step 1.} We first show that $\tau>\ep$. Using \eqref{eubound1},
		\begin{equation}\label{e0t}
			\int_0^{\ep}E(u)\, du\leq (
			C_{\ref{eubound1}} \ep^2+\ep_1)n,
		\end{equation}
		We  set
		\begin{equation}\label{cond1.5}
			\ep_1=	C_{\ref{eubound1}} \ep^2,
		\end{equation}
		and 
		\begin{equation}\label{cond1}
			\ep_2=\ep_6= \frac{\rho m_1(\lambda-\lambda_c) }{16\lambda_c^2}\ep.
		\end{equation}
		We now require
		\begin{equation}\label{cond0}
			C_{\ref{eubound1}} \ep^2+\ep_1=2	C_{\ref{eubound1}} \ep^2 \leq \frac{\ep_6}{2}=\frac{\rho m_1(\lambda-\lambda_c) }{32\lambda_c^2}\ep,
		\end{equation}
		which holds true if $2C_{\ref{eubound1}} \ep^2 \leq 
		\rho m_1(\lambda-\lambda_c)\ep/32\lambda_c^2$ or equivalently,
		\begin{equation}\label{condep}
			\ep<\frac{\rho m_1(\lambda-\lambda_c)}{64C_{\ref{eubound1}}\lambda_c^2}.
		\end{equation}
		We now prove $\tau>\ep$. Indeed, using \eqref{e0t}-\eqref{cond0},  $$
		\int_0^{\ep} E(u)du \leq \frac{\rho m_1 (\lambda-\lambda_c)}{32\lambda_c^2}\ep n<
		\left(\frac{\rho m_1(\lambda-\lambda_c)}{2\lambda_c^2}t+\frac{m_1\Delta}{4}t^2-\ep_2\right)n.
		$$
		This proves $\tau>\ep$ by the definition of $\tau$. 
		
		Now we show that  $\gamma_n>\tau$. If $\gamma_n>t_0$ then this is trivial.
		So we assume that $\gamma_n\leq t_0$. 
		Equation  \eqref{lbstep12} and the definition of $\tau$  in \eqref{lasttau} imply that, for $\ep<t\leq \tau$,
		\begin{equation}
			\begin{split}\label{xit4}
				\hat{X}_{I,t}&\geq \left(\frac{\rho m_1(\lambda-\lambda_c)}{2\lambda_c^2}t+\frac{m_1\Delta}{4}t^2-\ep_2\right)n-\int_0^t E(u)du\\
				&\geq \left(\frac{\rho m_1(\lambda-\lambda_c)}{2\lambda_c^2}t+\frac{m_1\Delta}{4}t^2-\ep_2\right)n-\left(\frac{\rho m_1(\lambda-\lambda_c) }{4\lambda_c^2}t+
				\frac{m_1\Delta}{8}t^2+\ep_6\right)n\\
				&\geq \left(\frac{\rho m_1(\lambda-\lambda_c) }{4\lambda_c^2}t+
				\frac{m_1\Delta}{8}t^2-\ep_2-\ep_6\right)n.
			\end{split}
		\end{equation}
		Note that when $t>\ep$,
		$$
		\frac{\rho m_1(\lambda-\lambda_c) }{4\lambda_c^2}t+
		\frac{m_1\Delta}{8}t^2-\ep_2-\ep_6>\frac{\rho m_1(\lambda-\lambda_c) }{4\lambda_c^2}\ep+
		\frac{m_1\Delta}{8}\ep^2-\ep_2-\ep_6>0,
		$$
		due to \eqref{cond1}. 
		This proves $\gamma_n>\tau$. 
		
		\mn
		{\bf Step 2.} We turn now to the second requirement that $\tau\geq t_1$ for some appropriately chosen constant $t_1$. We will show this by contradiction.  Assume $\tau<t_1$ which is also smaller than $t_0$. Since we have already proved  $\gamma_n>\tau>\ep$ in {\bf Step 1}, it follows that 
		\begin{equation}\label{contra}
			\int_0^{\tau} E(u)du= \left(\frac{\rho m_1(\lambda-\lambda_c) }{4\lambda_c^2}\tau+\frac{m_1\Delta}{8}\tau^2+\ep_6\right)n.
		\end{equation}
		We split the integral $\int_0^{\tau} E(u) \,du$ into two parts: 
		$\int_0^{\ep}E(u)\,du$ and $\int_{\ep}^{\tau}E(u)\, du$. This first part has already been controlled in \eqref{e0t} and \eqref{cond0}. 
		For the second part  $\int_{\ep}^{\tau}E(u)du$, \eqref{e1e2} implies that 
		\begin{equation}\label{ee1e2}
			E(t)\leq E_1(t)+E_2(t)=\frac{X(I,b,t)}{\hat{X}_{I,t}}L(t)+L(b,t).
		\end{equation}
		For $\ep<t<\tau$, using \eqref{xibt} and \eqref{xit4}, we get
		\begin{equation}
			\frac{X(I,b,t)}{\hat{X}_{I,t}}\leq
			\frac{U(t)n}{\left(\frac{\rho m_1(\lambda-\lambda_c) }{4\lambda_c^2}t+
				\frac{m_1\Delta}{8}t^2-\ep_2-\ep_6\right)n}.
			\label{Xratiobd}
		\end{equation}
		Using  \eqref{cond1} and $t >\ep$, we see 
		$$
		\left(\frac{\rho m_1(\lambda-\lambda_c) }{4\lambda_c^2}t+
		\frac{m_1\Delta}{8}t^2-\ep_2-\ep_6\right)n\geq \max\left\{\frac{\rho m_1(\lambda-\lambda_c)}{8\lambda_c^2}tn,\frac{m_1\Delta}{8}t^2n, \ep_2 n  \right\}.
		$$
		Hence using  \eqref{defUt2} and \eqref{Xratiobd} we get
		\begin{equation}\label{xibt5}
			\begin{split}
				\frac{X(I,b,t)}{\hat{X}_{I,t}}	\leq
				& \frac{ C_{\ref{fRHS}}\exp\left(-\frac{C_{ \ref{Ddef} }(1-b)}{\lambda-\lambda_c+t} \right)(\lambda-\lambda_c)t}{\frac{\rho m_1(\lambda-\lambda_c)}{8\lambda_c^2}t}\\
				&+
				\frac{ C_{\ref{fRHS}}\left(\exp\left(-\frac{C_{ \ref{Ddef} }(1-b)}
					{\lambda-\lambda_c+t} \right)
					+(\lambda-\lambda_c)+(1-b) +t\right)t^2}{\frac{m_1\Delta}{8}t^2}+\frac{\ep_4}{\ep_2}\\
				\leq & \exp\left(-\frac{C_{ \ref{Ddef} }(1-b)}{\lambda-\lambda_c+t} \right)\left(\frac{8 C_{\ref{fRHS}}\lambda_c^2 }{\rho m_1}
				+\frac{8 C_{\ref{fRHS}}}{m_1\Delta}\right)\\
				&+\frac{8 C_{\ref{fRHS}}}{m_1\Delta}\left((\lambda-\lambda_c)+(1-b)+t\right)+\frac{\ep_4}{\ep_2}.
			\end{split}
		\end{equation}
		Let $V(\lambda,t)$ denote the quantity on the last two lines of \eqref{xibt5}. $V(\lambda,t)$ 
		is increasing with respect to both $\lambda$ and $t$. Therefore for $\lambda\leq \lambda_c+\lambda_1$ and $t\leq t_0$, which we have supposed is $<1$,  $V(\lambda,t)$ is bounded above by its value at $(\lambda_c+\lambda_1, t_1)$
		\begin{equation}\label{defM}
			V^*:=\exp\left(-\frac{C_{ \ref{Ddef} }(1-b)}{\lambda_1+t_0} \right)\left(\frac{8 C_{\ref{fRHS}}\lambda_c^2 }{\rho m_1}
			+\frac{8 C_{\ref{fRHS}}}{m_1\Delta}\right)
			+\frac{8 C_{\ref{fRHS}}}{m_1\Delta}\left((\lambda-\lambda_c)+(1-b)+t_1\right)+\frac{\ep_4}{\ep_2}.
		\end{equation}
		Using  \eqref{ee1e2}, \eqref{defM}, \eqref{L1}, and \eqref{lbtctl}
		we see that
		\begin{equation}\label{eeptau}
			\begin{split}
				\int_{\ep}^{\tau}E(u)\,du 
				&\leq \int_{\ep}^{\tau}\frac{X(I,b,u)}{\hat{X}_{I,u}}L(u)\, du
				+\int_{\ep}^{\tau}L(b,u)\, du\\
				&\leq  \int_{\ep}^{\tau} V^*(C_{\ref{s13}}u+\ep_1) n\, du
				+\int_{\ep}^{\tau} n(C_{\ref{lbtctl}}(1-b)^{1/2}u+\ep_3)\, du\\
				&\leq \left(V^*C_{\ref{s13}}\frac{\tau^2}{2}
				+V^*\ep_1\tau + C_{\ref{lbtctl}} (1-b)^{1/2}\frac{\tau^2}{2}
				+\ep_3\tau \right)n.
			\end{split}
		\end{equation}
		We want to choose our parameters so that
		\begin{align}\label{cond2a}
			&	\frac{V^*C_{\ref{s13}}+C_{\ref{lbtctl}}(1-b)^{1/2} }{2}\leq \frac{m_1\Delta}{16},\\
			&V^*\ep_1+\ep_3\leq \frac{\ep_6}{2}.
			\label{cond2b}
		\end{align}
		Indeed, if \eqref{cond2a} and \eqref{cond2b} hold, then by \eqref{e0t}, \eqref{cond0} and \eqref{eeptau} we have
		\begin{equation}
			\int_0^{\tau}E(u)du=\int_0^{\ep}E(u) \, du+\int_{\ep}^{\tau}E(u)\, du\leq 
			\left(	\frac{\ep_6}{2}+\frac{m_1\Delta}{16}\tau^2+\frac{\ep_6}{2}\right)n,
		\end{equation}
		which is smaller than
		$$
		\left(\frac{\rho m_1(\lambda-\lambda_c) }{4\lambda_c^2}\tau+
		\frac{m_1\Delta}{8}\tau^2+\ep_6\right)n.
		$$
		This contradicts with \eqref{contra} and proves $\tau\geq t_1$ on $\Omega_n$. 
		
		\mn
		{\bf Step 3.} \eqref{cond2a} will hold if
		$$
		V^*\leq \frac{m_1\Delta }{32C_{\ref{s13}}} \quad\hbox{and} \quad
		(1-b)^{1/2}\leq \frac{m_1\Delta}{32C_{\ref{lbtctl}}}.
		$$
		We have $\ep_1\leq \ep_6/2$ by \eqref{cond0}, so
		\eqref{cond2b} will hold if
		$$
		V^* \le 1/4 \quad\hbox{and} \quad \ep_3=\frac{\ep_6}{4}.
		$$
		To make sure \eqref{cond2a} and \eqref{cond2b} are satisfied, it suffices to have
		\begin{equation}\label{cond3}
			V^*\leq K := \min\left\{\frac{m_1\Delta }{32C_{\ref{s13}}},\frac{1}{4}\right\},
			\quad		(1-b)^{1/2}\leq \frac{m_1\Delta}{32C_{\ref{lbtctl}}}, \quad
			\ep_3=\frac{\ep_6}{4C_{\ref{lbtctl}}}.
		\end{equation}
		Using the definition of $V(\lambda,t)$ in \eqref{defM}, \eqref{cond3} can be satisfied if we first choose $b$ sufficiently close to 1 such that
		\begin{equation}
			\frac{8 C_{\ref{fRHS}}}{m_1\Delta}(1-b)\leq \frac{K}{8},
			\label{cond4a}
		\end{equation}
		then choose $\lambda_1$, $t_1$ and $\ep_4$  such that
		\beq\label{cond4b}
		\exp\left(-\frac{C_{ \ref{Ddef} }(1-b)}{\lambda_1+t_1} \right) \left(\frac{b C_{\ref{fRHS}}\lambda_c^2 }{\rho m_1} 
		+\frac{8 C_{\ref{fRHS}}}{m_1\Delta} \right) 
		+\frac{8 C_{\ref{fRHS}}}{m_1\Delta}\left(\lambda-\lambda_c\right)+\frac{8 C_{\ref{fRHS}}}{m_1\Delta}t_1
		<	 \frac{K}{8},
		\eeq 
		and finally take
		\begin{equation}\label{cond4c}
			\ep_4= \frac{\ep_2}{8}K.
		\end{equation}
		
		\mn
		{\bf Step 4.} It remains to take care of \eqref{final1}. Using \eqref{lbstep14} and \eqref{eubound1}, we have
		\begin{equation}
			\hat{I}_{t_1}\geq 	\left(\frac{m_1t_1}{2}-\ep_5\right)n-\int_0^{t_1} E(u)du
			\geq \left(\frac{m_1t_1}{2}-\ep_5-C_{\ref{eubound1}}t_1^2-\ep_1\right)n.
		\end{equation}
		Recall that we set $\ep_1=C_{\ref{eubound1}}\ep^2$ in \eqref{cond1.5}.
		Thus we have 
		\begin{equation}
			\hat{I}_{t_1}\geq  \left(\frac{m_1t_1}{2}-\ep_5-C_{\ref{eubound1}}t_1^2-C_{\ref{eubound1}}\ep^2\right)n,
		\end{equation}
		which is bigger than $nm_1t_1/8$ if
		\begin{equation}\label{cond6}
			t_1\leq \frac{m_1}{16C_{\ref{eubound1}}}, \quad
			\ep_5= \frac{m_1 t_1}{16}, \quad
			\ep<\left(\frac{m_1t_1}{16C_{\ref{eubound1}}}\right)^{1/2}.
		\end{equation}
		
		\mn
		{\bf Summary} we can first choose $b$ close to 1,  then tale $\lambda_1$ and $t_1$ sufficiently small such that \eqref{cond4b} holds true and
		$$
		t_1\leq \min\left\{\frac{m_1}{16C_{\ref{eubound1}}},t_0\right\}.
		$$ 
		Then we let $\ep$ be smaller than 
		$$
		\min\left\{\frac{\rho m_1 (\lambda-\lambda_c)}{64C_{\ref{eubound1}}\lambda_c^2}, \left(\frac{m_1t_1}{16 C_{\ref{fRHS}}}\right)^{1/2}\right\}.
		$$
		Finally we determine $\ep_1,\ldots, \ep_6$ using \eqref{cond1}, \eqref{cond1.5}, \eqref{cond3}, \eqref{cond4c} and \eqref{cond6}. 
		\begin{align*}
			\ep_2=\ep_6 & = \frac{\rho m_1(\lambda-\lambda_c)}{16\lambda_c^2} \ep,
			\quad \ep_1 = C_{\ref{eubound1}}\ep^2 ,
			\quad \ep_3 = \frac{\ep_6}{4},\\
			\ep_4& = \frac{\ep_2}{8}
			\min\left\{\frac{m_1\Delta }{32C_{\ref{s13}}},\frac{1}{4}\right\}, 
			\quad \ep_5= \frac{m_1 t_0}{16}.
		\end{align*}
	\end{proof}
	
	\newpage 
	


\begin{acks}
	We are grateful to the referees for constructive comments.  Dong Yao is supported by  NSF of Jiangsu Province (No. BK20220677). 
\end{acks}

	
\end{document}